\documentclass[a4paper,11pt]{amsart}
\usepackage[margin=1in]{geometry}
\usepackage{amsmath, amssymb, amsthm, float, tikz, tikz-cd, epstopdf}

\usepackage[unicode,colorlinks,plainpages=false,hyperindex=true,bookmarksnumbered=true,bookmarksopen=false,pdfpagelabels]{hyperref}
\hypersetup{urlcolor=cyan,linkcolor=blue,citecolor=red,colorlinks=true}
\usepackage{color}

\usepackage{datetime}

\numberwithin{equation}{section}
\newtheorem{thm}{Theorem}[section]
\newtheorem{prop}[thm]{Proposition}
\newtheorem{lem}[thm]{Lemma}
\newtheorem{cor}[thm]{Corollary}

\theoremstyle{definition}
\newtheorem{defn}[thm]{Definition}
\newtheorem{exa}[thm]{Example}
\theoremstyle{remark}
\newtheorem{rem}[thm]{Remark}
\newtheorem{ques}[thm]{Question}

\title{Computations of Floer Lasagna Modules for Traces of Negative $L$-Space Knots}
\author{Colin McCulloch}
\address{Mathematical Institute, University of Oxford, Andrew Wiles Building, Radcliffe Observatory Quarter, Woodstock Road, Oxford, OX2 6GG, UK}
\email{colin.mcculloch@maths.ox.ac.uk}


\begin{document}

\begin{abstract}
    We show that for all negative $L$-space knots and framings $n\geq2g(K)$, if a certain cobordism map between cables of $K$ is non-vanishing, then the Floer lasagna module of the knot trace of $K$ with framing $-n$ is infinite dimensional. Further, for the maps on the hat flavour of link Floer homology induced by band maps, quasi-stabilisations and pair of pants cobordisms we give a description of the map on the $E^{\infty}$-page of Ozsvath and Szabo's spectral sequence for forgetting components.
\end{abstract}

\maketitle

\section{Introduction}

A skein lasagna module is an invariant of a pair $(X,L)$, where $X$ is a 4-manifold and $L\subset \partial X$ is a link, that was introduced by Morrison, Walker, and Wedrich in \cite{OGSkLas}. This invariant was recently used by Ren and Willis to give the first analysis-free proof of the existence of exotic compact orientable 4-manifolds in \cite{exoticsklas}. This result has since been reproved without the use of lasagna modules by Nahm in \cite{NoSkeinLasExotica}. The invariant is constructed using Khovanov-Rozansky homology \cite{KhovanovRozOG}, which is a generalisation of Khovanov's categorification of the Jones polynomial \cite{KhovanovOG}. Further, the module is defined by considering the vector space over a field $\mathbb{F}$ generated by what are called Lasagna fillings of the pair $(X,L)$, modulo some equivalence relation. This definition is unwieldy and near impossible to compute except in very simple cases. However, for 2-handlebodies, that is 4-manifolds obtained by attaching 2-handles to $D^4$ along a framed link $L$, it was shown by Manolescu and Neithalath in \cite{SkeinLasagna2Handle} that the skein lasagna module can be computed in terms of the direct sum of the Khovanov-Rozansky homologies of cables of $L$ quotiented by relations coming from cobordism maps. This work was then extended by Manolescu, Walker, and Wedrich in \cite{SkeinLasangaHandleDecomp} to any 4-manifold given a handle decomposition. However, even with these breakthroughs it remains difficult to compute the skein lasagna module as there are, in general, no formulas for the Khovanov-Rozansky homology of cables of a knot, let alone a link.

In another direction Ozsvath and Szabo introduced Heegaard Floer homology in \cite{OSHFOG}, the hat flavour is denoted $\widehat{HF}$, which assigns an isomorphism class of Abelian groups to a closed oriented 3-manifold. They later extended this in \cite{OSLinkFloer} to define an isomorphism class of bigraded $\mathbb{F}_{2}$-vector spaces for pointed null-homologous links in three manifolds called link Floer homology and denoted $\widehat{HFL}$. It was then proven by Juhasz in \cite{AndrasCobordisms} that given a decorated link cobordism, see Definition \ref{defn: decorated cobordism}, $\Sigma$ from $L$ to $L'$, there is a well defined linear map $F_\Sigma\colon \widehat{HFL}(L)\to \widehat{HFL}(L')$. Using different techniques Zemke defined linear maps for all of the flavours of link Floer homology in \cite{CobMapPaper}. These two techniques were shown to agree for the hat flavour in \cite{MapsAreSamePaper}.

Using link Floer homology and these cobordism maps, Chen generalised the definition of a skein lasagna module to define a Floer lasagna module in \cite{floerlasagna}, which is denoted $\mathcal{FL}(X,L)$. He also proved that similar to the skein lasagna module case there is a 2-handlebody formula for Floer lasagna modules in terms of the cables of the link defining the attaching spheres. Further, for an $L$-space knot $K$ and framing $n\geq 2g(K)$, Gorsky and Hom in \cite{linkLspacernrm} give a formula for the link Floer homology of the $(r,rn)$-cable of $K$ for all $r\geq1$. Hence, using this it should, in theory, be possible to compute the Floer lasagna module of a knot trace $X_n(K)$, that is the 2-handlebody obtained by attaching a single 2-handle to $D^4$ along $K$ with framing $n$. However, in the formula for the Floer lasagna module of the 2-handle maps infinitely many cobordism maps must be computed. This presents a challenge as these maps are defined in terms of holomorphic triangle counts.

In this paper we present a method for reducing the complexity of computing this set of infinitely many cobordism maps to a single computation of a cobordism map $F_{K_{-n}}$ in the case of negative $L$-space knots. Recall that a negative $L$-space knot is a knot that admits a negative surgery where the Heegaard Floer homology of the resulting 3-manifold is as simple as possible, see \cite{LSpaceAlexMonic}. The cobordism $F_{K_{-n}}$ is defined in Definition \ref{defn: F_K map} and is between cables of $K$, where the cable depends on $n$ and the Alexander polynomial $\Delta_K(t)$. Formally we prove the following.

\begin{thm}
\label{thm: main theorem intro version}
    Let $K$ be a negative $L$-space knot and $n\geq2g(K)$. Then there exists a linear map $F_{K_{-n}}$ that is a restriction of a pair of pants cobordism between the link Floer homology of particular cables of $K$, such that if $F_{K_{-n}}$ is non-vanishing then $\mathcal{FL}(X_{-n}(K))$ is infinite dimensional.
\end{thm}

We compute the map $F_{U_{-n}}$ for $U$ the unknot and all $n>0$ in Lemma \ref{lem: FUn is an isomorphism}. Using this lemma and the above theorem we prove the following.

\begin{cor}
    The vector space $\mathcal{FL}(X_{-n}(U))$ is infinite dimensional for all $n> 0$.
\end{cor}

One of the main technical ingredients in the proof of the above corollary is the holomorphic triangle counts carried out in Proposition \ref{prop: band map injective for n<0}. This uses a particularly nice Heegaard triple and so does not easily generalise to other negative $L$-space knots.  However, it is possible there is a more general method for doing these computations, which leads to the following question.

\begin{ques}
    Is $\mathcal{FL}(X_{-n}(K))$ infinite dimensional for all negative $L$-space knots $K$ and $n\geq2g(K)$?
\end{ques}

Before explaining the organisation of the paper we mention a detail about the proof of the main theorem that may be of independent interest. The main theorem relies on a technical result about how the maps on link Floer homology behave under spectral sequences. In particular we consider Ozsvath and Szabo's spectral sequence \cite{OSLinkFloer} for forgetting link components, that is the spectral sequence from $\widehat{HFL}(L)$ to $ \widehat{HFL}(L\backslash L_0)\otimes \mathbb{F}^2$, where $L_0\subset L$ is a single component. We prove that a band map $F_\Sigma \colon \widehat{HFL}(L)\to \widehat{HFL}(L')$ induces a map of spectral sequences and the map on the $E^{\infty}$-page is  $F_{\Sigma\backslash L_0\times I}\otimes \text{Id}_{\mathbb{F}^2}$, see Section \ref{sec: Spectral Sequences} for the relevant definitions and Proposition \ref{prop: spec maps for forgetting components of bands} for the exact statement. We also prove this for the quasi-stabilisation map $T^+$ and as a consequence pair of pants cobordisms. We believe that there should be a general version of these results for a link cobordism $\Sigma$ relating the $E^\infty$-page to the map induced by the cobordism obtained by forgetting all the components intersecting $L_0$. We plan to investigate this in future work.

\subsection{Organisation}
\hfill

In Sections \ref{sec: Link Floer Homology} and \ref{sec: Floer Lasagna Modules} we recall important results we need about link Floer homology and then the definition of Floer lasagna modules. We also use these sections to set up notation for the rest of the paper. Then Section \ref{sec: Spectral Sequences} is dedicated to proving how band maps act under forgetting components. We start this section by recalling notation and facts about spectral sequences. Computations of cables of L-space knots in certain gradings are then carried out in Section \ref{sec: Link Floer Homology of T(r,rn)}. Finally, we pull all of this together and prove the main theorem in Section \ref{subsec: Proof Main Theorem}, then apply it to traces of the unknot in Subsection \ref{subsec: Cobordism Maps}.

\subsection*{Acknowledgments} 
I wish to thank my supervisor Andr\'as Juh\'asz for his continued guidance and encouragement, as well as careful readings of drafts of this paper. I would also like to thank Ian Zemke for helpful correspondence. Finally, this work was supported by the Additional Funding Programme for Mathematical Sciences, delivered by EPSRC (EP/V521917/1) and the
Heilbronn Institute for Mathematical Research.

\section{Link Floer Homology}
\label{sec: Link Floer Homology}
\hfill

We assume the reader has some familiarity with Heegaard Floer homology and in particular link Floer homology. All of our links will be null-homologous in some closed but not necessarily connected 3-manifold $Y$, and we will denote by $\mathbb{L}=(L,\mathbf{w},\mathbf{z})$ a pointed link where we assume there is at least one $w$ and one $z$ basepoint on each component. Further, if $\mathcal{H}$ is a multipointed Heegaard diagram for the pair $(Y,\mathbb{L})$ then we denote by $\widehat{CFL}(\mathcal{H})$ the hat flavour chain complex associated to this Heegaard diagram over the field $\mathbb{F}:=\mathbb{F}_2$. We will always implicitly assume that the Heegaard diagrams mentioned satisfy the necessary admissibility requirements. In the case that the link is in $S^3$ there is a $\mathbb{Z}$-grading called the Maslov grading and a $\mathbb{Z}^{|L|}$-grading called the Alexander multigrading. We will denote by $\widehat{HFL}_i(\mathcal{H},\mathbf{k})$ the homology in Maslov grading $i$ and Alexander grading $\mathbf{k}$ and we denote by $\widehat{HFK}_{i}(\mathcal{H},k)$ the vector space obtained by collapsing the Alexander multigrading to a single grading. That is, the vector space, 
\begin{equation*}
\widehat{HFK}_{i}(\mathcal{H},k):=\bigoplus_{\substack{ |\mathbf{k}|=k}}\widehat{HFL}_i(\mathcal{H},\mathbf{k}),
\end{equation*}
where $|\mathbf{k}|=\mathbf{k}_1+\ldots + \mathbf{k}_{|L|}$. As we will be considering maps associated to decorated link cobordisms we will need the notion of a transitive system from \cite{Naturality}, but we follow the notation of \cite{CobMapPaper}.

\begin{defn}
    A \textit{transitive system} in a category $\mathcal{C}$ is a collection of objects $\{C_{i}\}_{i \in I}$ and for every $(i,j)\in I\times I$ a \text{distinguished morphism} $\Phi_{i \to j}\colon C_{i}\to C_{j}$ such that,
    \begin{itemize}
        \item $\Phi_{i\to i}=\text{Id}_{C_{i}}$ for all $i \in I$;
        \item $\Phi_{j \to k}\circ \Phi_{i \to j}=\Phi_{i \to k}$ for all $i,j,k \in I$.
    \end{itemize}
    Further, let $\{C_{i}\}_{i\in I}$ and $\{D_{j}\}_{j\in J}$ be transitive systems in a category $\mathcal{C}$, then a \textit{morphism of transitive systems} is a collection of morphisms $\{F_{i\to j}\}_{(i,j)\in I\times J}$ such that the following diagram commutes for all $i,i'\in I$ and $j,j' \in J$,
    \begin{equation}
    \label{eq: morphism of transitive systems}
        \begin{tikzcd}
            C_{i} \arrow[rr,"F_{i\to j}"] \arrow[dd,"\Phi_{i\to i'}"] && D_{j} \arrow[dd,"\Phi_{j\to j'}"] \\
            &&\\
            C_{i'} \arrow[rr,"F_{i'\to j'}"] && D_{j'}.
        \end{tikzcd}
    \end{equation}
\end{defn}

An important thing to note from this is that if $F$ is a morphism of transitive systems, then it is completely determined by one of the maps $F_{i \to j}$. The two categories we are particularly interested in are vector spaces and chain complexes over $\mathbb{F}$. In the case of chain complexes we ask for the distinguished morphisms to be chain homotopy equivalences and we will only ask that the diagram \eqref{eq: morphism of transitive systems} commutes up to chain homotopy equivalence. Further, given a transitive system of vector spaces $\{V_{i}\}_{i \in I}$ a canonical vector space $V_{I}$ can be defined in such a way that if $F$ is a morphism of transitive systems of vector spaces from $\{V_i\}_{i\in I}$ to $\{W_{j}\}_{j \in J}$, then there is a canonically defined linear map $F\colon V_{I}\to W_J$, see \cite{Naturality} for details. What is proven in \cite{Naturality} is that the collection $\{\widehat{CFL}(\mathcal{H})\}_{\mathcal{H}}$ over all Heegaard diagrams for $(Y,\mathbb{L})$ forms a transitive system of chain complexes. Formally, we have the following.

\begin{lem}
\label{lem: transitive system for hat}
    Let $\mathcal{H}$ and $\mathcal{H}'$ be Heegaard diagrams for $(Y,\mathbb{L})$, then there exists a chain homotopy equivalence,
    \begin{equation*}
        \Phi_{\mathcal{H}\to\mathcal{H'}}\colon \widehat{CFL}(\mathcal{H})\to \widehat{CFL}(\mathcal{H}'),
    \end{equation*}
    that is well defined up to chain homotopy equivalence and in the case of $Y \cong S^3$ preserves the Maslov grading and Alexander multigrading. That is there is a transitive system of chain complexes $\{\widehat{CFL}(\mathcal{H})\}_{\mathcal{H}}$ indexed over all Heegaard diagrams for $(Y,\mathbb{L})$.
\end{lem}

We will denote this transitive system by $\widehat{CFL}(Y,\mathbb{L})$ and by $\widehat{CFL}(\mathbb{L})$ when $Y\cong S^3$. Further, note that by taking homology of each of these we get a transitive system of vector spaces $\{\widehat{HFL}(\mathcal{H})\}_{\mathcal{H}}$ that preserves the gradings when they exist. Hence, there is a well defined vector space $\widehat{HFL}(Y,\mathbb{L})$ associated to a multipointed link in $Y$ and it has a $\mathbb{Z}\oplus \mathbb{Z}^{|L|}$ grading in $S^3$. To define morphisms between these it is not enough to just consider link cobordisms; instead, the notion of decorated link cobordisms is needed. These were introduced by Juhasz in \cite{AndrasCobordisms}, but we will use the reformulated definition from \cite{CobMapPaper}. 

\begin{defn}
\label{defn: decorated cobordism}
    A \textit{decorated link cobordism} from $(Y_0,\mathbb{L}_0)$ to $(Y_1,\mathbb{L}_1)$ is a pair $(W^4,\mathcal{F})$, denoted $(W,\mathcal{F})\colon (Y_0,\mathbb{L}_0)\to(Y_1, \mathbb{L}_1)$, where $W$ is a cobordism from $Y_0$ to $Y_1$ and $\mathcal{F}=(\Sigma,\mathcal{A})$ is a pair, such that
    \begin{itemize}
        \item $\Sigma \subseteq W$ is a properly embedded compact oriented surface with $\partial\Sigma\cap Y_i=L_i$;
        \item $\mathbb{\mathcal{A}}\subseteq \Sigma$ is a properly embedded compact 1-manifold such that each component of $L_i\backslash\ \partial \mathcal{A}$ contains one basepoint;
        \item $\Sigma$ is partitioned into two subsurfaces $\Sigma_{\mathbf{w}}$ and $\Sigma_{\mathbf{z}}$ that meet along $\mathcal{A}$ with $\mathbf{w}_{0}\cup \mathbf{w}_1\subseteq \Sigma_{\mathbf{w}}$ and $\mathbf{z}_{0}\cup \mathbf{z}_1\subseteq \Sigma_{\mathbf{z}}$.
    \end{itemize}
    Further, a \textit{colouring} of a decorated link cobordism is a pair $(J,\mathbb{J})$ where $\mathbb{J}$ is a finite set with the discrete topology and $J\colon \Sigma\backslash \mathcal{A}\to \mathbb{J}$ is a continuous function such that $J(\Sigma_\mathbf{w})$ and $J(\Sigma_{\mathbf{z}})$ are disjoint in $\mathbb{J}$, and we define, 
    \begin{equation*}
        (\Sigma_{j})_{\mathbf{w}}=J^{-1}(j)\cap\Sigma_{\mathbf{w}} \qquad \text{and} \qquad (\Sigma_{j})_{\mathbf{z}}=J^{-1}(j)\cap\Sigma_{\mathbf{z}}.
    \end{equation*}
\end{defn}

The colourings above are technically ``type-partitioned, indexed colourings" in the language of \cite{CobMapPaper}, however we do not need this level of generality. Further, the main colouring that we consider is $\mathbb{J}=\mathbb{J}_w\cup\mathbb{J}_z$ where each of $\mathbb{J}_w,\mathbb{J}_z$ is the set of connected components of $\Sigma$, and $J$ maps a connected component of $\Sigma_{\boldsymbol{\delta}}$ to the connected component of $\Sigma$ it lies on in $\mathbb{J}_{\boldsymbol{\delta}}$ for $\boldsymbol{\delta}\in \{\mathbf{w},\mathbf{z}\}$. For the Alexander grading this amounts to collapsing the Alexander gradings of all of the components that lie on the same connected component of $\Sigma$. Also we will often abuse notation and write $\Sigma\colon \mathbb{L}_0\to \mathbb{L}_1$ when the decorated link cobordism is in $S^3\times I$ and the set $\mathcal{A}$ is clear.

To a decorated link cobordism both \cite{AndrasCobordisms} and \cite{CobMapPaper} associate morphisms of transitive systems of chain complexes. The following theorem is a combination of theorems from the previous papers and the grading change formulas from \cite{CobGradings}.

\begin{thm}
\label{thm: grading change formulas}
    Let $(W,\mathcal{F})\colon (Y_0,\mathbb{L}_0)\to(Y_1, \mathbb{L}_1)$ be a decorated link cobordism then there is a morphism of transitive systems of chain complexes, 
    \begin{equation*}
        F_{(W,\mathcal{F})}\colon \widehat{CFL}(Y_0,\mathbb{L}_0)\to \widehat{CFL}(Y_1,\mathbb{L}_1).
    \end{equation*}
    Further, let $J$ be a colouring of $\Sigma$ and suppose $W=S^3\times I$. Then for all $v \in \widehat{CFL}(\mathbb{L})$ and $j \in J$ the following grading change formulas hold:
    \begin{align}
        &M\left(F_{\Sigma}(v)\right)-M(v)=\chi(\Sigma_{w})-|\mathbf{w}_{1}|,\label{eq: MasShift}\\
        &A_{j}(F_{\Sigma}(v))-A_{j}(v)=\frac{\chi\left((\Sigma_{j})_{\mathbf{w}}\right)-\chi\left((\Sigma_{j})_{\mathbf{z}}\right)}{2}, \label{eq: AlexShift}
    \end{align}
    where $A_{j}$ is understood to mean the sum of the Alexander gradings of all the link components intersecting $(\Sigma_{j})_{\mathbf{w}}$.
\end{thm}

Note from this that if we compute such a cobordism map for a single Heegaard diagram then this determines the cobordism maps for all Heegaard diagrams. There are two link cobordism maps that we will use throughout the paper so we recall their definitions now for ease of reading. The first is a quasi-stabilisation and we first recall the following topological preliminary.

\begin{defn}
    Let $\mathcal{H}=(\Sigma,\boldsymbol{\alpha},\boldsymbol{\beta},\mathbf{w},\mathbf{z})$ be a Heegaard diagram for a link $\mathbb{L}$ and let $w$ and $z$ be two basepoints adjacent to each other on the link component $L_i$. Further, let $R$ be the region of $\Sigma \backslash\boldsymbol{\alpha}$ that contains both $w$ and $z$ and let $p$ be a point in $R$. Then pick a curve $\alpha_s$ in $R$ such that it separates $w$ and $z$ and take $\beta_s$ to be a small circle around $p$. Then we define a new Heegaard diagram,
    \begin{equation*}
        \mathcal{H}^+:=(\Sigma,\boldsymbol{\alpha}\cup\{\alpha_s\},\boldsymbol{\beta}\cup\{\beta_s\},\mathbf{w}\cup\{w_s\},\mathbf{z}\cup\{z_s\})
    \end{equation*}
    where $w_s$ and $z_s$ are points both contained in the ball bounded by $\beta_s$ and such that $w_s$ and $z$ are in the same component of $R\backslash\alpha_s$ and similarly $z_s$ and $w$ are.
\end{defn}

\begin{figure}[h]
        \centering
        \def\svgwidth{0.5\textwidth}
        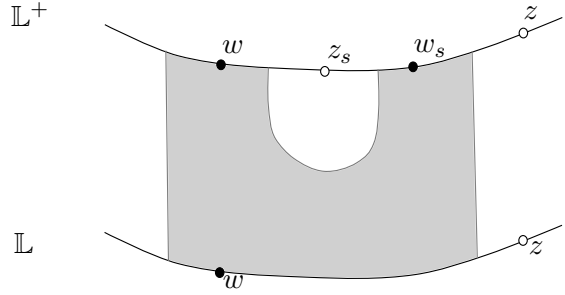
        \caption{Decoration for the quasi-stabilisation map $T^{+}$.}
        \label{fig: QS map dec}
\end{figure}

Note that $\mathcal{H}^+$ is a Heegaard diagram for the same underlying link $L$ except the component $L_i$ now has two additional basepoints. We will denote this new pointed link by $\mathbb{L}^+$. In \cite{QuasiMap} it is shown that this construction can be used to define what is called a quasi-stabilisation map and that the resulting map is a morphism of transitive systems. We recall the definition now.

\begin{lem}
\label{lem: Quasi stab link}
    Let $\mathcal{H}$ and $\mathcal{H}^+$ be as above. Then,
    \begin{equation*}
        \widehat{CFL}(\mathcal{H}^+)\cong\widehat{CFL}(\mathcal{H})\otimes\mathbb{F}\langle\xi^{\mathbf{w}},\xi^\mathbf{z}\rangle,
    \end{equation*}
    where $M(\xi^{\mathbf{w}})=-1,M(\xi^{\mathbf{z}})=0,A_{j}(\xi^{\mathbf{w}})=-\frac{1}{2}\delta_{ij}$ and $A_{j}(\xi^{\mathbf{z}})=\frac{1}{2}\delta_{ij}$. Further, \textit{the quasi-stabilisation map},
    \begin{equation*}
        T^{+}\colon \widehat{CFL}(\mathbb{L})\to \widehat{CFL}(\mathbb{L}^{+})
    \end{equation*}
    associated to the decorated cobordism $(L\times I,\mathcal{A})$, where $\mathcal{A}$ is the decoration in Figure \ref{fig: QS map dec}, is given by $T^{+}(v)=v\otimes \xi^{\mathbf{w}}$ for all $v \in \widehat{CFL}(\mathcal{H})$.
\end{lem}

Note that in \cite{CobMapPaper} there are three other quasi-stabilisation maps but we will not need these. The second cobordism we will need is a band map. We recall the definition of an $\alpha$-band.

\begin{defn}
    Let $\mathbb{L}$ be a pointed link in $S^3$ then a \text{band} on $L$ is an embedding $B\colon I\times I\to S^3$, such that $\text{im}(B)\cap L=B(\{0,1\}\times I)$. We say $B$ is an \textit{$\alpha$-band} if the ends of $B$ lie in components of $L\backslash(\mathbf{w}\cup \mathbf{z})$ oriented from the $\mathbf{w}$ to the $\mathbf{z}$ basepoints. Denote by $\mathbb{L}(B)$ the link that is obtained by smoothing the corners of $(L\backslash \text{im}(B))\cup B(I\times\{0,1\})$.
\end{defn}

There is clearly a saddle cobordism corresponding to such a band as pictured in Figure \ref{fig: Band map dec}. To define the corresponding map on link Floer complexes we need the following definition of a Heegaard triple subordinate to this band.

\begin{figure}[h]
        \centering
        \def\svgwidth{0.5\textwidth}
        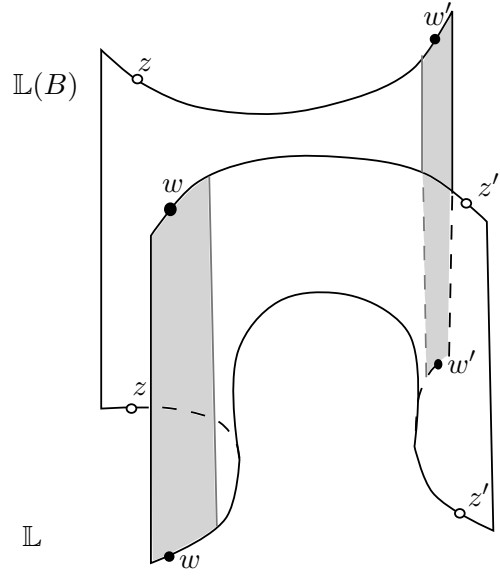
        \caption{Decoration for the cobordism associated to an $\alpha$-band.}
        \label{fig: Band map dec}
\end{figure}

\begin{defn}
\label{defn: subordinate to B}
    Let $\mathbb{L}$ be a pointed link and $B$ an $\alpha$-band. Then a Heegaard triple 
    \begin{equation*}
        \mathcal{T}=(\Sigma,\boldsymbol{\alpha}'=\{\alpha'_1,\ldots\alpha'_n\},\boldsymbol{\alpha}=\{\alpha_1,\ldots\alpha_n\},\boldsymbol{\beta}=\{\beta_1,\ldots,\beta_n\},\mathbf{w},\mathbf{z}),
    \end{equation*}
    is \textit{subordinate to $B$} if the following hold:
    \begin{enumerate}
        \item let $\Sigma_{0}=\Sigma\backslash N(\mathbf{w}\cup\mathbf{z})$, then the manifold obtained from $\Sigma_{0}\times I$ by attaching 3-dimensional 2-handles along $\alpha_{1},\ldots,\alpha_{n-1}\subset \Sigma_{0}\times\{0\}$ and $\beta_{1},\ldots ,\beta_{n}\subset \Sigma_{0}\times \{1\}$ is $S^3\backslash N(L\cup B)$;
        \item the curves $\alpha'_{1},\ldots, \alpha'_{n-1}$ are small Hamiltonian isotopies of the curves $\alpha_{1},\ldots,\alpha_{n-1}$ such that $|\alpha_{i}\cap\alpha_{j}'|=2\delta_{ij}$;
        \item let $Q$ be the arcs of $L\backslash(\mathbf{w}\cup \mathbf{z})$ that contain the ends of $B$, then $\alpha_{n}$ is the projection onto $\Sigma\backslash(\alpha_{1}\cup\ldots\cup\alpha_{n-1})$ of a curve on the boundary of $N(Q\cup B)$ that bounds a disc in $N(Q\cup B)$ that separates the components of $Q$, $\alpha_{n}'$ is defined similarly for $L(B)$. 
    \end{enumerate}
\end{defn}

Let $\mathcal{T}=(\Sigma,\boldsymbol{\alpha}',\boldsymbol{\alpha},\boldsymbol{\beta},\mathbf{w},\mathbf{z})$ be a Heegaard triple subordinate to some $\alpha$-band $B$, then we denote by
\begin{equation*}
    \mathcal{T}_{\alpha,\beta}:=(\Sigma,\boldsymbol{\alpha},\boldsymbol{\beta},\mathbf{w},\mathbf{z}), \qquad \mathcal{T}_{\alpha',\beta}:=(\Sigma,\boldsymbol{\alpha}',\boldsymbol{\beta},\mathbf{w},\mathbf{z}) \qquad \text{and} \qquad \mathcal{T}_{\alpha',\alpha}:=(\Sigma,\boldsymbol{\alpha}',\boldsymbol{\alpha},\mathbf{w},\mathbf{z}).
\end{equation*}
Then from this definition we have that $\mathcal{T_{\alpha,\beta}}$ and $\mathcal{T}_{\alpha',\beta}$ are Heegaard diagrams for $\mathbb{L}$ and $\mathbb{L}(B)$ respectively. It was also shown in \cite{CobMapPaper} that $\mathcal{T}_{\alpha',\alpha}$ is a Heegaard diagram for a $(|\mathbf{w}|-1)$-component unlink in $(S^1\times S^2)^{\#g(\Sigma)}$, with two basepoints on each component except one, which has four. Due to this, $\widehat{HFL}(\mathcal{T}_{\alpha',\alpha})$ has a unique element of highest Maslov grading. Finally, before defining the maps associated to these band maps we recall the definition of the holomorphic triangle count map associated to a Heegaard triple.

\begin{defn}
\label{defn: triangle count}
    Let $\mathcal{T}=(\Sigma,\boldsymbol{\alpha}',\boldsymbol{\alpha},\boldsymbol{\beta},\mathbf{w},\mathbf{z})$ be a Heegaard triple. Then there is a chain map,
    \begin{equation*}
        F_{\mathcal{T}}\colon \widehat{CFL}(\mathcal{T}_{\alpha',\alpha})\otimes \widehat{CFL}(\mathcal{T}_{\alpha,\beta}) \to \widehat{CFL}(\mathcal{T}_{\alpha',\beta}),
    \end{equation*}
    defined for $\boldsymbol{\theta}\in \mathbb{T}_{\boldsymbol{\alpha}'}\cap \mathbb{T}_{\boldsymbol{\alpha}} $ and $ \mathbf{x}\in \mathbb{T}_{\boldsymbol{\alpha}}\cap \mathbb{T}_{\boldsymbol{\beta}}$ by the following,
    \begin{equation}
        F_{\mathcal{T}}(\boldsymbol{\theta}\otimes \boldsymbol{x})= \sum_{\mathbf{y} \in \mathbb{T}_{\boldsymbol{\alpha}'}\cap \mathbb{T}_{\boldsymbol{\beta}}} \sum_{\substack{\psi \in \pi_{2}(\boldsymbol{\theta},\mathbf{x},\mathbf{y}) \\ \mu(\psi)=0 \\ n_{\mathbf{w}}(\psi)=n_{\mathbf{z}}(\psi)=0}} \# \mathcal{M}(\psi)\mathbf{y}.
    \end{equation}
\end{defn}

Combining all of these we get the map associated to the cobordism associated to an $\alpha$-band $B$ with the decorations as in Figure \ref{fig: Band map dec}.

\begin{lem}
\label{lem: defn of band map}
    Let $\mathbb{L}$ be a pointed link and $B$ be an $\alpha$-band. Further, let $F_{\Sigma}$ be the morphism of transitive systems associated to the cobordism coming from the band $B$ with decorations as in Figure \ref{fig: Band map dec}. Then for any Heegaard triple $\mathcal{T}$ subordinate to $B$ the linear map,
    \begin{equation*}
        F_{\Sigma}\colon \widehat{CFL}(\mathcal{T}_{\alpha,\beta})\to \widehat{CFL}(\mathcal{T}_{\alpha',\beta}),
    \end{equation*}
    is given for all $\mathbf{x}\in\widehat{CFL}(\mathcal{T}_{\alpha,\beta})$ by $F_{\Sigma}(\mathbf{x})=F_{\mathcal{T}}(\boldsymbol{\theta}^+\otimes\mathbf{x})$, where $\boldsymbol{\theta}^+\in \widehat{CFL}(\mathcal{T}_{\alpha',\alpha})$ is any representative of the maximum graded homology class.
\end{lem}

Finally, in this section we recall the definition of cylindrical boundary degenerations of curves. We will need this for some neck stretching arguments. Recall that in the Lipshitz cylindrical setup we consider certain $J$-holomorphic embeddings, 
\begin{equation*}
    u\colon S \to \Sigma \times [0,1]\times \mathbb{R},
\end{equation*}
where $S$ is a Riemann surface with $2n$ punctures $\{p_1,\ldots,p_n,q_1\ldots, q_n\}\subseteq\partial S$ and $J$ is an almost complex structure satisfying certain properties, see \cite{CylindricalReformulation} for details. Now recall the definition of cylindrical boundary degenerations, we use the formulation from \cite{GraphCobZemke}.

\begin{defn}
    Let $S$ be a Riemann surface with punctures $\{p_1,\ldots,p_n\}\subseteq \partial S$. Then a \textit{cylindrical $\boldsymbol{\alpha}$-boundary degeneration} is a $J$-holomorphic embedding,
    \begin{equation*}
        u\colon S \to \Sigma\times (-\infty,1]\times \mathbb{R},
    \end{equation*}
    such that the following hold,
    \begin{enumerate}
        \item $u(\partial S)\subseteq \boldsymbol{\alpha}\times (-\infty,1]\times \mathbb{R}$;
        \item the map $\pi_{(-\infty,1]\times \mathbb{R}}\circ u$ is non-constant on each component of $S$;
        \item for each $\alpha_i\in \boldsymbol{\alpha}$, the preimage $u^{-1}(\alpha_i\times (-\infty, 1]\times \mathbb{R})$ consists of exactly one component of $\partial S$;
        \item $u$ has finite energy;
        \item and if $\{x_i\}_{i \in I}$ is a sequence of points in $S$ approaching a puncture $p_j$ then $(\pi_{\mathbb{R}}\circ u)(x_i)$ approaches $\pm \infty$.
    \end{enumerate}
    A \textit{cylindrical $\boldsymbol{\beta}$-boundary degeneration} is defined similarly.
\end{defn}

\section{Floer Lasagna Modules}
\label{sec: Floer Lasagna Modules}

We recap here the definition of Floer lasagna modules and the construction of $c\widehat{HFK}(K_n)$ from \cite{floerlasagna}. For some properties and more exposition for Floer lasagna modules see \cite{floerlasagna} and for the definition of skein lasagna modules that this is based on see \cite{OGSkLas}. Further, see \cite{exoticsklas} for a more categorical approach to the definition. We start by recalling the definition of a lasagna filling.

\begin{defn}
    Let $W$ be a 4-manifold and $\mathbb{L}\subseteq \partial W$ a framed and pointed link, that may be the empty link. Then a \textit{lasagna filling of $(W,\mathbb{L})$} is a tuple 
    \begin{equation*}
        F=((\Sigma,\mathcal{A}),\{(B_{i},\mathbb{L}_{i},v_i)\}_{i \in I})
    \end{equation*}
    for some finite indexing set $I$, such that for all $i \in I$:
    \begin{enumerate}
        \item the $B_{i}$'s are disjoint embedded copies of $D^4$ in $W$;
        \item $\mathbb{L}_{i}\subseteq \partial B_{i}$ is a pointed framed link;
        \item $(\Sigma,\mathcal{A})$ is a decorated surface in $W\backslash(\cup_{i\in I}B_{i})$ such that $\partial W\cap \partial\Sigma=L$ and $\partial B_{i}\cap \partial \Sigma=L_{i}$;
        \item and $v_{i}\in\widehat{HFL}(\mathbb{L}_i)$.
    \end{enumerate}
\end{defn}

We can then consider the infinite dimensional vector space $\mathbb{F}\{\text{lasagna fillings of }(W,\mathbb{L})\}$. To turn this into a more interesting invariant we define an equivalence relation on this vector space.

\begin{defn}
    Let $F=((\Sigma,\mathcal{A}), \{(B_{i},\mathbb{L}_{i},v_{i})\}_{i\in I})$ and $F=((\Sigma',\mathcal{A}'), \{(B_{j}',\mathbb{L}_{j}',v_{j}')\}_{j\in J})$ be lasagna fillings of $(W,\mathbb{L})$. We say $F\sim_{1} F'$ if the following hold:
    \begin{enumerate}
        \item we can partition $J=\bigsqcup_{i\in I} J_{i}$ such that for each set $i\in I$ we have $\bigcup_{j\in J_{i}}B_{j}'\subseteq B_{i}$;
        \item $(\Sigma',\mathcal{A}')$ is isotopic, relative boundaries, to $(\Sigma,\mathcal{A})$ in $W\backslash(\cup_{i\in I}B_{i})$;
        \item for each $i\in I$, $\Sigma'\cap\partial B_{i}$ is isotopic to $\mathbb{L}_{i}$ where the basepoints are determined by the subsurfaces $\Sigma_{\mathbf{w}}$ and $\Sigma_{\mathbf{z}}$;
        \item for each $i\in I$, the map $F_{\mathcal{F}_i}$ coming from the link cobordism
        \begin{equation*}
            \mathcal{F}_{i}=(\Sigma',\mathcal{A}')\cap B_{i}\colon \cup_{j\in J_{i}}\mathbb{L}_{j}\to \mathbb{L}_{i}
        \end{equation*}
        has the property $F_{\mathcal{F}_{i}}(\otimes v_{j})=v_{i}$.
    \end{enumerate}
    We then define an equivalence relation $\sim$ on $\mathbb{F}\{\text{lasagna fillings of }(W,\mathbb{L})\}$ as the transitive and linear closure of $\sim_{1}$ and multilinearity in the $v_{i}$. That is, if $F=((\Sigma,\mathcal{A}), \{B_{i},L_{i},v_{i}\}_{i\in I})$ is a lasagna filling and for some $j\in I$ we have $v_{j}=w_{j}+z_{j}$, then 
    \begin{equation*}
        F\sim ((\Sigma,\mathcal{A}), \{B_{i},L_{i},w_{i}\}_{i\in I})+(\Sigma,\mathcal{A}), \{B_{i},L_{i},z_{i}\}_{i\in I})
    \end{equation*}
    where $w_{i}=z_{i}=v_{i}$ for $i\neq j$.
\end{defn}

Using this equivalence relation we can define the Floer lasagna module of the pair $(W,\mathbb{L})$.

\begin{defn}
    We define the \textit{Floer Lasagna Module of $(W,\mathbb{L})$} to be,
    \begin{equation}
        \mathcal{FL}(W,\mathbb{L}):=\mathbb{F}\{\text{lasagna fillings in }(W,\mathbb{L})\}/\sim.
    \end{equation}
\end{defn}

An observation made in \cite{floerlasagna} is that if any two lasagna fillings are equivalent then the surfaces must represent the same element in $H_{2}(W,L)$. Hence, we can write $\mathcal{FL}(W,\mathbb{L};\alpha)$ for the lasagna fillings in the class $\alpha$ and 
\begin{equation*}
    \mathcal{FL}(W,\mathbb{L})=\bigoplus_{\alpha\in H_{2}(W,L)}\mathcal{FL}(W,\mathbb{L};\alpha).
\end{equation*}
In fact Chen defines a further bigrading on $\mathcal{FL}(W,\mathbb{L})$ coming from the Maslov and Alexander gradings, see \cite{floerlasagna} for the proof.

\begin{lem}
    Let $[F=((\Sigma,\mathcal{A}), \{B_{i},\mathbb{L}_{i},v_{i}\}_{i\in I})] \in \mathcal{FL}(W,\mathbb{L};\alpha)$ be an equivalence class of lasagna fillings then the following gradings are well defined:
    \begin{align*}
         &M([F])=\chi(\Sigma_{\mathbf{w}})+\sum_{i\in I}M(v_{i}); \\
         &A([F])=\frac{\chi(\Sigma_{\mathbf{w}})-\chi(\Sigma_{\mathbf{z}})}{2}+\sum_{i\in I}A(v_{i}).
     \end{align*}
     Therefore, $\mathcal{FL}(W,\mathbb{L};\alpha)$ is bigraded, that is 
     \begin{equation*}
         \mathcal{FL}(W,\mathbb{L};\alpha)=\bigoplus_{(i,j)\in \mathbb{Z}^2}\mathcal{FL}_{i}(W,\mathbb{L};\alpha,j),
     \end{equation*}
     where $i$ is the Maslov grading and $j$ the Alexander grading.
\end{lem}

With the definition in hand, we now define $c\widehat{HFK}_n(K)$, which by a result from \cite{floerlasagna} is isomorphic as a graded vector space to $\mathcal{FL}(X_{n}(K))$. This vector space was defined by Chen in \cite{floerlasagna} and was based on a similar formula from \cite{SkeinLasagna2Handle}. We will only recall the definition in the case of a knot and no link in the boundary as this is all we will need for the paper.

The motivation for the definition comes from the observation that for a knot trace $X_n(K)$, if we consider the subspace 
\begin{equation*}
    Y=N\left((\{0\}\times D^2) \cup\partial X_n(K)\right),
\end{equation*}
that is a neighbourhood of the cocore of the 2-handle and the boundary, then $B:=X_n(K)\backslash Y$ is a 4-ball. Further, any surface $\Sigma\subseteq int(X_n(K))$ intersects $Y$ as a collection of parallel copies of $(D^2\times \{0\}) \cap Y$. Therefore, every lasagna filling is equivalent to a lasagna filling, 
\begin{equation*}
    \big((\sqcup_{r}D^2,\mathcal{A}),\{(B,\tilde{K},v)\}_r\big),
\end{equation*}
where $\tilde{K}$ is an $r$-component cable of $K$ with some orientations and $v \in \widehat{HFK}(\tilde{K})$. These fillings are then used to give the isomorphism with $c\widehat{HFK}(K_n)$. We start by defining the cables of $K$ that are needed.

\begin{defn}
    Let $K$ be a knot with framing $n\in \mathbb{Z}$ in $S^3$, and $f\colon S^1\times D^2\to N(K)$ be a parametrisation of the framing. Further, let $a_+,a_-\in \mathbb{Z}_{\geq 0}$ and 
    \begin{equation*}
        x_{1}^{+},\ldots,x_{a_+}^{+},x_{1}^{-}, \ldots,x_{a_-}^{-}\in \{1\} \times D^2,
    \end{equation*}
    then we define, as an unoriented link, 
    \begin{equation*}
        K_n(a_+,a_-)=f(S^{1}\times\{x_{1}^{+},\ldots,x_{a_-}^{-}\}).
    \end{equation*}
    We then orientate $f(S^{1}\times\{x_{i}^{\pm}\})$ in the same direction as $K$ for $x_{i}^+$, we say \textit{positively oriented}, and in the opposite direction, \textit{negatively oriented}, for $x_{i}^{-}$.  Finally we add $w$ and $z$ basepoints at $f(1,x_{i}^{\pm})$ and $f(-1,x_{i}^{\pm})$ respectively.
\end{defn}

The definition takes a direct sum over all of the link Floer homology of these and then quotients out by certain images of some cobordisms. We recall the definition of these cobordisms now starting with the basepoint moving map.

\begin{figure}[h]
    \centering
    \def\svgwidth{0.5\textwidth}
    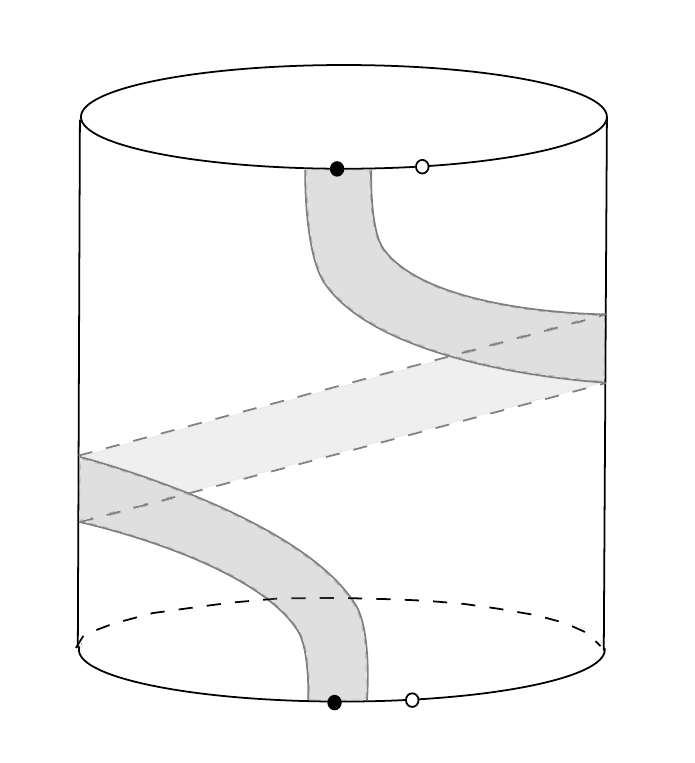
    \caption{Decoration for the basepoint moving map.}
    \label{fig: Basepoint moving cobordism}
\end{figure}

\begin{defn}
    Let $K_n(a_+,a_-)$ be as above, let $j \in \{1,\ldots,a_++a_-\}$ and $K_j\subseteq K_n(a_+,a_-)$ be the component $f(S^1\times\{x_{j}\})$, where $x_j=x_{j}^+$ if $j\leq a_+$ and $x_{j-a_+}^-$ if $j>a_+$. Then the \textit{basepoint moving map} $m_j$ is defined to be the cobordism $K_n(a_+,a_-)\times I$ with the decoration $\mathcal{A}$ that on $K_j$ is the decoration in Figure \ref{fig: Basepoint moving cobordism} and on every other component is $f(\pm i\times\{x\})\times I$. The associated map is denoted 
    \begin{equation*}
        F_{m_j}\colon \widehat{HFL}(K_n(a_+,a_-))\to  \widehat{HFL}(K_n(a_+,a_-)).
    \end{equation*}
\end{defn}

The only thing we need to note about this map is that $F_{m_j}$ is clearly an isomorphism, as the inverse is given by the decoration going the other way. Further, by applying the equations in Theorem \ref{thm: grading change formulas} we can see it also preserves the gradings. The next cobordisms are those associated to the braid action.

\begin{defn}
    Let $B_{a_++a_-}$ be the braid group on $a_++a_-$ strands and let it act on,
    \begin{equation*}
        \{1\}\times (D^2\backslash \{x_{1}^{+}, \ldots, x_{a_+}^{+}, x_{1}^{-}, \ldots, x_{a_-}^{-}\}).
    \end{equation*}
    Then define $B_{a_+,a_-}$ to be the subgroup that fixes both $\{x_{1}^{+},\ldots, x_{a_+}^{+}\}$ and $\{x_{1}^{-},\ldots,x_{a_-}^{-}\}$ setwise. Then there is an action on $K_n(a_+,a_-)$ by $B_{a_+,a_-}$, and for any $\tau \in B_{a_+,a_-}$ we denote by $\Sigma_\tau$ the cobordism from $K_n(a_+,a_-)$ to itself coming from seeing $\tau$ as a diffeomorphism. Further, a decoration $\mathcal{A}$ is defined by considering the image of $f(i\times \{x^\pm_j\})$ in the cobordism. Finally we denote by
    \begin{equation*}
        F_\tau\colon \widehat{HFK}(K_n(a_+,a_-))\to  \widehat{HFK}(K_n(a_+,a_-)),
    \end{equation*}
    the associated map.
\end{defn}

Note that again this is an isomorphism, with inverse given by $F_{\tau^{-1}}$, and preserves grading. To define the final cobordism map we need the following fact about link Floer homology.

\begin{lem}
\label{lem:  unknot component}
    Let $\mathbb{L}$ be a pointed link with $n$ components and $\mathbb{L}\sqcup U$ be the split link where $U$ is an unknot and has exactly two basepoints. Then we have the following isomorphism,
    \begin{equation}
        \widehat{HFL}(\mathbb{L}\sqcup U)\cong \widehat{HFL}(\mathbb{L})\otimes V,
    \end{equation}
    where $V$ is the vector space $\mathbb{F}^2$ generated by $B$ and $T$, with $M(B)=-1$, $M(T)=0$ and $A(B)=A(T)=\mathbf{0}$.
\end{lem}

Finally, using this we can define the pair of pants cobordism.

\begin{figure}[h]
    \centering
    \def\svgwidth{0.5\textwidth}
    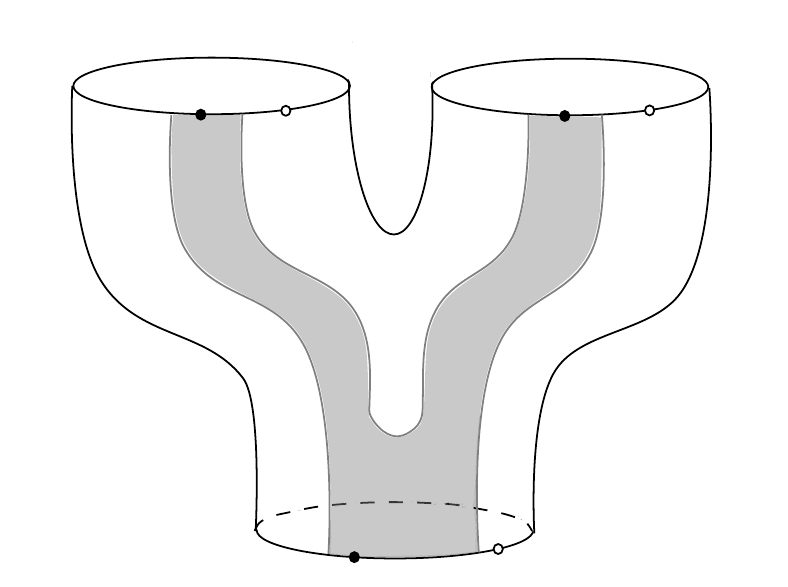
    \caption{Decoration for the pair of pants cobordism.}
    \label{fig: Pair of pants cobordism}
\end{figure}

\begin{defn}
\label{defn: pair of pants cobordism}
    Let $K_+$ and $K_-$ be two oppositely oriented components of $K_n(a_+,a_-)$. Then there is an annulus in $S^3\backslash K_n(a_+,a_-)$ bounded by these components. Push this annulus $A$ into $S^3\times I$ so that $K_n(a_+,a_-)\subseteq S^3\times\{1\}$ and the annulus intersects $S^3\times \{0\}$ in a small disc $D$. Then we define the cobordism 
    \begin{equation*}
        P:=\big((K_n(a_+,a_-)\backslash K_+\cup K_-)\times I\big)\cup (A\backslash \text{int}(D))\subseteq S^3\times I
    \end{equation*}
    and we decorate this with the trivial decoration on the cylinders and the decoration in Figure \ref{fig: Pair of pants cobordism} on the annulus. The associated map is then denoted, using Lemma \ref{lem:  unknot component},
    \begin{equation*}
        F_P\colon \widehat{HFL}(K_n(a_+-1,a_--1))\otimes V \to \widehat{HFL}(K_n(a_+,a_-)).
    \end{equation*}
\end{defn}

Note that we will often abuse notation and use $F_{P}$ to refer to both the map on homology and chain complexes, when the meaning is clear. Further, by using Theorem \ref{thm: grading change formulas} we can see that $F_{P}$ preserves the Alexander grading and shifts the Maslov grading by $-1$. Finally, note that $F_P$ is a composition of a quasi-stabilisation map and a band map. With these cobordism maps defined we can now state the definition of $c\widehat{HFK}(K)$ and the theorem that it is isomorphic to $\mathcal{FL}(X_n(K))$ from \cite{floerlasagna}.

\begin{defn}
    Let $K$ be a knot with framing $n$ in $S^3$. We define the \textit{$n$-cabled knot Floer homology} of $K$ as follows,
    \begin{equation}
        c\widehat{HFK}(K_n)=\left(\bigoplus_{a_+,a_- \in \mathbb{Z}_{\geq 0}}\widehat{HFK}(K_n(a_+,a_-))[a_++a_-]              \right)/\sim
    \end{equation}
    where $\sim$ is the linear and transitive closure of the following relations:
    \begin{enumerate}
        \item $v \sim F_{\tau}(v)$ for all $\tau \in B_{a_+,a_-}$;
        \item $v \sim F_{P}(v \otimes B)$ and $0 \sim F_{P}(v \otimes T)$;
        \item and $v \sim F_{m_{j}}(v)$ for all $1\leq j\leq a_++a_-$.
    \end{enumerate}
    Here, $F_{\tau},F_{P}$ and $F_{m_{j}}$ are the previously defined cobordism maps and $[\cdot]$ means the Maslov grading is shifted up by this amount.
\end{defn}

As noted in \cite{floerlasagna} the above construction inherits a bigrading from the Maslov and Alexander gradings of the $K_n(a_+,a_-)$, as the shift cancels out the shift from the cobordism maps. It also admits a $H_{2}(X_{n}(K))\cong\mathbb{Z}$ grading given by summing only over $a_{+}$ and $a_{-}$ such that $a_{+}-a_{-}=\alpha$. We denote the part in Maslov grading $i$, Alexander grading $j$ and homological grading $\alpha$ by $c\widehat{HFK}_i(K_n;\alpha,j)$. Hence, we can state the following theorem. Note in \cite{floerlasagna} it is proved for links but we will not need this.

\begin{thm}
\label{thm: cable and lasagna}
    Let $K$ be a knot in $S^3$. Then for all $n\in\mathbb{Z}$ the Floer lasagna module of $X_{n}(K)$ is isomorphic to $c\widehat{HFK}(K_{n})$ as a bigraded vector space, that is  the following isomorphism holds for all $i,j \in \mathbb{Z}$ and $\alpha \in H_{2}(X_{n}(K))$,
    \begin{equation}
        \mathcal{FL}_{i}(X_{n}(K);\alpha, j)\cong c\widehat{HFK}_{i}(K_{n};\alpha,j).
    \end{equation}
\end{thm}

\section{Spectral Sequences}
\label{sec: Spectral Sequences}

The main aim of this section is to examine how the maps associated to bands, quasi-stabilisations and pair of pants cobordisms on link Floer homology interact with the forgetting components spectral sequence. We split the section into four parts. In the first we recap some definitions and results about spectral sequences coming from filtered chain complexes and set up notation. The second subsection is then dedicated to showing that the identification of the $E^{\infty}$-page is canonical. Then the next subsection is dedicated to the statement and proof of how the maps interact with the spectral sequence, and the final subsection to giving an example.

\subsection{Notation for Spectral Sequences}
\label{subsec: Notation for Spectral Sequences}
\hfill

We start by recalling some definitions and results about spectral sequences. We use notation from \cite{KhovanovFloer}. However, note that all of the spectral sequences we consider will be of homological type and so the notation is slightly different. Further, we suppress the homological grading from the notation, so our chain complexes are really just vector spaces $C$ equipped with a map $\partial\colon C \to C$ such that $\partial^2=0$.

\begin{defn}
\label{defn: type of chain cx}
    Let $(C,\partial)$ be a chain complex with 
    \begin{equation*}
        C=\bigoplus_{i \in \mathbb{Z}}C_{i},
    \end{equation*}
    and $\partial=\partial_{0}+\partial_{1}+\ldots$ where $C_{i}=\{0\}$ for $|i|$ sufficiently large and $\partial_{i}(C_{j})\subseteq C_{i+j}$. Then we define a filtration by 
    \begin{equation*}
        \mathcal{F}_iC=\bigoplus_{j\geq i}C_{j}
    \end{equation*}
    and note $\mathcal{F}_iC\subseteq \mathcal{F}_{i-1}C$. Further, let $C$ and $C'$ be chain complexes as above. Then we define a \textit{filtered chain map of degree $k$} to be a chain map $f\colon C \to C'$ where $f=f_{k}+f_{k+1}+\ldots$, such that $f_{j}(C_{i})\subseteq C_{i+j}$.
\end{defn}

Further, recall that a filtered chain complex defines a spectral sequence, see for example \cite[Section 2.2]{SpecSeqBook}. We will use the notation $(E^p(C),\partial^p)$ for the $p$-th page of the spectral sequence, where 
\begin{equation*}
    E^0_i(C):=\mathcal{F}_i(C)/\mathcal{F}_{i+1}(C).
\end{equation*}
Further, when we say \textit{there exists a spectral sequence from $V$ to $W$} or write $V\implies W$, we mean there exists a $C$ as above such that 
\begin{equation*}
    V\cong E^1(C) \qquad \text{and} \qquad W\cong E^{\infty}(C).
\end{equation*}

We will also have cases when $C$ has an additional grading say $C^i$ such that $\partial(C^i)\subseteq C^i$ or $\partial(C^i)\subseteq C^{i-1}$. In these cases the pages of the spectral sequence and the homology inherit these additional gradings. In these cases we will say that \textit{the spectral sequence respects the grading} and will occasionally write $V^i\implies W^i$ to mean this. Note this notation is just as above for the case $\partial(C^i)\subseteq C^i$. However, it is a slight abuse of notation in the case $\partial(C^i)\subseteq C^{i-1}$ but if $[x]^{\infty}\in W^i$ then every element $x \in[x]^{\infty}$ is an element of $C^i$ justifying the notation.

Now, suppose that $f\colon C\to C'$ is a filtered chain map as above, then this defines a map of spectral sequences and we use the notation, 
\begin{equation*}
    E^{p}(f)\colon E^{p}(C)\to E^p(C'),
\end{equation*}
for the induced map on the $p$-th page. We also denote by 
\begin{equation*}
    f_*\colon H_*(C)\to H_*(C'),
\end{equation*}
the induced map on homology. To prove the main results of this section we will need to compare $E^{\infty}(f)$, $f_*$ and $E^1(f)$. To do this we first recall that there is a natural filtration on $H_{*}(C)$. This is a standard result and uses the fact that $\mathcal{F}_i(C)$ is also a chain complex.

\begin{lem}
    Let $C$ be a chain complex as above and for all $i$ let $\iota_i\colon \mathcal{F}_i(C)\to C$ be the inclusion. Then there is a filtration on $H_*(C)$ given by,
    \begin{equation*}
        \mathcal{F}_{i}H_{*}(C)=\text{im}({\iota_i}_*\colon H_{*}(\mathcal{F}_i(C))\to H_*(C)),
    \end{equation*}
    where $\mathcal{F}_{i}H_{*}(C)\subseteq \mathcal{F}_{i-1}H_{*}(C)$.
\end{lem}

Using this recall that the \textit{associated graded vector space} $\text{gr}(H_*(C))$ is defined as,
\begin{equation*}
    \text{gr}(H_*(C)):=\bigoplus_{i\in Z}\mathcal{F}_{i}H_{*}(C)/\mathcal{F}_{i+1}H_{*}(C).
\end{equation*}
We denote by $\text{gr}_i(H_*(C)):=\mathcal{F}_{i}H_{*}(C)/\mathcal{F}_{i+1}H_{*}(C)$. Using this notation we have the following.

\begin{lem}
    Let $f\colon C \to C'$ be a filtered chain map of degree $k$. Then $f_*$ is a filtered map of degree $k$ meaning,
    \begin{equation*}
        f_*(\mathcal{F}_{i}H_{*}(C))\subseteq \mathcal{F}_{i+k}H_{*}(C').
    \end{equation*}
    Further, the associated graded map $\textup{gr}f_*\colon \textup{gr}(H_*(C))\to \textup{gr}(H_*(C'))$ is a graded map that shifts the grading by $k$, that is
    \begin{equation*}
        \textup{gr}f_*(\textup{gr}_i(H_*(C)))\subseteq \textup{gr}_{i+k}(H_*(C')).
    \end{equation*}
\end{lem}

Now recall that the $E^{\infty}$ page is canonically isomorphic to $\text{gr}(H_*(C))$ as a graded vector space. That is for all $i$ there is a canonical isomorphism,
\begin{equation*}
    \sigma_i\colon E_i^{\infty}(C)\to \text{gr}_i(H_*(C)).
\end{equation*}
This allows us to connect $E^{\infty}(f)$ and $f_*$.

\begin{lem}
\label{lem: f_* and E^infty (f)}
    Let $f$ be as in the previous lemma. Then, we have the following equality,
    \begin{equation*}
        E^{\infty}(f)={\sigma'}^{-1}\circ \textup{gr}f_*\circ \sigma.
    \end{equation*}
\end{lem}

We end this subsection with a useful observation about filtered vector spaces that will be useful in many cases to compare the $E^\infty$-page and $H_*(C)$.

\begin{lem}
\label{lem: filtration and associated grading in single degree}
    Let $V$ be a filtered vector space and $\textup{gr}(V)$ be the associated graded vector space. If $\textup{gr}(V)=\textup{gr}_iV$ for some $i$ then there is a canonical isomorphism $V \cong \textup{gr}V$.
\end{lem}

\begin{proof}
    By assumption we have that for all $j\neq i$,
    \begin{equation*}
        \text{gr}_j(V)=\mathcal{F}_{j}V/\mathcal{F}_{j+1}V=\{0\},
    \end{equation*}
    which implies $\mathcal{F}_{j}V=\mathcal{F}_{j+1}V$. Further, we have,
    \begin{equation*}
        \text{gr}(V)=\text{gr}_i(V)=\mathcal{F}_{i}V/\mathcal{F}_{i+1}V.
    \end{equation*}
    Together, these imply that $\mathcal{F}_{i+1}V=\{0\}$. Now this gives the desired isomorphism as $\mathcal{F}_iV=V$.
\end{proof}

\subsection{Naturality and the Spectral Sequence for Forgetting Components}
\label{subsec: boundary maps}
\hfill

We start this subsection by recalling the spectral sequence on link Floer homology of Ozsvath and Szabo in \cite{OSLinkFloer} for forgetting a component of the link. We primarily use the  formulation of this spectral sequence from \cite{LemmasPaper}. Recall from Lemma \ref{lem:  unknot component} that $V$ is the two dimensional vector space generated by $B$ and $T$, where $A(T)=A(B)=0$, $M(T)=0$ and $M(B)=-1$.

\begin{thm}
\label{thm: OS Spec seq}
    Let $\mathcal{H}=(\Sigma,\boldsymbol{\alpha},\boldsymbol{\beta},\mathbf{w},\mathbf{z})$ be a Heegaard diagram for the $n$-component multipointed link $\mathbb{L}=(L,\mathbf{w},\mathbf{z})$. Further, let $L_0\subseteq L$ be a component of $L$ with exactly two basepoints, and $z_0$ be the $\mathbf{z}$ basepoint on $L_0$. Then the map $\partial_{z_{0}}\colon \widehat{CFL}(\mathcal{H})\to \widehat{CFL}(\mathcal{H})$ defined by,
    \begin{equation}
    \label{eq: spectral sequence boundary map}
        \partial_{z_{0}}(\mathbf{x})=\sum_{\mathbf{y}\in \mathbb{T}_{\boldsymbol{\alpha}}\cap\mathbb{T}_{\boldsymbol{\beta}}} \sum_{\substack{\phi \in \pi_{2}(\mathbf{x},\mathbf{y}) \\ \mu(\phi)=1 \\ n_{\mathbf{w}}(\phi)=n_{\mathbf{z}\backslash\{z_{0}\}}(\phi)=0}}\#\widehat{\mathcal{M}(\phi)} \mathbf{y},
    \end{equation}
    is such that $\partial_{z_{0}}^2=0$ and 
    \begin{equation*}
        H_{*}\left(\widehat{CFL}(\mathcal{H}),\partial_{z_{0}}\right)\cong \widehat{HFL}(\mathbb{L}\backslash \mathbb{L}_0)\otimes V.
    \end{equation*}
    Therefore, there is a spectral sequence from $\widehat{HFL}(\mathbb{L})$ to $\widehat{HFL}(\mathbb{L}\backslash \mathbb{L}_0)\otimes V$. Finally this spectral sequence respects the Maslov grading and shifts the Alexander multigrading $A_{L_i}$ down by $l_i:=\frac{1}{2}lk(L_i,L_0)$, meaning there is a spectral sequence,
    \begin{align}
        \bigoplus_{x\in \mathbb{Z}}\widehat{HFL}_i(\mathbb{L},(j_1,\ldots,x,\ldots,j_{n})) \implies&  \left(\widehat{HFL}_{i+1}(\mathbb{L}\backslash \mathbb{L}_0,(j_1+l_1,\ldots,j_{n}+l_n))\otimes B\right) \nonumber\\
        & \oplus \left(\widehat{HFL}_i(\mathbb{L}\backslash \mathbb{L}_0,(j_1+l_1,\ldots,j_{n}+l_n))\otimes T\right)
    \end{align}
    where the $x$ is the Alexander grading corresponding to $L_0$.
\end{thm}

There is a more general version of the above for forgetting multiple components but we will not need this. Further, note that this fits into the chain complexes in Definition \ref{defn: type of chain cx} as we can consider 
\begin{equation*}
    \widehat{CFL}(\mathcal{H})=\bigoplus_{j \in \mathbb{Z}}\widehat{CFL}(\mathcal{H},-j),
\end{equation*}
where $j$ is the component of the Alexander multigrading $A_{L_0}$, and we take $-j$ so that $\partial_{z_{0}}$ increases the grading. Note that this is the case as any holomorphic disc $u\in \pi_2(\mathbf{x},\mathbf{y})$ necessarily intersects $V_{z_0}$ positively and does not intersect $V_{w_0}$. Now, 
\begin{equation*}
    A_{0}(\mathbf{x})-A_{0}(\mathbf{y})=n_{z_{0}}(u)-n_{w_{0}}(u),    
\end{equation*}
so any holomorphic disc decreases $A_0$. Therefore,
\begin{equation*}
    \partial_{z_0}=(\partial_{z_0})_0+ (\partial_{z_0})_1 +\ldots
\end{equation*}
as in Definition \ref{defn: type of chain cx}. Note further that the filtration respects the Maslov grading and the other Alexander gradings, so we can split the complex as a direct sum over these.

We now prove that $\{(\widehat{CFL}(\mathcal{H}),\partial_{z_{0}})\}_{\mathcal{H}}$ is a transitive system of chain complexes. 

\begin{lem}
    Let $\mathcal{H}$ and $\mathcal{H}'$ be Heegaard diagrams for the $n$-component link $\mathbb{L}$ and let $z_{0}$ and $z_{0}'$ be the respective basepoints on the component $L_0$. Then there exists a chain homotopy equivalence, 
    \begin{equation*}
        \Phi_{\mathcal{H}\to\mathcal{H'}}\colon (\widehat{CFL}(\mathcal{H}),\partial_{z_{0}})\to (\widehat{CFL}(\mathcal{H}'),\partial_{z_{0}'}),
    \end{equation*}
    that preserves the Maslov grading, Alexander multigrading $A_i$ for $i\neq 0$ and is filtered in the grading $A_0$. Further it is well defined up to chain homotopy equivalence. That is there is a transitive system of chain complexes $\{(\widehat{CFL}(\mathcal{H}),\partial_{z_{0}})\}_{\mathcal{H}}$ indexed over all Heegaard diagrams for $\mathbb{L}$.
\end{lem}

\begin{proof}
    First, recall that in \cite{CobMapPaper} it is proven that $\{CFL^{-}(\mathcal{H})\}_{\mathcal{H}}$ forms a transitive system of chain complexes over the ring $\mathbb{F}[U_0,U_1,\ldots U_{n-1}, V_{0},V_{1},\ldots, V_{n-1}]$. Next, note that $\partial_{z_{0}}$ is obtained from $\partial^-$ by setting $V_{0}=1$ and the rest of the variables to be $0$. Also, if 
    \begin{equation*}
        \Phi_{\mathcal{H}\to \mathcal{H}'}^{-}\colon CFL^-(\mathcal{H})\to CFL^{-}(\mathcal{H'}),
    \end{equation*}
    is a map in this transitive system, then we obtain a map, 
    \begin{equation*}
        \Phi_{\mathcal{H}\to\mathcal{H'}}\colon (\widehat{CFL}(\mathcal{H}),\partial_{z_{0}})\to (\widehat{CFL}(\mathcal{H}'),\partial_{z_{0}'}),
    \end{equation*}
    by setting the variables as before. Finally, this map is clearly a chain homotopy equivalence as $\Phi_{\mathcal{H}\to \mathcal{H}'}^-$ is, which completes the proof. 
    
    The claims about the Maslov grading and Alexander gradings for $A_i$, $i\geq0$ follow from the fact that $\Phi_{\mathcal{H}\to\mathcal{H'}}^-$ preserves the Maslov grading and the Alexander filtration in these degrees. Hence, setting the variables $U_i$, for all $i$, and $V_i$, for $i>0$, to 0 gives the claim that it respects the Alexander gradings $A_i$ for $i>0$. Setting the variable $V_0=1$ gives the claim about being filtered with respect to $A_0$.
\end{proof}

We will denote the transitive system by $\widehat{CFL}_{z_0}(\mathbb{L})$ and by $\widehat{HFL}_{z_0}(\mathbb{L})$ the canonical graded vector space coming from the homology of this transitive system. Next we define a similar result for holomorphic triangle counts on Heegaard triples. This is effectively a consequence of Definition \ref{defn: triangle count}.  Recall, from the discussion after Definition \ref{defn: subordinate to B}, that if $\mathcal{T}$ is a Heegaard triple then $\mathcal{T}_{\alpha,\beta}$ is the Heegaard diagram given by only considering the $\boldsymbol{\alpha}$ and $\boldsymbol{\beta}$ curves, $\mathcal{T}_{\alpha',\beta}$ and $\mathcal{T}_{\alpha',\alpha}$ are defined similarly.  

\begin{lem}
\label{lem: forgetting z0 triangle count is a chain map}
    Let $\mathcal{T}=(\Sigma,\boldsymbol{\alpha'},\boldsymbol{\alpha},\boldsymbol{\beta},\mathbf{w},\mathbf{z})$ be a Heegaard triple and $z_0\in \mathbf{z}$. Then the map,
    \begin{equation*}
        F_{\mathcal{T},z_{0}}\colon \widehat{CFL}(\mathcal{T}_{\alpha',\alpha})\otimes\widehat{CFL}(\mathcal{T}_{\alpha,\beta})\to \widehat{CFL}(\mathcal{T}_{\alpha',\beta})
    \end{equation*}
    defined for $\boldsymbol{\theta}\in \mathbb{T}_{\boldsymbol{\alpha}'}\cap \mathbb{T}_{\boldsymbol{\alpha}} $ and $ \mathbf{x}\in \mathbb{T}_{\boldsymbol{\alpha}}\cap \mathbb{T}_{\boldsymbol{\beta}}$ by,
    \begin{equation}
        F_{\mathcal{T},z_{0}}(\boldsymbol{\theta}\otimes \mathbf{x})= \sum_{\mathbf{y}\in \mathbb{T}_{\boldsymbol{\alpha'}}\cap\mathbb{T}_{\boldsymbol{\beta}}} \sum_{\substack{\psi \in \pi_{2}(\boldsymbol{\theta},\mathbf{x},\mathbf{y}) \\ \mu(\psi)=0 \\ n_{\mathbf{w}}(\psi)=n_{\mathbf{z}\backslash\{z_{0}\}}(\psi)=0}}\#\mathcal{M}(\psi) \mathbf{y},
    \end{equation}
    is a well-defined chain map. Moreover, for a fixed $\boldsymbol{\theta}$ that is a cycle, $F_{\mathcal{T},z_0}(\boldsymbol{\theta}\otimes -)$ induces a morphism of transitive systems.
\end{lem}

\begin{proof}
     The fact that $F_{\mathcal{T},z_{0}}$ is well defined follows from the fact that the $\mathbb{F}[U_0,V_0,\ldots U_n,V_n]$ equivariant map,
     \begin{equation*}
         F_{\mathcal{T}}^-\colon CFL^{-}(\mathcal{T_{\alpha',\alpha}})\otimes CFL^{-}(\mathcal{T_{\alpha,\beta}})\to CFL^-(\mathcal{T}_{\alpha',\beta}),
     \end{equation*}
     is well defined by \cite{CobMapPaper} and that $F_{\mathcal{T},z_{0}}$ is obtained from this by setting $V_{0}=1$ and all the other variables to be $0$. Similarly, recall from the previous proof that $\partial_{z_{0}}$ is obtained from $\partial^-$ by setting the variables as before. It then follows that $F_{\mathcal{T},z_0}$ is a chain map as $F_{\mathcal{T}}^-$ is a chain map. Further, for a fixed cycle $\boldsymbol{\theta}$, $F_{\mathcal{T}}^-(\boldsymbol{\theta}\otimes-)$ is part of a morphism of transitive systems so it follows as before that $F_{\mathcal{T},z_{0}}(\boldsymbol{\theta}\otimes -)$ is too.
\end{proof}

Now, our aim is to define a canonical isomorphism,
\begin{equation*}
    g\colon\widehat{HFL}_{z_0}(\mathbb{L})\to \widehat{HFL}(\mathbb{L}\backslash\mathbb{L}_0)\otimes V.
\end{equation*}
To this end we recall the notion of free stabilisations. This will also be important for understanding the various cobordisms under the spectral sequence.

\begin{figure}[h]
    \centering
    \def\svgwidth{0.6\textwidth}
    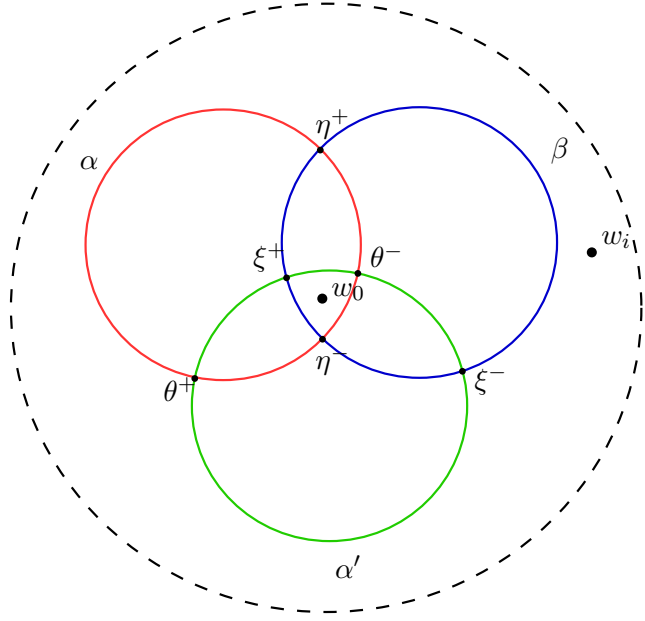
    \caption{Diagram for a free stabilisation of a Heegaard triple in the neighbourhood of $w_i$.}
    \label{fig: Heegaard Triple Triangle example}
\end{figure}

\begin{defn}
    Let $\mathcal{H}=(\Sigma,\boldsymbol{\alpha},\boldsymbol{\beta},\mathbf{w},\mathbf{z})$ be a Heegaard diagram for a pointed link $\mathbb{L}$. Then a \textit{free stabilisation of $\mathcal{H}$} is a Heegaard diagram,
    \begin{equation*}
        \mathcal{H}'=(\Sigma,\boldsymbol{\alpha}\cup\{\alpha_0\},\boldsymbol{\beta}\cup\{\beta_0\},\mathbf{w}\cup \{w_0\},\mathbf{z})
    \end{equation*}
    where $\alpha_0$ and $\beta_0$ are two circles that intersect twice, do not intersect any other curves and both bound discs containing $w_0$. See Figure \ref{fig: Heegaard Triple Triangle example} ignoring the curve $\alpha'$ for an example.
\end{defn}

Note that after a free stabilisation there are two new intersection points, which are distinguished by Maslov grading. Denote by $\theta^{+}$ the point with higher Maslov grading and $\theta^-$ the lower. Then there is a clear canonical isomorphism at the level of vector spaces,
\begin{equation*}
    \phi_{\text{stab}}\colon \widehat{CFL}(\mathcal{H}')\to \widehat{CFL}(\mathcal{H})\otimes \mathbb{F}^2,
\end{equation*}
given by $\phi_{\text{stab}}(\mathbf{x}\times\theta^{\pm})=\mathbf{x}\otimes \theta^{\pm}$. As part of the next lemma it follows that there is a graded isomorphism $\mathbb{F}\langle \theta^+,\theta^-\rangle\cong V$. The following lemma is standard and a proof appears as \cite[Proposition 7.1]{OSLinkFloer}, although we use the phrasing from \cite{LinkSurgeryFormulaOG}.

\begin{prop}
\label{prop: free stabilisation}
    Let $\mathcal{H}$ be a Heegaard diagram for a link. Then let $\mathcal{H}'$ be the result of a free stabilisation introducing a basepoint $w_0$ in the neighbourhood of some basepoint $w_i$. Then,
    \begin{equation*}
        \phi_{\text{stab}}\colon\widehat{CFL}(\mathcal{H}')\to \widehat{CFL}(\mathcal{H})\otimes V,
    \end{equation*}
    is an isomorphism of multigraded chain complexes.
\end{prop}

The main part of the proof from \cite[Proposition 6.5]{OSLinkFloer} shows that $CFL^{-}(\mathcal{H}')$ is isomorphic to the mapping cone,
\begin{equation*}
    CFL^{-}(\mathcal{H})[U_0]\xrightarrow{U_{i}+U_{0}}CFL^{-}(\mathcal{H})[U_0],
\end{equation*}
where,
\begin{equation*}
    CFL^{-}(\mathcal{H})[U_0]=CFL^{-}(\mathcal{H})\otimes_{\mathbb{F}[U_1,V_1,\ldots,U_n,V_n]}\mathbb{F}[U_{0},U_{1},V_1,\ldots,U_n,V_n].
\end{equation*}
Note that this makes the differential,
\begin{equation*}
    \partial_{\mathcal{H}'}(\mathbf{x}\times \theta^+)=\partial_{\mathcal{H}}(\mathbf{x})\otimes \theta^+ \qquad \text{and} \qquad \partial_{\mathcal{H}'}(\mathbf{x}\times \theta^-)=\partial_{\mathcal{H}}(\mathbf{x})\otimes \theta^-+(U_i+U_0)\mathbf{x}\otimes\theta^{+}.
\end{equation*}
Then setting all the $U_i$ and $V_i$ to $0$ we get the splitting of the differential as desired. The gradings then have to be checked, see the proof of \cite[Proposition 7.1]{OSLinkFloer} or the proof of \cite[Proposition 5.5]{GraphCobZemke} for details. We now define stabilisation for Heegaard triples.

\begin{defn}
\label{defn: triangle stabilisation}
    Let $\mathcal{T}=(\Sigma,\boldsymbol{\alpha}',\boldsymbol{\alpha},\boldsymbol{\beta},\mathbf{w},\mathbf{z})$ be a Heegaard triple. Then a \textit{free stabilisation of $\mathcal{T}$} is a Heegaard triple,
    \begin{equation*}
        \mathcal{T}'=(\Sigma,\boldsymbol{\alpha}'\cup\{\alpha_0'\},\boldsymbol{\alpha}\cup\{\alpha_0\},\boldsymbol{\beta}\cup\{\beta_0\},\mathbf{w}\cup \{w_0\},\mathbf{z})
    \end{equation*}
    where $\alpha_0',\alpha_0$ and $\beta_0$ are curves that pairwise intersect twice, do not intersect any other curves and each bound a disc containing $w_0$, see Figure \ref{fig: Heegaard Triple Triangle example}.
\end{defn}

Denote by $\theta^{\pm},\eta^{\pm}$ and $\xi^{\pm}$ the points in $\alpha'_0\cap \alpha_0$, $\alpha_0\cap \beta_0$ and $\alpha'_0\cap \beta_0$ respectively. Then as before there are multigraded isomorphisms,
\begin{equation*}
    \phi_{\text{stab}}\colon \widehat{CFL}(\mathcal{T}'_{\gamma,\delta})\to \widehat{CFL}(\mathcal{T}_{\gamma,\delta})\otimes V,
\end{equation*}
for each pair $\gamma,\delta\in \{\alpha',\alpha,\beta\}$. Further, for the next lemma we also need the following maps,
\begin{equation*}
    \phi_{\text{stab}}^{\pm}\colon \widehat{CFL}(\mathcal{T}_{\gamma,\delta})\to \widehat{CFL}(\mathcal{T}'_{\gamma,\delta}),
\end{equation*}
where $\phi_{\text{stab}}^{\pm}(\mathbf{x})=\mathbf{x}\otimes x^{\pm}$ for the appropriate $x \in \{\theta,\eta,\xi\}$. We now prove the following lemma which is a specialisation of \cite[Theorem 5.7]{GraphCobZemke}.

\begin{prop}
\label{prop: triple free stabilisation}
    Let $\mathcal{T}$ be a Heegaard triple diagram. Then let $\mathcal{T}'$ be the result of a free stabilisation introducing a basepoint $w_0$ in the neighbourhood of some basepoint $w_i$. Then,
    \begin{equation*}
        F_{\mathcal{T}'}(\phi_{\text{stab}}^{+}(\boldsymbol{\theta}),\phi_{\text{stab}}^{\pm}(\mathbf{x}))=\phi_{\text{stab}}^{\pm}(F_{\mathcal{T}}(\boldsymbol{\theta},\mathbf{x})).
    \end{equation*}
\end{prop}

\begin{proof}
    In \cite[Theorem 5.7]{GraphCobZemke} it is proven for triples without $\mathbf{z}$ basepoints that the map $F^{-}_{\mathcal{T}}$ satisfies,
    \begin{align*}
        F_{\mathcal{T}'}^{-}(\boldsymbol{\theta}\times \theta^+\otimes \mathbf{x}\times \eta^+)&=F_{\mathcal{T}}^-(\boldsymbol{\theta}\otimes \mathbf{x})\otimes \xi^{+}, \\
        F_{\mathcal{T}'}^{-}(\boldsymbol{\theta}\times \theta^+\otimes \mathbf{x}\times \eta^-)&=F_{\mathcal{T}}^-(\boldsymbol{\theta}\otimes \mathbf{x})\otimes \xi^{-}+\sum_{\mathbf{y}\in \mathbb{T}_{\boldsymbol{\alpha}'}\cap \mathbb{T}_{\boldsymbol{\beta}} } C_{\boldsymbol{\theta},\mathbf{x},\mathbf{y}} \mathbf{y}\otimes \xi^{+},
    \end{align*}
    for $C_{\boldsymbol{\theta},\mathbf{x},\mathbf{y}} \in \mathbb{F}[U_0,\ldots,U_n]$. The proof clearly adapts to this case as it is just keeping track of when the holomorphic triangles cross the $\mathbf{z}$ basepoints. Now, if we have that $C_{\boldsymbol{\theta},\mathbf{x},\mathbf{y}}$ has no constant term for all $\mathbf{y}$, then by setting $U_i=V_i=0$ for all $i$ we get the above result.

    To see that $C_{\boldsymbol{\theta},\mathbf{x},\mathbf{y}}$ has no constant term we argue that any holomorphic triangle giving such a term must have $n_{w_0}(\psi')+n_{w_i}(\psi')>0$. To see this write $\psi'=\psi\#\psi_0$ where,
    \begin{equation*}
        \psi\in \pi_2(\boldsymbol{\theta},\mathbf{x},\mathbf{\mathbf{y}}) \qquad \text{and} \qquad \psi_{0}\in \pi_2(\theta^+,\eta^-,\xi^+).
    \end{equation*}
    The second triangle is living in the Heegaard triple given by collapsing the dotted circle in Figure \ref{fig: Heegaard Triple Triangle example} to a point and the first in $\mathcal{T}$. It is shown in the proof of \cite[Theorem 5.7]{GraphCobZemke} that,
    \begin{equation*}
        \mu(\psi')=\mu(\psi)+\mu(\psi_0)-2n_{w_i}(\psi_0)=\mu(\psi)-1+2n_{w_0}(\psi_0).
    \end{equation*}
    Now, $\mu(\psi)$ and $n_{w_0}(\psi_0)$ are non-negative. We claim that as $\mu(\psi')=0$ we must have $\mu(\psi)=1$ and $n_{w_0}(\psi_{0})=0$. Otherwise, if $n_{w_0}(\psi_0)\geq 1$ then,
    \begin{equation*}
        \mu(\psi)=1-2n_{w_{0}}(\psi_0)\leq-1,
    \end{equation*}
    but this contradicts non-negativity. Similarly, if $\mu(\psi)=0$ then, $n_{w_0}(\psi_0)=\frac{1}{2}\notin\mathbb{Z}$ and if $\mu(\psi)\geq2$ we get a contradiction of non-negativity of $n_{w_0}(\psi_0)$. Therefore, as $\mu(\psi_0)\geq 0$ we get that $n_{w_i}(\psi_0)>0$. Therefore, $n_{w_i}(\psi')\geq 1$.
\end{proof}

With this in hand we are almost ready to state the canonical isomorphism statement we want. However, before stating this we frame forgetting the $z_{0}$ basepoint in terms of the effect on the sutured manifold $(S^3\backslash N(L),\gamma)$, where the sutures correspond to curves around the basepoints. This will be useful in the proofs to come.

\begin{rem}
\label{rem: sutured stuff}
    Let $\mathcal{H}=(\Sigma,\boldsymbol{\alpha},\boldsymbol{\beta},\mathbf{w},\mathbf{z})$ be a Heegaard diagram for the multipointed link $\mathbb{L}=(L,\mathbf{w},\mathbf{z})$ and suppose $L_0$ is a component of $L$ with only two basepoints. Recall from \cite{OGSuturedPaper} that a Heegaard diagram for a link also defines a sutured Heegaard diagram for $(S^3\backslash N(L),\gamma)$ where the sutures $\gamma$ are given by meridians on $N(L)$ around the basepoints. Now consider attaching a 3-dimensional 2-handle along the suture corresponding to $z_0$. This replaces the torus component of the boundary of $S^3\backslash N(L)$ corresponding to $L_0$ with a 2-sphere with a single suture. Now consider a sutured Heegaard diagram, $\tilde{\mathcal{H}}$, for $S^3\backslash N(L\backslash L_0)$ with sutures at basepoints, and with a free stabilisation in a small disc near $w_i$. We can then see that $\tilde{\mathcal{H}}$ defines a sutured diagram for,
    \begin{equation*}
        (Y,\tilde \gamma)=(S^3\backslash (N(L\backslash L_0)\cup B^3),\tilde{\gamma}),
    \end{equation*}
     where there is one suture on the $S^2$ and sutures on meridians at each basepoint on the boundary of the link. Therefore, $\mathcal{H}_0=(\Sigma,\boldsymbol{\alpha},\boldsymbol{\beta},\mathbf{w},\mathbf{z}\backslash\{z_0\})$ and $\tilde{\mathcal{H}}$ represent $(Y,\tilde{\gamma})$, and so by \cite[Proposition 2.15]{OGSuturedPaper} there exists a sequence of Heegaard moves taking one to the other. 
\end{rem}

\begin{prop}
\label{prop: canonical identification}
    Let $\mathbb{L}$ be a link and $L_0$ be a component with exactly two basepoints. Then, there is a canonical isomorphism,
    \begin{equation*}
        g\colon\widehat{HFL}_{z_0}(\mathbb{L})\to \widehat{HFL}(\mathbb{L}\backslash\mathbb{L}_0)\otimes V.
    \end{equation*}
\end{prop}

\begin{proof}
    Let $\mathcal{G}$ be a Heegaard diagram for the link $\mathbb{L}\backslash\mathbb{L}_0$. Now perform a free stabilisation to get $\mathcal{G}'$. Then this Heegaard diagram can be seen as a sutured diagram for,
    \begin{equation*}
        (Y,\tilde \gamma)=(S^3\backslash (N(L\backslash L_0)\cup B^3),\tilde{\gamma}),
    \end{equation*}
    where there is one suture on the $S^2$ and sutures on meridians at each basepoint on the boundary of the link. From Remark \ref{rem: sutured stuff} this is also the sutured manifold obtained by forgetting the basepoint $z_0$ from a Heegaard diagram $\mathcal{H}$ for $\mathbb{L}$. Let $\tilde{ \mathcal{H}}$ be another Heegaard diagram for $\mathbb{L}$. Further, let $\tilde{\mathcal{G}}$ be another Heegaard diagram for $\mathbb{L}\backslash\mathbb{L}_0$ and $\tilde{\mathcal{G}}'$ be a free stabilisation of it. Then by naturality from \cite{Naturality} we have a commutative diagram,
    \begin{equation}
        \begin{tikzcd}
            (\widehat{CFL}(\mathcal{H}),\partial_{z_0})\arrow[rr,"\Phi_{\mathcal{H}\to \tilde{\mathcal{H}}}"]\arrow[dd,"\Phi_{\mathcal{H}\to \mathcal{G}'}"] && (\widehat{CFL}(\tilde{\mathcal{H}}),\partial_{z_0})\arrow[dd,"\Phi_{\tilde{\mathcal{H}}\to \tilde{\mathcal{G}}'}"]\\
            && \\
            \widehat{CFL}(\mathcal{G}') \arrow[rr,"\Phi_{\mathcal{G}'\to \tilde{\mathcal{G}}'}"]& &\widehat{CFL}(\tilde{\mathcal{G}}')
        \end{tikzcd}
    \end{equation}
    where all of the maps are chain homotopy equivalences. Note that $\{\widehat{CFL}(\mathcal{G}')\}_{\mathcal{G}}$ forms a transitive system as it is a subset of all Heegaard diagrams that are equivalent to $\mathcal{G}'$. So there is a canonical isomorphism from $\widehat{HFL}_{z_0}(\mathbb{L})$ to the canonical vector space $H$ associated to the homology of this transitive system.

    Define the map $\varphi^{\pm}\colon \widehat{CFL}(\mathcal{G}')\to \widehat{CFL}(\mathcal{G})$ by,
    \begin{equation*}
        \varphi^\pm(\mathbf{x}\times \theta^{\pm})=\mathbf{x} \qquad \text{and} \qquad \varphi^{\pm}(\mathbf{x}\times \theta^{\mp})=0.
    \end{equation*}
    It was shown in \cite{GraphCobZemke} that $\varphi_{\text{stab}}^+$ defines a morphism of transitive systems. As we are working with the complexes $\widehat{CFL}$  the proof can be adapted, by using an argument similar to the proof of Proposition \ref{prop: triple free stabilisation}, to show that $\varphi_{\text{stab}}^-$ is also a morphism of transitive systems. Note that this adaptation does not work for the case of the complexes $CF^{-}$. Now note that,
    \begin{equation*}
        \phi_{\text{stab}}(\mathbf{x}\times \theta^{\pm})=\varphi^{+}(\mathbf{x}\times \theta^{\pm})\otimes \theta^++\varphi^{-}(\mathbf{x}\times \theta^{\pm})\otimes \theta^-.
    \end{equation*}
    Therefore, $\phi_{\text{stab}}$ is an isomorphism of transitive systems. That is we have the following commutative diagram,
    \begin{equation}
        \begin{tikzcd}
            \widehat{CFL}(\mathcal{G}') \arrow[rr,"\Phi_{\mathcal{G}'\to \tilde{\mathcal{G}}'}"]\arrow[dd,"\phi_{\text{stab}}"]& &\widehat{CFL}(\tilde{\mathcal{G}}')\arrow[dd,"\phi_{\text{stab}}"] \\
            && \\
            \widehat{CFL}(\mathcal{G})\otimes V \arrow[rr,"\Phi_{\mathcal{G}\to \tilde{\mathcal{G}}}\otimes \text{Id}_V"]& &\widehat{CFL}(\tilde{\mathcal{G}})\otimes V
        \end{tikzcd}
    \end{equation}
    The canonical vector space associated to the homology of the transitive system of the bottom row is $\widehat{HFL}(\mathbb{L}\backslash\mathbb{L}_0)\otimes V$. Therefore, composing the isomorphisms we get the desired canonical isomorphism.
\end{proof}

\subsection{Forgetting a Cylinder Component of Certain Link Cobordisms}
\label{subsec: Forgetting Components of Band Maps}
\hfill

Using the previous subsections we can now explain how the spectral sequences interact with cobordism maps. There are two parts to each of the claims. First we claim that there is a well-defined map when we ignore the $z_0$-basepoint. Then we identify the map this induces on homology with the cobordisms after forgetting a cylinder under the isomorphisms in Proposition \ref{prop: canonical identification}. We do  this for the map associated to an $\alpha$-band, a quasi-stabilisation and a pair of pants cobordism. We start with the band map.

\begin{lem}
\label{lem: triangle map is well defined and a filtred map}
    Let $\mathcal{T}=(\Sigma,\boldsymbol{\alpha'},\boldsymbol{\alpha},\boldsymbol{\beta},\mathbf{w}\cup\{w_{0}\},\mathbf{z}\cup\{z_{0}\})$ be a Heegaard triple subordinate to an $\alpha$-band $B$, and suppose $w_{0},z_{0}$ are the only basepoints on the component $L_{0}$, which is disjoint from the band. Then there is a unique maximal Maslov graded element $\boldsymbol{\Theta}^{\mathbf{w}}\in \widehat{HFL}(\mathcal{T}_{\alpha',\alpha})$ and the map,
    \begin{equation*}
        F_{\mathcal{T},z_{0}}(\boldsymbol{\theta}^{\mathbf{w}}\otimes -)\colon \widehat{CFL}(\mathcal{T}_{\alpha,\beta})\to \widehat{CFL}(\mathcal{T}_{\alpha',\beta})
    \end{equation*}
    where $\boldsymbol{\theta}^{\mathbf{w}}$ is any representative of $\boldsymbol{\Theta}^{\mathbf{w}}$, is a degree $0$ filtered chain map with respect to the $\partial_{z_{0}}$ differential, is well-defined up to chain homotopy and induces a morphism of transitive systems.
\end{lem}

\begin{proof}
    To start we claim that the complexes $(\widehat{CFL}(\mathcal{T}_{\alpha',\alpha}),\partial_{z_0})$ and $\widehat{CFL}(\mathcal{T}_{\alpha',\alpha})$ are isomorphic. To see this note that by \cite[Lemma 6.4]{CobMapPaper} $\mathcal{T}_{\alpha',\alpha}$ is a Heegaard diagram for an unlink in $(S^1\times S^2)^{\#m}$ for some $m$. It follows that $w_0$ and $z_0$ lie in the same component of $\Sigma\backslash(\boldsymbol{\alpha'}\cup \boldsymbol{\alpha})$. Therefore, for any holomorphic disc $u$ we have $n_{w_0}(u)=n_{z_0}(u)$. Thus, by the definition of the differentials we have the isomorphism. Further, from this it follows that there is a unique element of maximum grading $\boldsymbol{\Theta}^{\mathbf{w}}$. Finally, independence of the choice of representative $\boldsymbol{\theta}$ then also follows from \cite[Lemma 6.5]{CobMapPaper}.
    
    Note that by Lemma \ref{lem: forgetting z0 triangle count is a chain map} we know that $F_{\mathcal{T},z_0}(\boldsymbol{\theta}^{\mathbf{w}}\otimes -)$ is a chain map. It remains to show that it is a degree 0 filtered map. Recall for the complex $\mathcal{CFL}^{\infty}$ the variable $V_{z_i}$ has grading $+1$ in the $A_i$ Alexander grading. The statement of Theorem \ref{thm: grading change formulas} in \cite{CobGradings} is for this complex. Let $\mathbf{x}\in \mathbb{T}_{\boldsymbol{\alpha}} \cap \mathbb{T}_{\boldsymbol{\beta}}$ then the map in this complex is,
    \begin{equation*}
        F_{\mathcal{T}}^{\infty}(\boldsymbol{\theta}^{\mathbf{w}}\otimes \mathbf{x})= \sum_{\mathbf{y}\in \mathbb{T}_{\boldsymbol{\alpha'}}\cap\mathbb{T}_{\boldsymbol{\beta}}} \sum_{\substack{\psi \in \pi_{2}(\boldsymbol{\theta}^{\mathbf{w}},\mathbf{x},\mathbf{y}) \\ \mu(\psi)=0 \\ }}\#\mathcal{M}(\psi) U_{\mathbf{w}}^{\mathbf{n_{\mathbf{w}}}(\psi)}V_{\mathbf{z}}^{\mathbf{n_{\mathbf{z}}}(\psi)} \mathbf{y},
    \end{equation*}
    where $U_{\mathbf{w}}^{\mathbf{n_{\mathbf{w}}}(\psi)}=U_0^{n_{w_0}(\psi)}\cdots U_n^{n_{w_n}(\psi)}$ and $V_{\mathbf{z}}^{\mathbf{n_{\mathbf{z}}}(\psi)}$ is defined similarly. Now set $U_i=0$ for all $i$, $V_{i}=0$ for $i>0$ and $V_0=1$. Suppose that the coefficient of $\mathbf{y}$ is non-zero in $F_{\mathcal{T},z_{0}}(\boldsymbol{\theta}^{\mathbf{w}}\otimes \mathbf{x})$. Then $V_{z_0}^{n_{z_0}(\psi)}\mathbf{y}$ is non-zero in $F_{\mathcal{T}}^{\infty}(\boldsymbol{\theta}^{\mathbf{w}}\otimes \mathbf{x})$ with $n_{z_0}(\psi)\geq0$. Therefore, as $F_{\mathcal{T}}^{\infty}$ is graded, and by the grading formula drops the grading by 0, we have that $A_0(\mathbf{x})\geq A_0(\mathbf{y})$. Thus, $F_{\mathcal{T},z_{0}}(\boldsymbol{\theta}^{\mathbf{w}}\otimes -)$ is a filtered map of degree 0. Finally, by Lemma \ref{lem: forgetting z0 triangle count is a chain map} we have that $F_{\mathcal{T},z_0}(\boldsymbol{\theta}^{\mathbf{w}}\otimes -)$ induces a morphism of transitive systems.
\end{proof}

\begin{prop}
\label{prop: spec maps for forgetting components of bands}
    Let $\mathcal{T}$, $B$ and $\boldsymbol{\theta}^{\mathbf{w}}$ be as above. Then, the induced map on homology $(F_{\mathcal{T},z_{0}}(\boldsymbol{\theta}^{\mathbf{w}}\otimes -))_{*}$ is identified with
    $F_{\Sigma\backslash L_0\times I}\otimes \text{Id}_{V}$ under the canonical isomorphisms, where $\Sigma$ is the cobordism induced by $B$.
\end{prop}

\begin{proof}
    By the Lemma above we have that $F_{\mathcal{T},z_0}(\boldsymbol{\theta}^{\mathbf{w}}\otimes -)$ induces a morphism of transitive systems, $F\colon \widehat{HFL}_{z_0}(\mathbb{L}) \to \widehat{HFL}_{z_0}(\mathbb{L}')$. Let,
    \begin{equation*}
        g\colon\widehat{HFL}_{z_0}(\mathbb{L})\to \widehat{HFL}(\mathbb{L}\backslash\mathbb{L}_0)\otimes V \qquad \text{and} \qquad g'\colon\widehat{HFL}_{z_0}(\mathbb{L}')\to \widehat{HFL}(\mathbb{L}'\backslash\mathbb{L}_0)\otimes V
    \end{equation*}
    be the canonical isomorphisms from Proposition \ref{prop: canonical identification}. Then the proposition claims that,
    \begin{equation*}
        g'\circ F\circ g^{-1}=F_{\Sigma\backslash L_0\times I}\otimes \text{Id}_{V}.        
    \end{equation*}

    Let $B_0$ be the $\alpha$-band $B$ but on the link $L\backslash L_0$. Then take $\mathcal{Q}_0$ to be a Heegaard triple subordinate to $B_0$ and let $\boldsymbol{\theta}^{\mathbf{w}}_{\mathcal{Q}_0}\in \widehat{CFL}(\mathcal{Q}_0)$ be a representative of the element of highest grading. Further, let $\mathcal{Q}$ be a free stabilisation of $\mathcal{Q}_0$ and let $\boldsymbol{\theta}^{\mathbf{w}}_{\mathcal{Q}}\in \widehat{CFL}(\mathcal{Q})$ be a representative of the element of highest grading. Then by Proposition \ref{prop: free stabilisation} we have that,
    \begin{equation*}
        \phi_{\text{stab}}(\boldsymbol{\theta}^{\mathbf{w}}_{\mathcal{Q}})=\boldsymbol{\theta}^{\mathbf{w}}_{\mathcal{Q}_0}\otimes \theta^+, \qquad \phi_{\text{stab}}(\mathbf{x}\times\eta^{\pm})=\mathbf{x}\otimes \eta^{\pm} \qquad \text{and} \qquad \phi_{\text{stab}}(\mathbf{y}\times\xi^{\pm})=\mathbf{y}\otimes \xi^{\pm},
    \end{equation*}
    where $\mathbf{x}\in \mathbb{T}_{\boldsymbol{\alpha}} \cap \mathbb{T}_{\boldsymbol{\beta}}$, $\mathbf{y}\in \mathbb{T}_{\boldsymbol{\alpha}'} \cap \mathbb{T}_{\boldsymbol{\beta}}$ and $\theta^{+},\eta^{\pm}$ and $\xi^{\pm}$ are the intersection points coming from the stabilisation, as in Figure \ref{fig: Heegaard Triple Triangle example} and the discussion under Definition \ref{defn: triangle stabilisation}. Now applying Proposition \ref{prop: triple free stabilisation} we have,
    \begin{equation*}
        F_{\mathcal{Q}}(\boldsymbol{\theta}^{\mathbf{w}}_{\mathcal{Q}}\otimes (\mathbf{x},\eta^{\pm}))=\phi_{\text{stab}}^{\pm}(F_{\mathcal{Q}_0}(\boldsymbol{\theta}^{\mathbf{w}}_{\mathcal{Q}_0}\otimes \mathbf{x}))=F_{\mathcal{Q}_0}(\boldsymbol{\theta}^{\mathbf{w}}_{\mathcal{Q}_0}\otimes \mathbf{x})\otimes \xi^{\pm}=(F_{\Sigma\backslash(L_{0}\times I)}\otimes \text{Id}_V)(\mathbf{x}\otimes \eta^{\pm}).
    \end{equation*}
    Here, the final equality comes from the definition of $\mathcal{Q}_0$ and the gradings of $\eta^{\pm}$ and $\xi^{\pm}$. That is the following diagram commutes,
    \begin{equation}
    \label{eq: com diagram for forget map 1}
        \begin{tikzcd}
            \widehat{CFL}(\mathcal{Q}_{\alpha,\beta})\arrow[dd,"\phi_{\text{stab}}"]\arrow[rr,"F_{\mathcal{Q}}"] && \widehat{CFL}(\mathcal{Q}_{\alpha',\beta})\arrow[dd,"\phi_{\text{stab}}"] \\
            &&\\
            \widehat{CFL}((\mathcal{Q}_0)_{\alpha,\beta})\otimes V \arrow[rr,"F_{\mathcal{Q}_0}\otimes \text{Id}_V"] && \widehat{CFL}((\mathcal{Q})_{\alpha',\beta})\otimes V
        \end{tikzcd}
    \end{equation}

    Recall that $g$ is defined on the chain level as the composition $\phi_{\text{stab}}\circ \Phi_{\mathcal{T}_{\alpha,\beta}\to \mathcal{Q}_{\alpha,\beta}}$. Hence, if we prove that the following diagram commutes up to chain homotopy equivalence then we are done,
    \begin{equation}
    \label{eq: com diagram for forget map 2}
        \begin{tikzcd}
            (\widehat{CFL}(\mathcal{T}_{\alpha,\beta}),\partial_{z_0})\arrow[rr,"F_{\mathcal{T},z_0}(\boldsymbol{\theta}^{\mathbf{w}}\otimes -)"]\arrow[dd,"\Phi_{\mathcal{T}_{\alpha,\beta}\to \mathcal{Q}_{\alpha,\beta}}"] && (\widehat{CFL}(\mathcal{T}_{\alpha',\beta}),\partial_{z_0})\arrow[dd,"\Phi_{\mathcal{T}_{\alpha',\beta}\to \mathcal{Q}_{\alpha',\beta}}"]\\
            && \\
            \widehat{CFL}(\mathcal{Q}_{\alpha,\beta}) \arrow[rr,"F_{\mathcal{Q}}"]& &\widehat{CFL}(\tilde{\mathcal{Q}}_{\alpha',\beta})
        \end{tikzcd}
    \end{equation}
    To see that this is enough note that commutativity of the above plus commutativity of diagram \eqref{eq: com diagram for forget map 1} gives,
    \begin{equation*}
        (F_{\mathcal{Q}_0}\otimes \text{Id}_V)\circ\phi_{\text{stab}}\circ \Phi_{\mathcal{T}_{\alpha,\beta}\to \mathcal{Q}_{\alpha,\beta}}\simeq\phi_{\text{stab}}\circ \Phi_{\mathcal{T}_{\alpha',\beta}\to \mathcal{Q}_{\alpha',\beta}}\circ F_{\mathcal{T},z_0}(\boldsymbol{\theta}^{\mathbf{w}}\otimes -),
    \end{equation*}
    which gives the result on homology.

    To prove commutativity of diagram \eqref{eq: com diagram for forget map 2} we begin with an observation similar to the one in Remark \ref{rem: sutured stuff}. By definition of being subordinate to $B$ we have that, 
    \begin{equation*}
        \mathcal{T}^n_{\alpha,\beta}:=(\Sigma,\boldsymbol{\alpha}\backslash\{\alpha_n\} ,\boldsymbol{\beta},\mathbf{w},\mathbf{z}),
    \end{equation*}
    is a sutured Heegaard diagram for $S^3\backslash N(L\cup B)$ with meridional sutures $\gamma$. Therefore, as in Remark \ref{rem: sutured stuff} we see that   
    \begin{equation*}
        \mathcal{T}^n_{\alpha,\beta,z_0}:=(\Sigma,\boldsymbol{\alpha}\backslash\{\alpha_n\},\boldsymbol{\beta},\mathbf{w},\mathbf{z}\backslash\{z_0\})
    \end{equation*}
    is a sutured Heegaard diagram for,
    \begin{equation*}
        (Y,\tilde{\gamma}):=(S^3\backslash (N((L\backslash L_0)\cup B)\cup B^3),\tilde{\gamma})
    \end{equation*}
    where the sutures are meridional sutures on the link components and one suture on the sphere component. Now this is the same sutured manifold as the one from $\mathcal{Q}^n_{\alpha,\beta}$ which is defined similarly. Hence, by \cite[Proposition 2.15]{OGSuturedPaper} there exists a sequence of Heegaard moves taking $\mathcal{T}^n_{\alpha,\beta,z_0}$ to $\mathcal{Q}^n_{\alpha,\beta}$. Further, by definition of being subordinate to a band $\mathcal{T}^n_{\alpha',\beta,z_0}$ and $\mathcal{Q}^n_{\alpha',\beta}$ are also sutured Heegaard diagrams for $(Y,\tilde{\gamma})$. Therefore, we can perform the corresponding sequence of Heegaard moves to go between these diagrams. Next, to get Heegaard moves from $\mathcal{T}_{\alpha,\beta}$ to $\mathcal{Q}_{\alpha,\beta}$ we need to do handleslides of the curve $\alpha_n$. Similarly we do handleslides of the $\alpha_n'$ to get from $\mathcal{T}_{\alpha',\beta}$ to $\mathcal{Q}_{\alpha',\beta}$. The handleslides on the $\alpha_n,\alpha_n'$ and the moves between the sutured manifolds are exactly the Heegaard moves that the band map is invariant under by \cite[Lemma 6.5]{CobMapPaper}. Therefore, after setting the variables as before we get the commutativity of diagram \eqref{eq: com diagram for forget map 2}. This completes the proof. 
\end{proof}

We now prove similar results for the quasi-stabilisation map $T^+$. Recall that this map is defined by introducing two new basepoints and two new curves that intersect in two points $\xi^\mathbf{w}$ and $\xi^{\mathbf{z}}$. Then $T^{+}(\mathbf{x})=\mathbf{x}\otimes \xi^{\mathbf{w}}$.

\begin{lem}
\label{lem: quasi-stab maps are degree 0 and chain maps}
    Let $\mathcal{H}=(\Sigma,\boldsymbol{\alpha},\boldsymbol{\beta},\mathbf{w}\cup\{w_{0}\},\mathbf{z}\cup\{z_{0}\})$ be a Heegaard diagram for $\mathbb{L}$, $L_0$ be a component of $L$ with two basepoints and suppose $\mathcal{H}^+$ is a quasi-stabilisation on a component that is not $L_0$. Then there is a well-defined degree 0 filtered chain map with respect to the $\partial_{z_0}$ differential, 
    \begin{equation*}
        T^+_{z_0}\colon \widehat{CFL}(\mathcal{H}) \to \widehat{CFL}(\mathcal{H}^+),
    \end{equation*}
    given by $T^{+}_{z_0}(\mathbf{x})=\mathbf{x}\otimes \xi^{\mathbf{w}}$, and it defines a morphism of transitive systems.
\end{lem}

\begin{proof}
    In \cite[Proposition 5.3]{QuasiMap} it is shown that the map,
    \begin{equation*}
        (T^+)^-\colon CFL^{-}(\mathcal{H}) \to CFL^{-}(\mathcal{H}^+),
    \end{equation*}
    is a well defined $\mathbb{F}[U_0,\ldots, V_n]$ equivariant filtered chain map. Specialising the variables so that $U_i=0$ for all $i$, $V_i=0$ for $i>0$ and $V_0=1$ gives that $T^+_{z_0}$ is a well defined filtered chain map. That it is a degree 0 filtered map follows from the fact that $A_0(\xi^{\mathbf{w}})=0$. That is
    \begin{equation*}
        A_0(T^+(\mathbf{x}))=A_0(\mathbf{x}\otimes \xi^{\mathbf{w}})=A_0(\mathbf{x})+A_0(\xi^{\mathbf{w}})=A_0(\mathbf{x}).
    \end{equation*}
    Finally, as $(T^+)^-$ is a morphism of transitive systems it follows that $T^+_{z_0}$ is too.
\end{proof}

To prove the statement about the map on homology we will need to understand how a quasi-stabilisation and free stabilisation interact. To do this we will need a double neck stretching argument. The main point is that quasi-stabilisations and free stabilisations are defined by taking particular almost complex structures $J(T)$ such that the annulus, or neck, where the stabilisation is taking place is of length $T$. See \cite[Definition 5.1]{QuasiMap} for the conditions on the quasi-stabilisation and \cite[Definition 5.1]{GraphCobZemke} for the free stabilisation. Therefore, as in the proofs of \cite[Proposition 5.14]{GraphCobZemke} and \cite[Lemma 8.2]{QuasiMap} we need to consider the map associated to the change in almost complex structure. We will do this by combining the two arguments previously mentioned. Throughout this we will use Lipshitz's cylindrical reformulation, see \cite{CylindricalReformulation} for details.

\begin{lem}
\label{lem: commutativity from double neck stretching}
    Let $\mathcal{H}_0$ be a Heegaard diagram for a link $\mathbb{L}$. Further, suppose that $\mathcal{H}^+$ is the result of both a quasi-stabilisation and a free stabilisation. Then we have the following chain homotopy equivalence,
    \begin{equation*}
        T^+_0\circ \varphi^{\pm}\simeq \varphi^{\pm}_+\circ T^+.
    \end{equation*}
    
\end{lem}

The following proof is an adaptation of \cite[Proposition 5.14]{GraphCobZemke} and \cite[Lemma 8.2]{QuasiMap} with a few tweaks as these deal with two free stabilisations and two quasi-stabilisations respectively. Therefore, we omit a few details and mainly explain the adaptations. 

\begin{proof}
    Let $T_1,T_2,T_1',T_2'>0$ be neck lengths such that,
    \begin{enumerate}
        \item $T_1$ and $T_1'$ are neck lengths for the free stabilisation;
        \item $T_2$ and $T_2'$ are neck lengths for the quasi-stabilisation;
        \item $T_1>T_2$ and $T_1'<T_2'$.
    \end{enumerate}
    Then we will show that the following diagram commutes,
    \begin{equation}
        \begin{tikzcd}
            \widehat{CFL}(\mathcal{H},J(T_1'))\arrow[rr,"T^+"] \arrow[dd,"\varphi^{\pm}"]&& \widehat{CFL}(\mathcal{H}^+,J(T_1',T_2'))\arrow[d,"\Phi_{J(T_1',T_2')\to J(T_1,T_2)}"] \\
            &&\widehat{CFL}(\mathcal{H}^+,J(T_1,T_2))\arrow[d,"\varphi^{\pm}_+"]\\
            \widehat{CFL}(\mathcal{H}_0)\arrow[rr,"T_0^+"] &&\widehat{CFL}(\mathcal{H}_0^+,J(T_2))
        \end{tikzcd}
    \end{equation}
    where $\Phi_{J(T_1',T_2')\to J(T_1,T_2)}$ is the map associated to the change in almost complex structure. Then we will show that,
    \begin{equation*}
        \Phi_{J(T_1',T_2')\to J(T_1,T_2)}\simeq \text{Id}
    \end{equation*}
    where $\text{Id}$ is understood to identify the Heegaard states between the two complexes. Note that this is sufficient to prove the statement of the lemma.
    
    Recall that the map $\Phi_{J(T_1',T_2')\to J(T_1,T_2)}$ is obtained by counting Maslov index 0 pseudo-holomorphic curves. Let $\phi$ be the homology class of a holomorphic curve contributing to this. It is shown in \cite{LinkSurgeryFormulaOG} that by taking longer and longer neck lengths and holomorphic representatives of $\phi$ then the limiting curves are broken holomorphic curves $\phi_{\mathcal{H}_0}$ in $\mathcal{H}_0^s=(\Sigma,\boldsymbol{\alpha}\cup \alpha_s',\boldsymbol{\beta},\mathbf{w},\mathbf{z})$; $\phi_+$ in $(S^2,\alpha_s,\beta_s)$; $\phi_0$ in $(S^2,\alpha_0,\beta_0)$; and a collection of cylindrical boundary degenerations $\mathcal{A}$ with boundary degenerations $\boldsymbol{\alpha}\cup \alpha_s$. Further, let $R$ be the region of $\Sigma\backslash\boldsymbol{\alpha}$ where the quasi-stabilisation takes place. Then note that one region of $R\backslash\alpha_s$ contains $w_i$, call it $R_{w_i}$. The other region contains $z_i$, call it $R_{z_i}$. Then we denote by $m_{w_i}(\mathcal{A})$ the coefficient in the domain $R_{w_i}$ and similarly $m_{z_i}(\mathcal{A})$ for $R_{z_i}$.

    Then, combining formulas from the proofs of \cite[Proposition 5.3]{GraphCobZemke} and \cite[Proposition 5.3]{QuasiMap} we get,
    \begin{align*}
        \mu(\phi)=\mu&(\phi_{\mathcal{G}_0})+n_{w_s}(\phi)+n_{z_s}(\phi)+m_{z_i}(\mathcal{A})+ m_{w_i}(\mathcal{A}) +
        \\ &\text{gr}(\theta_1,\theta_2) +2n_{w_0}(\phi_0)+2\sum_{\substack{\mathcal{D}\in C(\Sigma\backslash\boldsymbol{\alpha}) \\ \alpha_s\cap \mathcal{D=\emptyset}}}n_{\mathcal{D}}(\phi),
    \end{align*}
    where $\theta_1,\theta_2\in \{\theta^+,\theta^-\}$. It is shown between the two previously mentioned papers that all of these terms are non-negative except possibly $\text{gr}(\theta_1,\theta_2)$ in the case $\text{gr}(\theta^-,\theta^+)=-1$. Hence, in the cases $\theta_1=\theta_2$ we have that all the above terms are 0. Now by the argument in \cite{QuasiMap} we have that only constant discs are counted in these cases. Further, in the final case we have $\text{gr}(\theta^+,\theta^-)=1$ so no such domains exist.

    Now in the case of $\text{gr}(\theta^-,\theta^+)=-1$ we see that exactly one of $\mu(\phi_{\mathcal{G}_0}),n_{w_s}(\phi),n_{z_s}(\phi),m_{z_i}(\mathcal{A})$ or $m_{w_i}(\mathcal{A})$ equals one. Note that as we are counting holomorphic discs that do not cross any basepoints we can conclude that $\mu(\phi_{\mathcal{G}_0})=1$ as otherwise the disc would not be counted. Now, similar to in Proposition \ref{prop: triple free stabilisation} we also have the following equality,
    \begin{align*}
        \mu(\phi)=\mu&(\phi_{\mathcal{H}_0})+n_{w_s}(\phi)+n_{z_s}(\phi)+m_{z_i}(\mathcal{A})+ m_{w_i}(\mathcal{A}) +
        \\ &\mu(\phi_0) -2n_{w_j}(\phi_0)+2\sum_{\substack{\mathcal{D}\in C(\Sigma\backslash\boldsymbol{\alpha}) \\ \alpha_s\cap \mathcal{D=\emptyset}}}n_{\mathcal{D}}(\phi),
    \end{align*}
    where $w_j$ is the basepoint where the free stabilisation takes place. So we have,
    \begin{equation*}
        0=1+\mu(\phi_0) -2n_{w_j}(\phi_0).
    \end{equation*}
    However, as $\mu(\phi_0)$ is non-negative we must have that $n_{w_j}(\phi_0)>0$. Therefore, as we are considering the hat complex we do not count any domains in this case. Combining all of these points shows that the map is chain homotopic to the identity, which completes the proof.
\end{proof}

Using this we can prove the following.

\begin{prop}
\label{prop: spec map for forgetting components of quasi-stab}
    Let $\mathcal{H}$, $\mathcal{H}^+$, $\mathbb{L}$ and $L_0$ be as in Lemma \ref{lem: quasi-stab maps are degree 0 and chain maps}. Then, the induced map on homology, $(T^+_{z_0})_*$ is equal to $T^+_{L_0}\otimes \text{Id}_V$ under the canonical identifications, where $T^+_{L_0}$ is the quasi-stabilisation on $\mathbb{L}\backslash\mathbb{L}_0$.
\end{prop}

\begin{proof}
    As in the proof of Proposition \ref{prop: spec maps for forgetting components of bands} we want to show that,
    \begin{equation*}
        g'\circ T^+_{z_0} \circ g^{-1}=T_{L_0}^+\otimes \text{Id}_V.
    \end{equation*}
    To start let $\mathcal{G}_0$ be a Heegaard diagram for $\mathbb{L}\backslash\mathbb{L}_0$. Then we let $\mathcal{G}_0^+$ be the result of quasi-stabilisation, $\mathcal{G}$ the result of a free stabilisation and $\mathcal{G}^+$ the result of both stabilisations. Further, let $T^+_{\mathcal{G}}\colon \widehat{CFL}(\mathcal{G}) \to \widehat{CFL}(\mathcal{G}^+)$ be the quasi-stabilisation map.
    
    First, we want to show that,
    \begin{equation}
    \label{eq: com diagram for stab 1}
        \phi_{\text{stab}}\circ T^+_{\mathcal{G}}\simeq (T^+_{L_0}\otimes \text{Id}_V)\circ \phi_{\text{stab}}.
    \end{equation}
    Recall that for $i \in \{+,-\}$ we have,
    \begin{equation*}
        \phi_{\text{stab}}(\mathbf{x}\times \theta^i)=\varphi^+(\mathbf{x}\times \theta^i)\otimes \theta^++\varphi^-(\mathbf{x}\times \theta^i)\otimes\theta^-.
    \end{equation*}
    Then by applying Lemma \ref{lem: commutativity from double neck stretching} we get the commutativity of equation \eqref{eq: com diagram for stab 1}.

    Now as in the proof of Proposition \ref{prop: spec maps for forgetting components of bands} we need to show that the following diagram commutes,
    \begin{equation}
    \label{eq: com diagram for stab 2}
        \begin{tikzcd}
            (\widehat{CFL}(\mathcal{H}),\partial_{z_0}) \arrow[rr,"T^+_{z_0}"] \arrow[dd,"\Phi_{\mathcal{H}\to \mathcal{G}}"] && (\widehat{CFL}(\mathcal{H}^+),\partial_{z_0}) \arrow[dd,"\Phi_{\mathcal{H}^+\to \mathcal{G}^+}"] \\
            && \\
            \widehat{CFL}(\mathcal{G}) \arrow[rr,"T^+_{\mathcal{G}}"] &&\widehat{CFL}(\mathcal{G}^+)
        \end{tikzcd}
    \end{equation}
    This follows from identical arguments to the previous proofs by noting that $(T^+)^-$ is a map of transitive systems so commutes with the maps from Heegaard moves and then specialising the variables. Further, the vertical maps come from Remark \ref{rem: sutured stuff}. Now combining equation \eqref{eq: com diagram for stab 1} and diagram \eqref{eq: com diagram for stab 2} we get the result.
\end{proof}

Combining these results for the band map and the quasi-stabilisation gives the result for the pair of pants map. Formally, we have the following.

\begin{cor}
\label{cor: spec maps for pants}
    Let $P\colon \mathbb{L}\to \mathbb{L}'$ be a pair of pants cobordism, that is a composition of a quasi-stabilisation and then an $\alpha$-band $B$. Then,
    \begin{equation*}
        F_P\colon \widehat{CFL}_{z_0}(\mathbb{L})\to \widehat{CFL}_{z_0}(\mathbb{L}'),
    \end{equation*}
    is a morphism of transitive systems and is a degree 0 filtered chain map. Further, the induced map on homology,
    \begin{equation*}
        (F_{P})_*\colon \widehat{HFL}_{z_0}(\mathbb{L})\to \widehat{HFL}_{z_0}(\mathbb{L}'),
    \end{equation*}
    is equal to $F_{P\backslash(L_0\times I)}\otimes \text{Id}_V$ under the canonical isomorphisms.
\end{cor}

\begin{proof}
    The first statement that $F_P$ is a morphism of transitive systems and degree 0 follows from Lemmas \ref{lem: triangle map is well defined and a filtred map} and \ref{lem: quasi-stab maps are degree 0 and chain maps}. Propositions \ref{prop: spec maps for forgetting components of bands} and \ref{prop: spec map for forgetting components of quasi-stab} together prove the second claim about the induced map on homology.
\end{proof}

\subsection{Example of Forgetting Components}
\hfill

We now consider an example  that illustrates the difference between the map $f_*$ and $E^{\infty}(f)$ for band maps. This example is inspired by the calculations in \cite[Proposition 4.2]{floerlasagna}.

\begin{exa}
\label{ex: forgetting cobordisms}
Let $\mathbb{L}$ be two disjoint unknots one labelled $L_0$ with two basepoints $w_0$ and $z_0$, and the other labelled $L_1$ with four basepoints. Further, let $\mathbb{L}'$ be the three component link where one component is the unknot $L_0$ with the same two basepoints and the other components $L_1'$ and $L_2'$ are meridians of $L_0$ oriented in opposite directions both with two basepoints. Now there is a link cobordism $\Sigma\colon \mathbb{L}\to \mathbb{L}'$ which is an $\alpha$-band, where the band $B$ intersects just the component $L_1$ in $L$. See Figure \ref{fig: example cob for spec seq sec} for a diagram.

\begin{figure}[h]
    \centering
    \def\svgwidth{1\textwidth}
\begingroup%
  \makeatletter%
  \providecommand\color[2][]{%
    \errmessage{(Inkscape) Color is used for the text in Inkscape, but the package 'color.sty' is not loaded}%
    \renewcommand\color[2][]{}%
  }%
  \providecommand\transparent[1]{%
    \errmessage{(Inkscape) Transparency is used (non-zero) for the text in Inkscape, but the package 'transparent.sty' is not loaded}%
    \renewcommand\transparent[1]{}%
  }%
  \providecommand\rotatebox[2]{#2}%
  \newcommand*\fsize{\dimexpr\f@size pt\relax}%
  \newcommand*\lineheight[1]{\fontsize{\fsize}{#1\fsize}\selectfont}%
  \ifx\svgwidth\undefined%
    \setlength{\unitlength}{568.61830343bp}%
    \ifx\svgscale\undefined%
      \relax%
    \else%
      \setlength{\unitlength}{\unitlength * \real{\svgscale}}%
    \fi%
  \else%
    \setlength{\unitlength}{\svgwidth}%
  \fi%
  \global\let\svgwidth\undefined%
  \global\let\svgscale\undefined%
  \makeatother%
  \begin{picture}(1,0.35211543)%
    \lineheight{1}%
    \setlength\tabcolsep{0pt}%
    \put(0,0){\includegraphics[width=\unitlength,page=1]{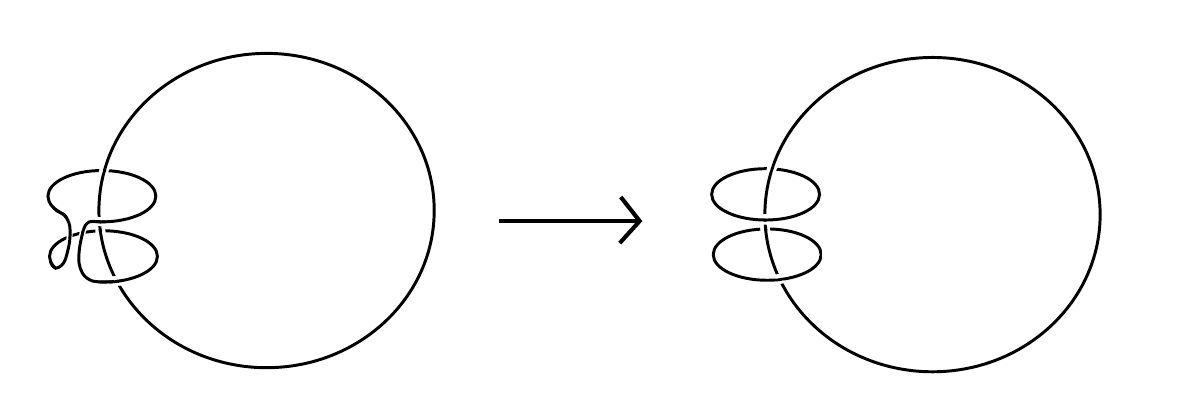}}%
    \put(0.47672472,0.18583666){\color[rgb]{0,0,0}\makebox(0,0)[lt]{\lineheight{1.25}\smash{\begin{tabular}[t]{l}$\Sigma$\end{tabular}}}}%
    \put(0.37854952,0.25305573){\color[rgb]{0,0,0}\makebox(0,0)[lt]{\lineheight{1.25}\smash{\begin{tabular}[t]{l}$L_0$\end{tabular}}}}%
    \put(0.03891631,0.22121514){\color[rgb]{0,0,0}\makebox(0,0)[lt]{\lineheight{1.25}\smash{\begin{tabular}[t]{l}$L_1$\end{tabular}}}}%
    \put(0.60408717,0.21856176){\color[rgb]{0,0,0}\makebox(0,0)[lt]{\lineheight{1.25}\smash{\begin{tabular}[t]{l}$L_1'$\end{tabular}}}}%
    \put(0.59878041,0.10181283){\color[rgb]{0,0,0}\makebox(0,0)[lt]{\lineheight{1.25}\smash{\begin{tabular}[t]{l}$L_2'$\end{tabular}}}}%
    \put(0.94283586,0.2450956){\color[rgb]{0,0,0}\makebox(0,0)[lt]{\lineheight{1.25}\smash{\begin{tabular}[t]{l}$L_0$\end{tabular}}}}%
    \put(0,0){\includegraphics[width=\unitlength,page=2]{Cobordism_map_example_for_spec_seq_section.pdf}}%
  \end{picture}%
\endgroup%

    \caption{Diagram of band map in Example \ref{ex: forgetting cobordisms}.}
    \label{fig: example cob for spec seq sec}
\end{figure}

Note that $L_0\times I$ is a component of $\Sigma$. Hence, we also consider the cobordism $\Sigma\backslash (L_0\times I)$, which is also a band map. For this example we denote the associated maps by,
\begin{equation*}
    F:=F_{\Sigma}\colon \widehat{HFL}(\mathbb{L})\to \widehat{HFL}(\mathbb{L}') \qquad \text{and} \qquad G:=F_{\Sigma\backslash (L_0\times I)}\colon \widehat{HFL}(\mathbb{L}_1)\to \widehat{HFL}(\mathbb{L}_1' \cup \mathbb{L}_2').
\end{equation*}
We will also write $f\colon \widehat{CFL}_{z_0}(\mathbb{L})\to \widehat{CFL}_{z_0}(\mathbb{L}')$ for the morphism of transitive systems. Then, we want to compare $f_*$ and $E^\infty(f)$, and show how this gives us some information about $F=E^{1}(f)$. Note that by Proposition \ref{prop: spec maps for forgetting components of bands} we have $g'\circ f_*\circ g^{-1}=G\otimes \text{Id}_V$.

\begin{figure}[h]
    \centering
    \def\svgwidth{0.5\textwidth}
    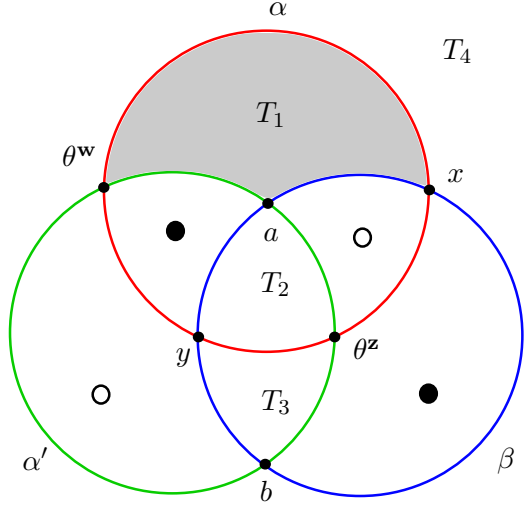
    \caption{Heegaard triple for $\Sigma\backslash(L_0\times I)$ that computes the map $G$.}
    \label{fig: Heegaard Triple for simple band map}
\end{figure}

To start we analyse the cobordism $\Sigma\backslash (L_0\times I)$. Note that we have the identifications,
\begin{equation*}
    \widehat{HFL}(\mathbb{L}_1)\cong \mathbb{F}\langle x,y \rangle \qquad \text{and} \qquad \widehat{HFL}(\mathbb{L}_1'\cup \mathbb{L}_2')\cong \mathbb{F}\langle a,b\rangle,
\end{equation*}
where $M(x)=-1,M(y)=0,A(x)=-\frac{1}{2}$ and $A(y)=\frac{1}{2}$, and $M(a)=-1,M(b)=0$ and $A(a)=A(b)=0$. Then we claim the map associated to this cobordism is given by $G(x)=a$ and $G(y)=0$. To show this we first compute using Theorem \ref{thm: grading change formulas} that the shift in the Alexander grading is $\frac{1}{2}$ and the shift in the Maslov grading is $0$. Note that this implies $G(y)=0$. To compute $G(x)$ we consider the Heegaard triple in Figure \ref{fig: Heegaard Triple for simple band map}. We claim that the grey triangle is the only holomorphic triangle contributing to $G(x)$. To see this we first note that by the grading change formula there are no holomorphic triangles connecting $x$, $\Theta^{\mathbf{w}}$ and $b$. Next we consider, 
\begin{equation*}
    D:=c_1T_1+c_2T_2 +c_3T_3+c_4T_4,
\end{equation*}
and note that we must have $\partial(\partial D \cap \alpha')=a-\theta^{\mathbf{w}}$ so that it represents a class in $\pi_2(\theta^{\mathbf{w}},x,a)$.  Further, from the diagram we have 
\begin{equation*}
    \partial(\partial D \cap \alpha')=(c_4-c_1)\theta^{\mathbf{w}} +(c_1+c_2)a-(c_3+c_4)b+(c_3-c_2)\theta^{\mathbf{z}},
\end{equation*}
and by the fact that for a holomorphic disc to exist we must have $c_i\geq 0$ for all $i$, we have that the only solution is $c_1=1$ and $c_2=c_3=c_4=0$. Finally, we note that by the Riemann mapping theorem there is a unique holomorphic representative in this class of Whitney triangles, and this class clearly has Maslov index $0$. Therefore, $G(x)=a$ and so by Proposition \ref{prop: spec maps for forgetting components of bands} we get a description of $f_*$.

We now want to examine both of the spectral sequences. To do this we first note that,
\begin{equation*}
    \widehat{HFL}(\mathbb{L})\cong \widehat{HFL}(\mathbb{L}_1)\otimes V\cong \mathbb{F}\langle \tilde{x},\tilde{y}\rangle\otimes V,
\end{equation*}
where the first isomorphism uses Lemma \ref{lem:  unknot component} and we add the tilde to distinguish the generators of the $E^1$- and $E^{\infty}$-page. The second isomorphism is from the previous paragraph. Thus as $L\backslash L_0=L_1$ we have that the spectral sequence collapses on the $E^1$-page. Therefore, we will abuse notation and drop the tildes from now. 

We now compute the second spectral sequence and the grading on the $E^\infty$-page. First, from \cite[Proposition 4.2]{floerlasagna} we have that,
\begin{equation*}
    \dim\left(\widehat{HFL}(\mathbb{L}')\right)=16,
\end{equation*}
this comes from the fact it is alternating and from computing the Alexander polynomial. To write the gradings we define the following notation,
\begin{equation*}
    \mathbb{F}_{(i,(j_0,j_1+j_2))}^d\cong \widehat{HFL}_i(\mathbb{L}',(j_0,j_1+j_2)),
\end{equation*}
where $d$ is the dimension, $i$ the Maslov grading and $(j_0,j_1+j_2)$ is the collapsed Alexander grading. Using this we can write the gradings of $\widehat{HFL}(\mathbb{L}')$, which is the $E^1$-page, with the last two Alexander multigradings collapsed,
\begin{equation}
\label{eq: HFL from example}
    \begin{tikzcd}
        \mathbb{F}_{(-1,(-1,1))} & \mathbb{F}^2_{(0,(0,1))} \arrow[l, two heads] & \mathbb{F}_{(1,(1,1))} \arrow[l,hook',"i"'] \\
        \mathbb{F}_{(-2,(-1,0))}^2 & \mathbb{F}^4_{(-1,(0,0))} \arrow[l] & \mathbb{F}_{(0,(1,0))}^2 \arrow[l]\\
        \mathbb{F}_{(-3,(-1,-1))} & \mathbb{F}^2_{(-2,(0,-1))} \arrow[l, two heads] & \mathbb{F}_{(-1,(1,-1))}. \arrow[l,hook']
    \end{tikzcd}
\end{equation}
Further, the $E^\infty$-page is,
\begin{equation*}
    E^{\infty}\left(\widehat{CFL}(\mathbb{L'})\right)\cong\widehat{HFL}(\mathbb{L}_1'\cup \mathbb{L}_2')\otimes V\cong \mathbb{F}\langle a,b \rangle\otimes V,
\end{equation*}
and is supported in Alexander grading $0$. Using this we see that as a bigraded vector space the $E^{\infty}$-page is,
\begin{equation}
\label{eq: H*C' for example}
\mathbb{F}_{(-2,0)}\oplus\mathbb{F}_{(-1,0)}^2\oplus\mathbb{F}_{(0,0)}
\end{equation}
where $\mathbb{F}_{(i,j)}^d$ is the part supported in Maslov grading $i$ and Alexander grading $j$. From this it is clear that the spectral sequence collapses on the $E^2$-page. Further, from inspection of equation \eqref{eq: HFL from example} we can see the grading on $E^\infty$-page is given by,
\begin{equation}
\label{eq: Einfty' for example}
    \mathbb{F}_{(-2,(-1,0))}\oplus \mathbb{F}^2_{(-1,(0,0))}\oplus \mathbb{F}_{(0,(1,0))},
\end{equation}
where the grading notation is as before.

We now claim that,
\begin{equation*}
    E^{\infty}(f)(x\otimes B)=E^{\infty}(f)(y\otimes B)=E^{\infty}(f)(y\otimes T)=0 \qquad \text{and} \qquad E^{\infty}(f)(x\otimes T)\neq 0.
\end{equation*}
Note that this differs from $f_{*}$ as $f_*(x\otimes B)=a\otimes B$. To see the above equalities note that the first $E^1$-page is isomorphic to,
\begin{equation*}
    \mathbb{F}_{(-2,(0,-\frac{1}{2}))}\oplus \mathbb{F}_{(-1,(0,-\frac{1}{2}))}\oplus \mathbb{F}_{(-1,(0,\frac{1}{2}))} \oplus \mathbb{F}_{(0,(0,\frac{1}{2}))},
\end{equation*}
where each of these is generated in order by $x\otimes B,x\otimes T,y\otimes B,y\otimes T$. Further, by Theorem \ref{thm: grading change formulas} $f$ increases the final grading by $\frac{1}{2}$. Also by Lemma \ref{lem: f_* and E^infty (f)} $E^{\infty}(f)$ is graded with respect to the first Alexander grading. These together give the equalities to $0$.

To see non-vanishing of $E^{\infty}(f)(x\otimes T)$ we note that $\widehat{HFL}_{z_0}(\mathbb{L})$ is isomorphic to,
\begin{equation*}
    \mathbb{F}_{(-2,-\frac{1}{2})}\oplus \mathbb{F}_{(-1,-\frac{1}{2})}\oplus \mathbb{F}_{(-1,\frac{1}{2})} \oplus \mathbb{F}_{(0,\frac{1}{2})}.
\end{equation*}
So there is a canonical isomorphism to the associated graded vector space by Lemma \ref{lem: filtration and associated grading in single degree} and the fact that the filtration respects these gradings. Similarly applying Lemma \ref{lem: filtration and associated grading in single degree} to $\widehat{HFL}_{z_0}(\mathbb{L}')$ and using equation \eqref{eq: H*C' for example} and \eqref{eq: Einfty' for example} we get a canonical isomorphism. So abusing notation we have,
\begin{equation*}
    E^{\infty}(f)(x\otimes T)=f_*(x\otimes T)=(G\otimes \text{Id}_V)(x\otimes T)=a\otimes T.
\end{equation*}

We can use this to compute some values of $F=E^1(f)$. First, from a grading change argument we see that $F(x\otimes B)=F(y\otimes B)=0$. Further, we see that $F(x\otimes T)\neq 0$ as,
\begin{equation*}
[F(x\otimes T)]^{\infty}=E^{\infty}(f)(x\otimes T)\neq 0,
\end{equation*}
where by $[\cdot]^{\infty}$ we mean the representative of that element on the $E^{\infty}$-page. This example also shows a limitation of this trick as we cannot determine $F(y\otimes T)$. To see this first note that by grading arguments $F(y\otimes T)\in\mathbb{F}^2_{(0,(0,1))}$ and so by the above calculations $E^{\infty}(f)(y\otimes T)=0$. Now there is no way to tell from this information if
\begin{equation*}
    F(y\otimes T)=0 \qquad \text{or} \qquad 0\neq F(y\otimes T)\in \text{im}(i),
\end{equation*}
where $i$ is the map in equation \eqref{eq: HFL from example}. This completes the example.
\end{exa}

The above example gives an outline of how the rest of the paper will proceed. The main idea is that we will be using the same trick as in the last paragraph to show that infinitely many pair of pants maps are non-vanishing. To get to this trick we will need to know both the $E^{1}$- and $E^{\infty}$-pages of the spectral sequences. We spend the next section setting this up. Then an inductive argument is used to prove that if one of these pair of pants maps is non-vanishing the higher ones are too. Finally, we specialise to the case of the $(-n)$-framed unknot, where we compute a simple cobordism. This then allows us to prove non-vanishing of $\mathcal{FL}(X_{-n}(U))$.

\section{Link Floer Homology of $K_{-n}(a,b)$}
\label{sec: Link Floer Homology of T(r,rn)}

Our main aim is to use Theorem \ref{thm: cable and lasagna} to prove some non-vanishing results about Floer lasagna modules of traces of negative L-space knots. To do this we will use Proposition \ref{prop: spec maps for forgetting components of bands}, which means as in Example \ref{ex: forgetting cobordisms} we need to understand the $E^1$- and $E^\infty$-pages of the forgetting components link cobordisms. This is the primary aim of this section. Also note to ease notation from now on we will use $a$ and $b$ rather than $a_+$ and $a_-$.

We start by fully computing the link Floer homology of $U_{-n}(a,b)$ for all $n,a$ and $b$. This is a generalisation of computations of the link Floer homology of $T(n,n)$-torus links started in \cite{nntorus} and completed in \cite{linkLspacernrm}. The main tool of this section is \cite[Theorem 3]{linkLspacernrm} which gives a formula for the link Floer homology of cables of L-space knots. Our contribution is explicitly applying this to the cables of the unknot with varying orientation and certain Alexander multigradings for all L-space knots. We start by recalling the symmetry of link Floer homology in the Alexander multigrading about $0$. For a proof see \cite{OSLinkFloer}.

\begin{lem}
\label{lem: alexander symmetry}
    Let $\mathbb{L}$ be an $r$-component link, $i\in \mathbb{Z}$ and $\mathbf{k}\in\mathbb{Z}^r$. Then $\widehat{HFL}(\mathbb{L})$ has the following symmetry,
    \begin{equation}
        \widehat{HFL}_{i}(L,\textbf{k})\cong \widehat{HFL}_{i-2|\textbf{k}|}(L,-\textbf{k}).
    \end{equation}
\end{lem}

Note that for $n>0$ the link $U_{n}(r,0)$ is the torus link $T(r,rn)$. Using this we have the following proposition about the link Floer homology of cables of torus links. We state this in terms of $T(r,rn)$ to draw the comparison to \cite{nntorus}.

\begin{prop}
\label{prop: HFL of Kn(r,0)}
    Let $n\geq 1$, $r\geq 1$ and $k=\frac{n}{2}(r-1)-ni-j$ be such that $-\frac{n}{2}(r-1)\leq k \leq \frac{n}{2}(r-1)$,  $i \in \mathbb{Z}_{\geq0}$ and $0\leq j \leq n-1$ . Then:
    \begin{itemize}
        \item for $i\leq\frac{r-2}{2}$ we have,
            \begin{align}
            \label{eq: Kn(r,0) i small}
                \widehat{HFL}\left(T(r,rn),(k,\ldots,k)\right)\cong \bigoplus_{p=0}^{i}\mathbb{F}_{-ni(i+1)-2j(i+1)-p}^{\binom{r-1}{p}}\oplus \bigoplus_{p=0}^{i-\delta_{j,0}}\mathbb{F}_{-ni(i+1)-2j(i+1)-r+2+p}^{\binom{r-1}{p}},
            \end{align}
        \item for $i>\frac{r-2}{2}$ we have,
            \begin{align}
            \label{eq: Kn(r,0) i large}
                \widehat{HFL}\left(T(r,rn),(k,\ldots,k)\right)\cong \bigoplus_{p=0}^{r-2-i+\delta_{j,0}}\mathbb{F}_{-ni(i+1)-2j(i+1)-p}^{\binom{r-1}{p}}\oplus \bigoplus_{p=0}^{r-2-i}\mathbb{F}_{-ni(i+1)-2j(i+1)-r+2+p}^{\binom{r-1}{p}},
            \end{align}
        \item finally for $\mathbf{k}\in \{k-1,k\}^r$ with $1\leq p\leq r-1$ coordinates equal to $k-1$ we have,
            \begin{equation}
            \label{eq: Kn(r,0) not a meeting point}
                \widehat{HFL}(T(r,rn),\mathbf{k})\cong \mathbb{F}_{-ni(i+1)-2j(i+1)-i-p}^{\binom{r-2}{i}}.
            \end{equation}
    \end{itemize}
    where $\delta_{j,0}=1$ if $j=0$ and $\delta_{j,0}=0$ otherwise.
\end{prop}

\begin{proof}
    Using \cite[Theorem 3]{linkLspacernrm} we can compute the link Floer homology of $T(r,rn)$, for all $n\geq 1$ and $r\geq1$. In their notation this is $U_{r,rn}$. To do this we consider the following Laurent series in $\mathbb{Z}[t^{\pm\frac{1}{2}}]$,
    \begin{equation}
        q(t)=\sum_{k\in\mathbb{Z}} \mathbf{h}(k)t^k=\frac{t^{-1}(t^{\frac{nr}{2}}-t^{-\frac{nr}{2}})}{(1-t^{-1})^2(t^{\frac{n}{2}}-t^{-\frac{n}{2}})},
    \end{equation}
    and the function $\beta\colon \mathbb{Z}\to \mathbb{Z}$ given by $\beta(k):=\mathbf{h}(k-1)-\mathbf{h}(k)-1$. We also have the following two identifications,
    \begin{equation*}
        (1-t^{-1})^{-2}=\sum_{i=0}^{\infty}(i+1)t^{-i} \qquad \text{and} \qquad (1-t^{-n})^{-1}=\sum_{i=0}^{\infty}t^{-ni}.
    \end{equation*}
    Now substituting these in and rearranging we get,
    \begin{align}
        q(t)&=(t^{\frac{n}{2}(r-1)-1}-t^{-\frac{n}{2}(r+1)-1})\left(\sum_{i=0}^{\infty}\sum_{j=0}^{n-1}\left(\frac{ni(i+1)}{2}+(i+1)(j+1)\right)t^{-(ni+j)}\right).
    \end{align}
    Note that for $k=\frac{n}{2}(r-1)+l$ with $l\geq 0$ we have that $\mathbf{h}(k)=0$. Therefore, we have $\beta(\frac{n}{2}(r-1)+l)=-1$ for $l>0$ and so by \cite[Theorem 3]{linkLspacernrm} we have,
    \begin{equation*}
        \widehat{HFL}(T(r,rn),(k,\ldots,k))\cong \{0\}.
    \end{equation*}
    Now using this, the fact that $-\frac{n}{2}(r+1)-1<-k$ and Lemma \ref{lem: alexander symmetry} we see that we only need to consider the following Laurent polynomial,
    \begin{equation}
    \label{eq: unknot simp poly}
        p(t)=\sum_{i=0}^{\infty}\sum_{j=0}^{n-1}\left(\frac{ni(i+1)}{2}+(i+1)(j+1)\right)t^{\frac{n(r-1)}{2}-1-ni-j}.
    \end{equation}
    From the above we can extract the functions $\mathbf{h}\colon \frac{1}{2}\mathbb{Z}\to \mathbb{Z}$ and $\beta\colon\frac{1}{2}\mathbb{Z}\to \mathbb{Z}$. In fact for 
    \begin{equation*}
        k=\frac{n}{2}(r-1)-ni-j,
    \end{equation*}
    we have,
    \begin{equation*}
        2\mathbf{h}(k)=ni(i+1)+2j(i+1) \qquad \text{and} \qquad \beta(k)=i
    \end{equation*}
    Note that we have reindexed $j$ here. Thus, for $i\leq\frac{r-2}{2}$ we have that $\beta(k)+\beta(k+1)\leq r-2$ so by \cite[Theorem 3]{linkLspacernrm} we get equation \eqref{eq: Kn(r,0) i small}. Similarly, for $i>\frac{r-2}{2}$ we have $\beta(k)+\beta(k+1)\geq r-2$ and again applying \cite[Theorem 3]{linkLspacernrm} we get equation \eqref{eq: Kn(r,0) i large}. Finally, again applying \cite[Theorem 3]{linkLspacernrm} we get equation \eqref{eq: Kn(r,0) not a meeting point}.
\end{proof}

Using this proposition we can fully compute the link Floer homology of $U_{n}(a,b)$ and $U_{-n}(a,b)$. To do this we simply need to know how link Floer homology changes under changing orientations and mirroring. This is the content of the following two lemmas from \cite{OSHFOG}, although we use the formulation from \cite{LemmasPaper}. The first tells us about orientation reversals.

\begin{lem}
\label{lem: orientation reversing}
    Let $\mathbb{L}$ be a pointed $n$-component link and $L_{k}$ be the $k^{th}$ component of $L$. Let $\mathbb{L}'$ be the result of reversing the orientation of $L_{k}$. Then the link Floer homologies of $\mathbb{L}$ and $\mathbb{L}'$ are related by the following isomorphism,
    \begin{equation}
    \label{eq: orientation reversal iso}
        \widehat{HFL}_{i}(\mathbb{L},(j_1,\ldots, j_n))\cong \widehat{HFL}_{i-2j_k+l_{k}}(\mathbb{L}',(j_1,\ldots, -j_k,\ldots , j_n))
    \end{equation}
    where $l_k$ is the linking number of $L_{k}$ with the rest of $L$.
\end{lem}

In the above we keep the basepoints in the same place when we reverse the orientation. However, the result still holds if we do not do this but it will change the isomorphism. This nuance will appear again in the next section, but we can ignore it for now. The next lemma is about mirroring the link. In this context by mirroring we mean take any diagram of $L$ and swap the sign of all the crossings. 

\begin{lem}
\label{lem: mirror}
    Let $m(\mathbb{L})$ denote the mirror of $\mathbb{L}$. Then the following isomorphism holds for all $i\in \mathbb{Z}$ and $\mathbf{j}\in \mathbb{Z}^{n}$,
    \begin{equation}
        \widehat{HFL}_{i}(\mathbb{L},\textbf{j})\cong \widehat{HFL}_{1-|L|-i+2|\textbf{j}|}(m(\mathbb{L}),\textbf{j}).
    \end{equation}
\end{lem}

Note that as a pointed link $m(U)\simeq U$ so as pointed links 
\begin{equation*}
    m(U_{n}(a,b))\simeq U_{-n}(a,b).
\end{equation*}
Using this, we now apply these lemmas to Proposition \ref{prop: HFL of Kn(r,0)} to get the following corollary about certain Alexander gradings. Note that this corollary actually gives the more general cases of,
\begin{equation*}
    U_{n}(a+t,b+t) \qquad \text{and} \qquad U_{-n}(a+t,b+t),
\end{equation*}
as these will be useful for the cobordism calculations.

\begin{cor}
\label{cor: HFL of Kn(a+t,b+t)}
    Let $r=a+b$, $k=\frac{n}{2}(r-1)-ni-j$ and $\mathbf{k}=(k,\ldots,k,-k,\ldots,-k)\in \mathbb{Q}^{r+2t}$ where the first $a+t$ terms are $k$ and the last $b+t$ terms are $-k$. Then if $k\geq 0$ we have for all $t\geq 0$,
    \begin{equation}
    \label{eq: Kn(a+t,b+t)}
        \widehat{HFL}_{*}(U_{n}(a+t,b+t),\mathbf{k})\cong \bigoplus_{p=0}^{i+t}\mathbb{F}_{(nb-ni-2j)(i+1-b)-p}^{\binom{r+2t-1}{p}}\oplus\bigoplus_{p=0}^{i+t-\delta_{j,0}}\mathbb{F}_{(nb-ni-2j)(i+1-b)-r+2+p-2t}^{\binom{r+2t-1}{p}}
    \end{equation}
    \begin{equation}
    \label{eq: K-n(a+t,b+t)}
        \widehat{HFL}_{*}(U_{-n}(a+t,b+t),\mathbf{k})\cong \bigoplus_{p=0}^{i+t}\mathbb{F}_{-(na-ni-2j)(i+1-a)-r+1+p-2t}^{\binom{r+2t-1}{p}}\oplus\bigoplus_{p=0}^{i+t-\delta_{j,0}}\mathbb{F}_{-(na-ni-2j)(i+1-a)-p-1}^{\binom{r+2t-1}{p}}.
    \end{equation}
    Further, if $k<0$ then the above equations hold but with $i+t$ replaced with $r-2-i-t$ in the summands, and the $-\delta_{j,0}$'s in the second summand are changed to a $+\delta_{j,0}$'s on the first summand.
\end{cor}

\begin{proof}
     We first prove the isomorphism in equation \eqref{eq: Kn(a+t,b+t)} in the case $t=0$ using Proposition \ref{prop: HFL of Kn(r,0)} and Lemma \ref{lem: orientation reversing}. To apply these we need to compute $2\sum_{p=0}^{b-1} \mathbf{k}_{r-p}$ and $\sum_{p=0}^{b-1} l_{r-p}$, where 
     \begin{equation*}
         l_{r-p}=lk\big(U_{n}(r-p,p)\backslash L_{r-p},L_{r-p}\big)
     \end{equation*}
     and $L_{r-p}$ is the $(r-p)$-th component. First, we note that $\mathbf{k}_{p}=k$ for all $p$ so 
     \begin{equation*}
         2\sum_{p=0}^{b-1} \mathbf{k}_{r-p}=2bk=bn(r-1)-2nbi-2bj.
     \end{equation*}
     Further, we have that $l_{r-p}=n((r-p-1)-(p))$ so $\sum_{p=0}^{b-1} l_{r-p}=nab$, by substituting $r=a+b$. Finally, we note that $k=\frac{n}{2}(r-1)-ni-j$ so using equation \eqref{eq: Kn(r,0) i small} or \eqref{eq: Kn(r,0) i large}, as appropriate, Lemma \ref{lem: orientation reversing} and the two above sums we obtain equation \eqref{eq: Kn(a+t,b+t)}, for the case $t=0$. To obtain the general case we simply substitute $a'=a+t$, $b'=b+t$, $r'=r+2t$ and $i'=i+t$ into the equation and obtain equation \eqref{eq: Kn(a+t,b+t)}. Note that we are using $k=\frac{n}{2}(r-1)-ni-j$ so to use the case $t=0$ we need to take $k=\frac{n}{2}(r'-1)-ni'-j$, which is why we take $i'=i+t$.

     To prove the isomorphism in equation \eqref{eq: K-n(a+t,b+t)} we note that $2|\mathbf{k}|=2(a-b)k$ and that $l=r+2t$. Now applying Lemma \ref{lem: mirror} to equation \eqref{eq: Kn(a+t,b+t)} and rearranging gives equation \eqref{eq: K-n(a+t,b+t)}.
\end{proof}

We now deal with the other Alexander multigradings that we will need to analyse. The proof of the following corollary is very similar to the above. The only additional detail is that we have some more complicated conditions on the Alexander multigradings.

\begin{cor}
\label{cor:HFL of Kn(a+t,b+t) non meeting}
    Let $r=a+b$, $k=\frac{n}{2}(r-1)-ni-j$ and $\mathbf{k}_0\in \{k,k-1,-k,-k+1\}^{r}$, where $p$-coordinates, with $1\leq p\leq r-1$, are equal to $k-1$ or $-k+1$. Further, let $q$ be the number of $-k$ and $k-1$ coordinates in $\mathbf{k}_0$. Now define $\mathbf{k}_t\in \{k,k-1,-k,-k+1\}^{r+2t}$ such that the coordinates $(\mathbf{k}_0)_s=(\mathbf{k}_t)_s$ for $1\leq s\leq a$, $(\mathbf{k}_0)_s=(\mathbf{k}_t)_{s+t}$ for $a+1\leq s \leq a+b$, and for all $1\leq s \leq t$ we have $(\mathbf{k}_t)_{a+s}=-(\mathbf{k}_t)_{a+t+b+s}$. Then we have the following isomorphisms for all $t\geq 0$:
    \begin{align}
        \widehat{HFL}(U_{n}(a+t,b+t),\mathbf{k}_t)&\cong \mathbb{F}^{\binom{r+2t-2}{i}}_{(nb-ni-2j)(i+1-b)+b-i-q-t} \label{eq: Kn(a,b) not a meeting point} \\
         \widehat{HFL}(U_{-n}(a+t,b+t),\mathbf{k}_t)&\cong \mathbb{F}^{\binom{r+2t-2}{i}}_{-(na-ni-2j)(i+1-a)+1-a+i-q-t} \label{eq: K-n(a,b) not a meeting point}
    \end{align}
\end{cor}

\begin{proof}
    Again we first deal with the case $t=0$. First, define $p_{1}$ and $p_2$ to be the number of $k-1$ and $-k+1$ terms in $\mathbf{k}_0$ respectively. Then we note that $p=p_{1}+p_{2}$ and $q=p_{1}+b-p_{2}$. We now, as in the previous proof, compute $2\sum_{s=0}^{b-1} (\mathbf{k}_0)_{r-s}$ and $\sum_{s=0}^{b-1} l_{r-s}$. As previously proven we have $\sum_{s=0}^{b-1} l_{r-s}=nab$. Further, we have 
    \begin{equation*}
        2\sum_{s=0}^{b-1} (\mathbf{k}_0)_{r-s}=2(k-1)p_{2}+2k(b-p_{2})=2bk-2p_2.
    \end{equation*}
    Therefore, applying Lemma \ref{lem: orientation reversing} to equation \eqref{eq: Kn(r,0) not a meeting point} we get that the Maslov grading is,
    \begin{equation*}
        (nb-ni-2j)(i+1-b)-i-p+2p_{2},
    \end{equation*}
    but note that $-p+2p_2=-p_1+p_2=-q+b$. This proves the isomorphism in equation \eqref{eq: Kn(a,b) not a meeting point}, for the case $t=0$. Finally, we have that $2|\mathbf{k}_0|=2(a-b)k+2p_2-2p_1$ and using this, Lemma \ref{lem: orientation reversing} and equation \eqref{eq: Kn(a,b) not a meeting point} for the case $t=0$ we get the isomorphism in equation \eqref{eq: K-n(a,b) not a meeting point}, for the case $t=0$.

    To prove the general case we first define $q_t$ be the number of $-k$ and $k-1$ coordinates in $\mathbf{k}_t$. We claim that $q_t=q+t$. To see this note that we have the $q$, $-k$ and $k-1$ terms coming from the fact that the coordinates $(\mathbf{k}_0)_s=(\mathbf{k}_t)_s$ for $1\leq s\leq a$ and $(\mathbf{k}_0)_s=(\mathbf{k}_t)_{s+t}$ for $a+1\leq s \leq a+b$. Further, for all $1\leq s \leq t$ the coordinate $(\mathbf{k}_t)_{a+s}$ either contributes to the count or $(\mathbf{k}_t)_{a+t+b+s}$ does as $(\mathbf{k}_t)_{a+s}=-(\mathbf{k}_t)_{a+t+b+s}$. Now simply substitute $a'=a+t$, $b'=b+t$, $r'=r+2t$, $i'=i+t$ and $q'=q_t$ into the equations as in the previous proof. This then gives the general form of the above equations.
\end{proof}

Combining these two corollaries together gives a complete description of both 
\begin{equation*}
    \widehat{HFL}(U_{n}(a+t,b+t)) \qquad \text{and} \qquad \widehat{HFL}(U_{-n}(a+t,b+t)),
\end{equation*}
as bi-graded vector spaces for all $n>0$. We will use these calculations in the following section. First though, we compute $\widehat{HFL}(K_{-n}(a,b),\mathbf{0})$ for $K$ a negative L-space knot that is not the unknot, $n\geq2g(K)$, and $a+b$ odd. To do this we recall the following result from \cite[Corollary 1.3]{LSpaceAlexMonic}.

\begin{prop}
\label{prop: alex poly of L-space knots}
    Let $K$ be a negative L-space knot and $\Delta_{K}(t)$ be the Alexander polynomial. Then there is a sequence of positive integers $l_{1}<\ldots< l_{d}$ such that,
    \begin{equation}
        \Delta_{K}(t)=(-1)^d+\sum_{j=1}^d (-1)^{d-j}(t^{l_{j}}+t^{-l_{j}}).
    \end{equation}
\end{prop}

The following calculations will depend on the parity of $d$ so we define $\delta_d=1$ when $d$ is odd and $\delta_d=0$ when $d$ is even. Also recall the classical fact that $\Delta_K(t)=\Delta_{m(K)}(t)$. Then we have the following.

\begin{lem}
\label{lem: Kn(a,b) odd}
    Let $K$ be a negative L-space knot that is not the unknot. Then for all $n\geq 2g(K)$ and all $a,b\geq 0$ such that $a+b=2s+1$ with $s\geq0$, we have the following isomorphism,
    \begin{equation}
    \label{eq: K_n(a,b) general}
        \widehat{HFL}(K_{-n}(a,b),\mathbf{0})\cong\bigoplus_{p=0}^{s-\delta_d}\mathbb{F}_{\frac{n}{4}((a-b)^2-1)-2s+p+2c_K(0)}^{\binom{2s}{p}}\oplus \bigoplus_{p=0}^{s-1+\delta_d}\mathbb{F}_{\frac{n}{4}((a-b)^2-1)-1-p+2c_K(0)}^{\binom{2s}{p}}
    \end{equation}
    where $c_K(0)=\sum_{p=1}^d(-1)^{d-p}l_p$.
\end{lem}

\begin{proof}
    We again use \cite[Theorem 3]{linkLspacernrm} and let $r=a+b$, then we are interested in the polynomial,
    \begin{equation}
        q(t)=\sum_{k\in\mathbb{Z}} \mathbf{h}(k)t^k=\frac{t^{-1}\Delta_{K}(t)(t^{\frac{nr}{2}}-t^{-\frac{nr}{2}})}{(1-t^{-1})^2(t^{\frac{n}{2}}-t^{-\frac{n}{2}})}.
    \end{equation}
    From the fact that $l_d\leq g(K)\leq \frac{n}{2}$ we have the following inequality,
    \begin{equation*}
        l_d-\frac{n}{2}(r+1)-1\leq l_d-n-1\leq -g(K)-1<-1.
    \end{equation*}
    The last inequality holds as it can only be an equality if $g(K)=0$, which by a classical result implies $K$ is the unknot. As we are trying to determine Alexander grading $\mathbf{0}$ we only need to consider the polynomial in degrees $1,0$ and $-1$. Therefore, using similar arguments to the proof of Proposition \ref{prop: HFL of Kn(r,0)} we can reduce to considering the following polynomial,
    \begin{equation}
    \label{eq: any knot simp poly}
        p_{r}(t)=\Delta_{K}(t)\sum_{i=0}^{\infty}\sum_{j=0}^{n-1}\left(\frac{ni(i+1)}{2}+(i+1)j\right)t^{\frac{n(r-1)}{2}-ni-j}.
    \end{equation}
    
    We now need to extract the function $\mathbf{h}\colon \mathbb{Z}\to \mathbb{Z}$ for the cases $k \in \{-1,0,1\}$ from this. To do this we need to find $i_{\pm p,k}\geq 0$ and $0\leq j_{\pm p,k}\leq n-1$, for $0\leq p\leq d$ and $k \in \{-1,0,1\}$, satisfying,
    \begin{equation}
    \label{eq: coefficients}
        n(s-i_{\pm p,k})-j_{\pm p,k}\pm l_p=k
    \end{equation}
    where we have substituted $r=2s+1$ and $l_0=0$. There are two cases, when $n=2$ and when $n>2$. These are the only cases as otherwise $1=n\geq2g(K)$ so $g(K)=0$, which again implies $K$ is the unknot. Further, note that the case $n=2$ is equivalent to $\Delta_{K}(t)=t-1+t^{-1}$ by Proposition \ref{prop: alex poly of L-space knots}, which is the Alexander polynomial of the trefoil. However, it was shown in \cite[Theorem 1.2]{LSpaceAlexMonic} that $\widehat{HFK}(K)$ of L-space knots is completely determined by the Alexander polynomial, and in \cite[Corollary 1.5]{HFKdetectstrefoil} it is shown that $\widehat{HFK}(K)$ detects the trefoil. Hence, the only L-space knot with $\Delta_{K}(t)=t-1+t^{-1}$ is the trefoil and this is the only knot in the case $n=2$.

    We start with the case $n\neq2$ and $s\neq 0$. Then solving the equation \eqref{eq: coefficients} for $k=0$ gives for $p\geq 0$, 
    \begin{equation*}
        j_{p,0}=l_p \qquad \text{and} \qquad i_{p,0}=s,
    \end{equation*}
    and for $p> 0$ we have,
    \begin{equation*}
        j_{-p,0}=n-l_p \qquad \text{and} \qquad i_{-p,0}=s-1.
    \end{equation*}
    For $k=1$ we get for $p>0$,
    \begin{equation*}
        j_{p,1}=l_p-1 \qquad \text{and} \qquad i_{p,1}=s,
    \end{equation*}
    and for $p\geq 0$,
    \begin{equation*}
        j_{-p,1}=n-1-l_p\qquad \text{and} \qquad i_{-p,1}=s-1.
    \end{equation*}
    Finally, for $k=-1$ we get for $p\geq 0$,
    \begin{equation*}
        j_{p,-1}=l_p+1 \qquad \text{and} \qquad  i_{p,-1}=s,
    \end{equation*}
    for $p>0$ and $l_1\neq 1$ we get,
    \begin{equation*}
        j_{-p,-1}=n-1-l_p\qquad \text{and} \qquad  i_{-p,-1}=s-1,
    \end{equation*}
    and if $l_1=1$ we get,
    \begin{equation*}
        j_{-1,-1}=0\qquad \text{and} \qquad i_{-1,-1}=s.
    \end{equation*}
    From this we get,
    \begin{align}
        \mathbf{h}(0)&=(-1)^d\left( \frac{ns(s+1)}{2}  \right)+\sum_{p=1}^{d}(-1)^{d-p}\left( \frac{ns(s+1)}{2}+l_p(s+1) + \frac{ns(s-1)}{2} +(n-l_p)s \right) \nonumber\\
        &=(-1)^d\left( \frac{ns(s+1)}{2} \right)+ ns(s+1)\sum_{p=1}^{d}(-1)^{d-p} + \sum_{p=1}^{d}(-1)^{d-p}l_p \\
        &= \frac{ns(s+1)}{2} +c_{K}(0), \nonumber
    \end{align}
    where to get the final equality we use that $\sum_{p=1}^{d}(-1)^{d-p}$ equals $0$ if $d$ is even and $1$ if $d$ is odd. Similarly we can obtain the following equalities,
    \begin{align}
        \mathbf{h}(1)&= \frac{ns(s+1)}{2} -s +c_{K}(1),\\
        \mathbf{h}(-1)&= \frac{ns(s+1)}{2} +s +c_{K}(-1),
    \end{align}
    where $c_{K}(1)=\sum_{p=1}^d(-1)^{d-p}(l_p-1)$ and $c_{K}(-1)=\sum_{p=0}^d(-1)^{d-p}(l_p+1)$. Further, it is useful to note that, $c_{K}(1)=c_{K}(0)$, $c_{K}(-1)=c_{K}(0)+1$ if $d$ is even and $c_{K}(1)=c_{K}(0)-1$, $c_{K}(-1)=c_{K}(0)$ if $d$ is odd. We can then compute that,
    \begin{equation*}
        \beta(0)=\mathbf{h}(-1)-\mathbf{h}(0)-1=s-1+c_{K}(-1)-c_{K}(0)=s-1+(-1)^{d}+\sum_{p=1}^d(-1)^{d-p}.
    \end{equation*}
    Therefore, $\beta(0)=s-1$ for $d$ odd and $\beta(0)=s$ for $d$ even. Similarly, we get $\beta(1)=s$ for $d$ odd and $\beta(1)=s-1$ for $d$ even. Then \cite[Theorem 3]{linkLspacernrm} gives,
    \begin{equation}
    \label{eq: K_n(r,0) general}
        \widehat{HFL}(m(K)_{n}(r,0),\mathbf{0})\cong\bigoplus_{p=0}^{s-\delta_d}\mathbb{F}_{-ns(s+1)-2c_K(0)-p}^{\binom{2s}{p}}\oplus \bigoplus_{p=0}^{s-1+\delta_d}\mathbb{F}_{-ns(s+1)-2c_K(0)+1-2s+p}^{\binom{2s}{p}}
    \end{equation}
    All of the above arguments work for the case of $s=0$ and the descriptions of $\mathbf{h}$ and $\beta$ are just given by substituting $s=0$ into the ones above.

    We now deal with the case $n=2$, which as previously mentioned is equivalent to $K$ being the trefoil. In this case the equation we are interested in is,
    \begin{equation*}
        2(s-i_{\pm p,k})-j_{\pm p,k}\pm l_p=k,
    \end{equation*}
    for $p=0$ and $p=1$. The computations are then similar and use that $c_{K}(0)=1$, $c_{K}(1)=0$ and $c_{K}(-1)=1$. Then we get $\beta(0)=s-1$ and $\beta(1)=s$, which agrees with the case where $d$ is odd.

    To obtain equation \eqref{eq: K_n(a,b) general} we apply Lemmas \ref{lem: orientation reversing} and \ref{lem: mirror} in a similar way to in the proof of Corollary \ref{cor: HFL of Kn(a+t,b+t)}. To apply Lemma \ref{lem: orientation reversing} we compute $\sum_{s=0}^{b-1} l_{r-s}=nab$ and note that the coordinates are all $0$. Combining these and rearranging gives equation \eqref{eq: K_n(a,b) general}.
\end{proof}

The proofs of the next three lemmas are nearly identical to the above so we omit the full proofs but we do explain the differences. The following lemma is for a middle step in an inductive argument later, and is similar to the above but with $r$ even.

\begin{lem}
\label{lem: Kn(a,b) even}
    Let $K$ be a negative L-space knot that is not the unknot. Then for all $n\geq 2g(K)$ and all $a,b\geq 0$ such that $a+b=2s$ with $s\geq1$, we have the following isomorphism,
    \begin{equation}
    \label{eq: K_n(a,b) even general}
        \widehat{HFL}(K_{-n}(a,b),\mathbf{k}_{a,b})\cong\bigoplus_{p=0}^{s-1-\delta_d}\mathbb{F}_{\frac{n}{4}((a+1-b)^2-1)+1-2s+p+2c_K(0)}^{\binom{2s-1}{p}}\oplus \bigoplus_{p=0}^{s-2+\delta_d}\mathbb{F}_{\frac{n}{4}((a+1-b)^2-1)-1-p+2c_K(0)}^{\binom{2s-1}{p}}
    \end{equation}
    where $\mathbf{k}_{a,b}=(\frac{n}{2},\ldots,\frac{n}{2},-\frac{n}{2},\ldots,-\frac{n}{2})$ with $b$ negative terms.
\end{lem}

\begin{proof}
    The whole beginning of the proof of Lemma \ref{lem: Kn(a,b) odd} works here with the only change being that equation \eqref{eq: coefficients} becomes $n(s-1-i_{\pm p,k})-j_{\pm p,k}\pm l_p=k$. Using this we get that, 
    \begin{align}
        \mathbf{h}\left(\frac{n}{2}\right)&= \frac{ns(s-1)}{2}  +c_{K}(0), \\
        \mathbf{h}\left(\frac{n}{2}+1\right)&= \frac{ns(s-1)}{2} -s+1 +c_{K}(1),\\
        \mathbf{h}\left(\frac{n}{2}-1\right)&= \frac{ns(s-1)}{2} +s -1+c_{K}(-1).
    \end{align}
    Again in the above there are subtleties in the cases when $n=2$ or $s=1$. However, as in the previous proof they give the same result in the end. Then substituting these into the definition of $\beta$ gives: $\beta(\frac{n}{2})=s-2$ for $d$ odd and $\beta(\frac{n}{2})=s-1$ for $d$ even; and $\beta(\frac{n}{2}+1)=s-1$ for $d$ odd and $\beta(\frac{n}{2}+1)=s-2$ for $d$ even. A difference appears in the applications of Lemma \ref{lem: orientation reversing} and \ref{lem: mirror}. We need $2\sum_{p=0}^{b-1} \mathbf{k}_{r-p}=nb$, where $\mathbf{k}=(\frac{n}{2},\ldots, \frac{n}{2})$. Also after reversing the orientations we have that $2|\mathbf{k}_{a,b}|=n(a-b)$. Combining these gives the desired result.
\end{proof}

The next two lemmas are the equivalent of Lemma \ref{cor:HFL of Kn(a+t,b+t) non meeting} but for general negative L-space knots. This first one is needed in the computation of the spectral sequence for forgetting components.

\begin{lem}
\label{lem: Kn(a,b) nonmeeting general}
    Let $K$ be a negative L-space knot that is not the unknot. Then for all $n\geq 2g(K)$ and all $a,b\geq 0$ such that $a+b=2s$ with $s\geq1$, we have the following isomorphisms,
    \begin{align}
        \widehat{HFL}(K_{-n}(a+1,b),\mathbf{e}_1)&\cong \mathbb{F}_{\frac{n}{4}((a+1-b)^2-1)-s-\delta_d+2c_{K}(0)+1}^{\binom{2s-1}{s-1+\delta_d}}, \label{eq: e1 for spec seq}\\
        \widehat{HFL}(K_{-n}(a+1,b),-\mathbf{e}_1)&\cong \mathbb{F}_{\frac{n}{4}((a+1-b)^2-1)-s-\delta_d+2c_{K}(0)-1}^{\binom{2s-1}{s-\delta_d}}, \label{eq: -e1 for spec seq}\\
        \widehat{HFL}(K_{-n}(a,b),\mathbf{e}_{a+1}+\mathbf{k}_{a,b})&\cong \mathbb{F}_{\frac{n}{4}((a+1-b)^2-1)-s-\delta_d+2c_{K}(0)+1}^{\binom{2s-2}{s-1-\delta_d}},\label{eq: ea+1 for spec seq}\\
        \widehat{HFL}(K_{-n}(a,b),-\mathbf{e}_{a+1}+\mathbf{k}_{a,b})&\cong \mathbb{F}_{\frac{n}{4}((a+1-b)^2-1)-s-\delta_d+2c_{K}(0)}^{\binom{2s-2}{s-2+\delta_d}}, \label{eq: -ea+1 for spec seq}
    \end{align}
    where as before $\mathbf{k}_{a,b}=(\frac{n}{2},\ldots,\frac{n}{2},-\frac{n}{2},\ldots,-\frac{n}{2})$ with $b$ negative terms, $\mathbf{e}_k=(0,\ldots,1,\ldots,0)$ with the $1$ in the $k$-th position.
\end{lem}

\begin{proof}
    A very similar proof to that of equation \eqref{eq: Kn(r,0) not a meeting point} combined with Corollary \ref{cor:HFL of Kn(a+t,b+t) non meeting} works here. First we have,
    \begin{equation}
        \widehat{HFL}(m(K)_{n}(a+1+b,0),\mathbf{e}_1)\cong \mathbb{F}_{-2\mathbf{h}(1)-\beta(1)-2s}^{\binom{2s-1}{\beta(1)}}.
    \end{equation}
    The computations from the proof of Lemma \ref{lem: Kn(a,b) odd} then give $\mathbf{h}(1)$ and $\beta(1)$. Then we apply Lemma \ref{lem: orientation reversing} and use that $\sum_{s=0}^{b-1} l_{r-s}=nab$. Finally, applying Lemma \ref{lem: mirror} gives equation \eqref{eq: e1 for spec seq} after using that $2|\mathbf{e}_1|=2$ and rearranging. The next equation comes from,
    \begin{equation}
        \widehat{HFL}(m(K)_{n}(a+1+b,0),-\mathbf{e}_1)\cong \mathbb{F}_{-2\mathbf{h}(0)-\beta(0)-1}^{\binom{2s-1}{\beta(0)}}.
    \end{equation}
    and the proof then follows identically, except that $2|-\mathbf{e}_1|=-2$.

    Equation \eqref{eq: ea+1 for spec seq} is similar but with a few key differences. We start with,
    \begin{equation}
        \widehat{HFL}(m(K)_{n}(a+b,0),-\mathbf{e}_{a+1}+\mathbf{k})\cong \mathbb{F}_{-2\mathbf{h}(\frac{n}{2})-\beta(\frac{n}{2})-1}^{\binom{2s-2}{\beta(\frac{n}{2})}}.
    \end{equation}
    Then use the computations of $\mathbf{h}(\frac{n}{2})$ and $\beta(\frac{n}{2})$ from the proof of Lemma \ref{lem: Kn(a,b) even}. Then again apply Lemmas \ref{lem: orientation reversing} and \ref{lem: mirror} using the computations,
    \begin{equation*}
        2\sum_{p=0}^{b-1} (\mathbf{k}-\mathbf{e}_{a+1})_{r-p}=nb-2, \qquad \sum_{s=0}^{b-1} l_{r-s}=nab \qquad \text{and} \qquad 2|\mathbf{k}_{a,b}+\mathbf{e}_{a+1}|=n(a-b)+2.
    \end{equation*}
    Finally equation \eqref{eq: -ea+1 for spec seq} comes from,
    \begin{equation}
        \widehat{HFL}(m(K)_{n}(a+b,0),\mathbf{e}_{a+1}+\mathbf{k})\cong \mathbb{F}_{-2\mathbf{h}(\frac{n}{2}+1)-\beta(\frac{n}{2}+1)-2s+1}^{\binom{2s-2}{\beta(\frac{n}{2}+1)}},
    \end{equation}
    and then a similar argument. Note that for all of the equations part of the computation uses $\beta$, which varies depending on the parity of $d$. These are the source of the $\delta_d$'s appearing in the lemma.
\end{proof}

This next lemma will be needed to compute the pair of pants cobordism maps.

\begin{lem}
\label{lem: Kn(a,b) nonmeeting general 2}
    Let $K, n, a$ and $b$ be as above with the additional assumption that $a,b>0$. Then we have the following isomorphisms:
    \begin{align}
        \widehat{HFL}(K_{-n}(a+1,b),\mathbf{e}_{a+1}-\mathbf{e}_{a+b+1})&\cong \mathbb{F}_{\frac{n}{4}((a+1-b)^2-1)-s-\delta_d+2c_{K}(0)}^{\binom{2s-1}{s-1+\delta_d}}, \label{eq: ea-eab1}\\
        \widehat{HFL}(K_{-n}(a+1,b),-\mathbf{e}_{a+1}+\mathbf{e}_{a+b+1})&\cong \mathbb{F}_{\frac{n}{4}((a+1-b)^2-1)-s-\delta_d+2c_{K}(0)}^{\binom{2s-1}{s-\delta_d}}, \label{eq: -ea+eab1} \\
        \widehat{HFL}(K_{-n}(a,b),-\mathbf{e}_{a}+\mathbf{e}_{a+b}+\mathbf{k}_{a,b})&\cong \mathbb{F}_{\frac{n}{4}((a+1-b)^2-1)-s-\delta_d+2c_{K}(0)}^{\binom{2s-2}{s-1-\delta_d}},  \label{eq: -ea+eab+k}\\
        \widehat{HFL}(K_{-n}(a,b),\mathbf{e}_{a}-\mathbf{e}_{a+b}+\mathbf{k}_{a,b})&\cong \mathbb{F}_{\frac{n}{4}((a+1-b)^2-1)-s-\delta_d+2c_{K}(0)+1}^{\binom{2s-2}{s-2+\delta_d}},\label{eq: ea-eab+k}
    \end{align}
\end{lem}

\begin{proof}
    We use a very similar proof to that of the above lemma with a few key caveats. For equation \eqref{eq: ea-eab1} we start with,
    \begin{equation}
        \widehat{HFL}(m(K)_{n}(a+1+b,0),\mathbf{e}_{a+1}+\mathbf{e}_{a+b+1})\cong \mathbb{F}_{-2\mathbf{h}(1)-\beta(1)-2s+1}^{\binom{2s-1}{\beta(1)}}.
    \end{equation}
    The sum of the linking numbers is the same and we have to subtract 2 in the use of Lemma \ref{lem: orientation reversing} because of the $\mathbf{e}_{a+b+1}$. Then in the use of Lemma \ref{lem: mirror} we have $|\mathbf{e}_{a+1}-\mathbf{e}_{a+b+1}|=0$. Combining these gives the result. Then for equation \eqref{eq: -ea+eab1} we start with,
    \begin{equation}
        \widehat{HFL}(m(K)_{n}(a+1+b,0),-\mathbf{e}_{a+1}-\mathbf{e}_{a+b+1})\cong \mathbb{F}_{-2\mathbf{h}(0)-\beta(0)-2}^{\binom{2s-1}{\beta(0)}}.
    \end{equation}
    A nearly identical proof then gives the result, the only difference being we add 2 in the use of Lemma \ref{lem: orientation reversing}.
    
    For equation \eqref{eq: -ea+eab+k} the proof is very similar. We start with,
    \begin{equation}
        \widehat{HFL}(m(K)_{n}(a+b,0),-\mathbf{e}_{a}-\mathbf{e}_{a+b}+\mathbf{k})\cong \mathbb{F}_{-2\mathbf{h}(\frac{n}{2})-\beta(\frac{n}{2})-2}^{\binom{2s-2}{\beta(\frac{n}{2})}}.
    \end{equation}
    Then we apply Lemmas \ref{lem: orientation reversing} and \ref{lem: mirror} using,
    \begin{equation*}
        2\sum_{p=0}^{b-1} (\mathbf{k}-\mathbf{e}_a-\mathbf{e}_{a+b})_{r-p}=nb-2, \qquad \sum_{s=0}^{b-1} l_{r-s}=nab \qquad \text{and} \qquad 2|\mathbf{k}_{a,b}-\mathbf{e}_a+\mathbf{e}_{a+b}|=n(a-b).
    \end{equation*}
    Finally, equation \eqref{eq: ea-eab+k} comes from
    \begin{equation}
        \widehat{HFL}(m(K)_{n}(a+b,0),\mathbf{e}_{a}+\mathbf{e}_{a+b}+\mathbf{k})\cong \mathbb{F}_{-2\mathbf{h}(\frac{n}{2}+1)-\beta(\frac{n}{2}+1)-2s+2}^{\binom{2s-2}{\beta(\frac{n}{2}+1)}},
    \end{equation}
    and then a very similar argument. As in the previous lemma the $\delta_d$'s come from the appearance of the $\beta$ in the above.
\end{proof}

The final lemma of this section is needed in the proof of the main theorem.

\begin{lem}
\label{lem: Kn(a,b) nonmeeting general 3}
    Let $K$ and $n$ be as in the previous lemmas and let $r=2s+1$, $t\geq0$. Further, let $\mathbf{k}\in \{-1,0,1\}^{r+2t}$ such that $\mathbf{k}_i\in\{0,1\}$ for $ i \leq r+t$, $\mathbf{k}_i\in \{-1,0\}$ for $i>r+t$ and $|\mathbf{k}|=0$. Then we have the following isomorphisms,
    \begin{align}
        \widehat{HFL}(K_{-n}(r+t,t),\mathbf{k})&\cong \mathbb{F}^{\binom{2(s+t)-1}{s+t-1+\delta_d}}_{n(s^2+s)-s-t-\delta_d+2c_{K}(0)}, \label{eq: HFK 0 not meeting +ve} \\
        \widehat{HFL}(K_{-n}(r+t,t),-\mathbf{k})&\cong \mathbb{F}^{\binom{2(s+t)-1}{s+t-1+\delta_d}}_{n(s^2+s)-s-t-\delta_d+2c_{K}(0)}. \label{eq: HFK 0 not meeting -ve}
    \end{align}
\end{lem}

\begin{proof}
    The argument is similar to the previous two lemmas. We start by considering $\mathbf{k}'\in \{0,1\}$ such that $\mathbf{k}'_i=\mathbf{k}_i$ for $i\leq r+t$ and $\mathbf{k}'_i=-\mathbf{k}_i$ for $i> r+t$. Then by \cite[Theorem 3]{linkLspacernrm},
    \begin{equation*}
        \widehat{HFL}(m(K)_{n}(r+2t,0),\mathbf{k}')\cong \mathbb{F}^{\binom{2(s+t)-1}{\beta(1)}}_{-2\mathbf{h}(1)-\beta(1)-2s-2t-1+2q}.
    \end{equation*}
    We have computed $\mathbf{h}(1)$ and $\beta(1)$ in the proof of Lemma \ref{lem: Kn(a,b) odd}, although we substitute $(s+t)$ for $s$ in the formulas. Let $q$ be the number of 1's in $\mathbf{k}$. We then want to apply Lemma \ref{lem: orientation reversing} so we compute,
    \begin{equation*}
        \sum_{p=0}^{t-1} l_{r+2t-p}=n(2s+t+1)t \qquad \text{and} \qquad 2\sum_{p=0}^{t-1}\mathbf{k}'_{r+2t-p}=2q.
    \end{equation*}
    To apply Lemma \ref{lem: mirror} we note that $|\mathbf{k}|=0$. Combining all of these computations and rearranging gives equation \eqref{eq: HFK 0 not meeting +ve}. Then to get equation \eqref{eq: HFK 0 not meeting -ve} we simply apply Lemma \ref{lem: alexander symmetry}.  
\end{proof}

\section{Non-Vanishing of Floer Lasagna Modules}
\label{sec: Non-Vanishing of Floer Lasagna Modules}

We are now ready to prove the main theorem. We do this in the first subsection and then apply it to the case of the unknot in the final subsection.

\subsection{Proof of Main Theorem}
\label{subsec: Proof Main Theorem}
\hfill

In this subsection we will prove that for a negative L-space knot $K$ and an integer $n\geq2g(K)$ if a certain cobordism map, depending on the framing and parity of $d$, is an isomorphism then the vector space $\mathcal{FL}(X_{-n}(K))$ is infinite dimensional. Recall that $d$ is the number of positive exponents in the symmetrised Alexander polynomial of $K$. We start with a definition to set up notation.

\begin{defn}
\label{defn: Fab maps}
    Let $K$ be a negative L-space knot, $n \geq 2g(K)$ and $a,b\geq0$ such that $a+b=2s+1$ with $s\geq0$ for $d$ even and $s\geq2$ for d odd. Define,
    \begin{equation*}
        i_{a,b}:=\frac{n}{4}((a-b)^2-1)-2s+2c_{K}(0),
    \end{equation*}
    that is the minimal Maslov grading such that $\widehat{HFL}_{i}(K_{-n}(a,b),\mathbf{0})$ is non-zero.  Further, denote by
    \begin{equation*}
        F_{a,b}\colon \widehat{HFL}_{i_{a,b}-1}(K_{-n}(a,b)\sqcup U,\mathbf{0})\to \widehat{HFL}_{i_{a+1,b+1}}(K_{-n}(a+1,b+1),\mathbf{0}),
    \end{equation*}
    the restriction of the pair of pants maps $F_P$ defined in Definition \ref{defn: pair of pants cobordism}.
\end{defn}

To justify the Maslov grading in the above definition, note that $i_{a+1,b+1}=i_{a,b}-2$ and that $F_{a,b}$ drops the Maslov grading by one as noted after Definition \ref{defn: pair of pants cobordism}. Further, we can treat the second $\mathbf{0}$ as an element of $\mathbb{Z}^{a+b+2}$ by noting that by Lemma \ref{lem: Kn(a,b) nonmeeting general 2} the support of the Maslov grading in the other possible Alexander multigradings is strictly greater than $i_{a+1,b+1}$. We now prove a key inductive lemma relating $F_{a,b}$ to $F_{a+1,b+1}$.

\begin{lem}
\label{lem: inductive step}
    Let $K,n,a$ and $b$ be as above and suppose $s\geq 0$ if $d$ is even or $s\geq 2$ if $d$ is odd. Then $F_{a+1,b+1}$ is an isomorphism, if $F_{a,b}$ is an isomorphism.
\end{lem}

\begin{proof}
    We prove this in two stages. Define $i_{a,b+1}$ to be the minimal Maslov grading such that $\widehat{HFL}(K_{-n}(a,b+1),\mathbf{k}_{a,b+1})$ is non-zero, where recall $\mathbf{k}_{a,b+1}=(\frac{n}{2},\ldots, \frac{n}{2},-\frac{n}{2},\ldots,-\frac{n}{2})$ with $b+1$ negative terms. We first show that the map 
    \begin{equation*}
        F_{a,b+1}\colon \widehat{HFL}_{i_{a,b+1}}(K_{-n}(a,b+1)\sqcup U,(\mathbf{k}_{a,b+1},0))\to  \widehat{HFL}_{i_{a+1,b+2}}(K_{-n}(a+1,b+2),\mathbf{k}_{a+1,b+2}),
    \end{equation*}
    is an isomorphism, where the $0$ in $(\mathbf{k}_{a,b+1},0)$ corresponds to the unknotted component. Then, we show that the map in the lemma is an isomorphism. These isomorphisms will be a consequence of Proposition \ref{prop: spec maps for forgetting components of bands} and the calculations in Section \ref{sec: Link Floer Homology of T(r,rn)}. We write the proof out for $K$ not the unknot as the proof for the unknot is the same except using Corollaries \ref{cor: HFL of Kn(a+t,b+t)} and \ref{cor:HFL of Kn(a+t,b+t) non meeting} in place of the general results.

    Recall that the spectral sequence in Theorem \ref{thm: OS Spec seq} for forgetting a component of a link preserves the Maslov grading and shifts the Alexander grading by half the linking number. Consider this spectral sequence for the link $K_{-n}(a,b+1)\sqcup U$ forgetting the component $L_{a+1}$. Note that $lk(L_{a+1},L_{i})$ equals $n$ for $i\in\{1,\ldots,a\}$, $-n$ for $i\in\{a+2,\ldots,a+b+1\}$ and $lk(L_{a+1},U)=0$, so the spectral sequence is from,
    \begin{equation*}
        \bigoplus_{x\in\{-1,0,1\}}\widehat{HFL}(K_{-n}(a,b+1),(\mathbf{k}_{a,b+1},0)+x\mathbf{e}_{a+1})\qquad \text{to} \qquad \widehat{HFL}(K_{-n}(a,b),\mathbf{0})\otimes V,
    \end{equation*}
    where we are using \cite[Theorem 3]{linkLspacernrm} to limit the values of $x$. 

    From the assumptions of the lemma we have that $a+b+1\geq2$ if $d$ even or $a+b+1\geq6$ if $d$ is odd. Then using Lemmas \ref{lem: Kn(a,b) even} and \ref{lem: Kn(a,b) nonmeeting general} to limit the gradings of $\widehat{HFL}(K_{-n}(a,b+1))$ and Lemma \ref{lem:  unknot component} to understand the unknotted component we can see that the minimal Maslov grading of,
    \begin{equation*}
        \bigoplus_{x\in\{-1,0,1\}}\widehat{HFL}(K_{-n}(a,b+1)\sqcup U,(\mathbf{k}_{a,b+1},0)+x\mathbf{e}_{a+1})
    \end{equation*}
    is 1-dimensional and is supported in Maslov grading
    \begin{equation*}
        i_{a,b+1}-1=\frac{n}{4}((a-b)^2-1)-2-2s+2c_{K}(0).
    \end{equation*}
    Let $v_{a,b+1}\otimes B$ be the unique non-zero element in this grading, again using Lemma \ref{lem:  unknot component} for the tensor product. Again the minimal graded part of $\widehat{HFL}(K_{-n}(a,b)\sqcup U,\mathbf{0})$ is isomorphic to $\mathbb{F}$ and is supported in grading, 
    \begin{equation*}
        i_{a,b}-1=\frac{n}{4}((a-b)^2-1)-1-2s+2c_{K}(0)
    \end{equation*}
    so define $v_{a,b}\otimes B$ to be the unique non-zero element. Therefore, in this Maslov and Alexander grading the spectral sequence restricts to a spectral sequence $\mathbb{F}\implies \mathbb{F}$. Hence, $v_{a,b+1}\otimes B$ is sent to $v_{a,b}\otimes B\otimes B$.

    Now we consider the image of $F_{a,b+1}(v_{a,b+1}\otimes B)$. From Theorem \ref{thm: grading change formulas} we see that the Maslov grading satisfies,
    \begin{equation*}
        M(F_{a,b+1}(v_{a,b+1}\otimes B))=i_{a+1,b+2}=\frac{n}{4}((a-b)^2-1)-3-2s+2c_{K}(0),
    \end{equation*}
    and the collapsed Alexander grading is $(\mathbf{k}_{a,b+1},0)$, where we have collapsed the gradings coming from the two components in the pair of pants cobordism. Using the calculations in Lemma \ref{lem: Kn(a,b) nonmeeting general 2} we see that it is supported in Alexander grading $\mathbf{k}_{a+1,b+2}$. In fact by using a similar argument to the previous paragraph we have,
    \begin{equation*}
        F_{a,b+1}(v_{a,b+1}\otimes B)=v_{a+1,b+2}\qquad \text{or} \qquad 0,
    \end{equation*}
    where $v_{a+1,b+2}$ is the minimal Maslov graded element in $\widehat{HFL}(K_{-n}(a+1,b+2),\mathbf{k}_{a+1,b+2})$. 

    A similar argument to the above shows that the spectral sequence, 
    \begin{equation*}
        \bigoplus_{x\in\{-1,0,1\}}\widehat{HFL}(K_{-n}(a+1,b+2),\mathbf{k}_{a+1,b+2}+x\mathbf{e}_{a+1})\qquad \text{to} \qquad \widehat{HFL}(K_{-n}(a+1,b+1),\mathbf{0})\otimes V,
    \end{equation*}
    restricted to the minimal Maslov grading is $\mathbb{F}\implies \mathbb{F}$. Further, using the calculations in Section \ref{sec: Link Floer Homology of T(r,rn)} we see that the $E^1$-page is generated by $v_{a+1,b+2}$ and the $E^\infty$-page by $v_{a+1,b+1}$, where again $v_{a+1,b+1}$ is the element with minimal Maslov grading in $\widehat{HFL}(K_{-n}(a+1,b+1),\mathbf{0})$.

    We will now apply Corollary \ref{cor: spec maps for pants} to prove that $F_{a,b+1}$ is an isomorphism. To start we let,
    \begin{equation*}
        F\colon \widehat{CFL}_{z_{a+1}}(K_{-n}(a,b+1)\sqcup U) \to \widehat{CFL}_{z_{a+1}} (K_{-n}(a+1,b+2))
    \end{equation*}
    denote the morphism of transitive systems coming from the pair of pants cobordism, where $z_{a+1}$ is the basepoint corresponding to $L_{a+1}$. Further, let $\nu_{a,b+1}$ and $\nu_{a+1,b+2}$ be representatives of $v_{a,b+1}\otimes B$ and $v_{a+1,b+2}$ respectively at the chain level. Finally, let,
    \begin{align*}
        g_{a,b+1}&\colon \widehat{HFL}_{z_{a+1}}(K_{-n}(a,b+1)\sqcup U)\to \widehat{HFL}(K_{-n}(a,b)\sqcup U)\otimes V, \\
        g_{a+1,b+2}&\colon \widehat{HFL}_{z_{a+1}}(K_{-n}(a+1,b+2))\to \widehat{HFL}(K_{-n}(a+1,b+1))\otimes V,
    \end{align*}
    be the canonical isomorphisms from Proposition \ref{prop: canonical identification}. Then applying Corollary \ref{cor: spec maps for pants} gives,
    \begin{align*}
        F_{*}([\nu_{a,b+1}])=g_{a+1,b+2}^{-1} (F_{a,b}\otimes \text{Id}_{V}) g_{a,b+1}([\nu_{a,b+1}])&=g_{a+1,b+2}^{-1}(F_{a,b}\otimes \text{Id}_{V})((v_{a,b}\otimes B)\otimes B)\\
        &=g_{a+1,b+2}^{-1}(v_{a+1,b+1}\otimes B)=[\nu_{a+1,b+2}],
    \end{align*}
    where $[\cdot]$ means the class in homology and the equality between the lines comes from the assumption in the Lemma that $F_{a,b}(v_{a,b}\otimes B)=v_{a+1,b+1}$.

    Note that in this grading both the domain and codomain of $F_*$ are 1-dimensional so by Lemma \ref{lem: filtration and associated grading in single degree} there is a canonical identification between the homology and the associated graded vector space. Hence, in these gradings we have
    \begin{equation}
        E^{\infty}(F)=F_{*},
    \end{equation}
    where we are using Lemma \ref{lem: f_* and E^infty (f)} and slightly abusing notation. Now, above we showed that in this grading the $E^{1}$-page and $E^{\infty}$-page are isomorphic, say, by an isomorphism $\phi$, for both the target and source spectral sequences. Therefore,
    \begin{equation*}
        F_{a,b+1}(v_{a,b+1}\otimes B)=E^{1}(F)(v_{a,b+1}\otimes B)=\phi^{-1}F_*((v_{a,b}\otimes B)\otimes B)=\phi^{-1}(v_{a+1,b+1}\otimes B)=v_{a+1,b+2}.
    \end{equation*}
    This proves the claim.

     We now prove the map $F_{a+1,b+1}$ is an isomorphism using the previous claim. As before we start by considering the spectral sequence for the link $K_{-n}(a+1,b+1)\sqcup U$ forgetting the component $L_{1}$. Then note, $lk(L_{1},L_{i})$ equals $-n$ for $i\in\{2,\ldots,a+1\}$, $n$ for $i\in\{a+2,\ldots,a+b+2\}$ and $lk(L_1,U)=0$. Therefore as before the spectral sequence is from,
     \begin{equation*}
         \bigoplus_{x\in\{-1,0,1\}}\widehat{HFL}(K_{-n}(a+1,b+1)\sqcup U,x\mathbf{e}_{1}) \qquad \text{to} \qquad \widehat{HFL}(K_{-n}(a,b+1)\sqcup U,\mathbf{k}_{a,b+1})\otimes V.
     \end{equation*}
     For all the cases except $a+b+2=3$ we have by Lemmas \ref{lem: Kn(a,b) odd} and \ref{lem: Kn(a,b) nonmeeting general} that there is a unique non-zero element in the $E^1$-page of minimum Maslov grading, and it is $v_{a+1,b+1}\otimes B$ from the previous paragraphs. Further, the minimal graded element on the $E^\infty$-page is $v_{a,b+1}\otimes B\otimes B$.

     In the case of $a+b+2=3$, we see from Lemma \ref{lem: Kn(a,b) even} that the $E^\infty$-page is isomorphic to,
     \begin{equation*}
         \mathbb{F}_{-3 +2c_{K}(0)}\oplus\mathbb{F}^2_{-2 +2c_{K}(0)}\oplus\mathbb{F}_{-1 +2c_{K}(0)}.
     \end{equation*}
    Additionally by counting dimensions we see that the spectral sequence collapses on the $E^2$-page and that the $E^1$-page is,
    \begin{equation*}
        \mathbb{F}_{-3 +2c_{K}(0)}\oplus\mathbb{F}_{-2 +2c_{K}(0)} \leftarrow\mathbb{F}_{-3 +2c_{K}(0)}\oplus\mathbb{F}^4_{-2 +2c_{K}(0)}\oplus \mathbb{F}_{-1 +2c_{K}(0)}^3 \leftarrow \mathbb{F}_{-1 +2c_{K}(0)}\oplus \mathbb{F}_{-0 +2c_{K}(0)}
    \end{equation*}
    where both arrows reduce the Maslov grading by one. Therefore, as $v_{a+1,b+1}\otimes B$ is the lowest graded Maslov term in the middle term, there is clearly no way for it to be a boundary. Hence, we have that $v_{a+1,b+1}\otimes B$ and $v_{a,b+1}\otimes B \otimes B$ are identified under the spectral sequence.

    Finally, by using Lemma \ref{lem: Kn(a,b) nonmeeting general 2} we see that $F_{a+1,b+1}(v_{a+1,b+1}\otimes B)$ is either $v_{a+2,b+2}$ or $0$. Then by a similar argument to those in the previous paragraphs we can conclude,
    \begin{equation*}
        F_{a+1,b+1}(v_{a+1,b+1}\otimes B)=v_{a+2,b+2},
    \end{equation*}
    which completes the proof.
\end{proof}

Before stating the next lemmas recall the following observation from \cite{OSLinkFloer}. If 
\begin{equation*}
    \mathcal{H}=(\Sigma,\boldsymbol{\alpha},\boldsymbol{\beta},\{w_1,\ldots,w_n\},\{z_{1},\ldots,z_n\})
\end{equation*}
is a Heegaard diagram for the oriented link $L$, then,
\begin{equation*}
    \tilde{\mathcal{H}}=(\Sigma,\boldsymbol{\alpha},\boldsymbol{\beta},\{z_1,w_2\ldots,w_n\},\{w_{1},z_2,\ldots,z_n\})
\end{equation*}
is a Heegaard diagram for the oriented link $\tilde{L}$, which is $L$ with the orientation of $L_1$ reversed. Note that by swapping these basepoints there is a canonical isomorphism, 
\begin{equation*}
    s_1\colon \widehat{CFL}(\mathcal{H})\to \widehat{CFL}(\tilde{\mathcal{H}}),
\end{equation*}
given by the identification of the Heegaard states.

\begin{prop}
\label{prop: s1 is a morphism of transitive system}
    The map $s_1$ is a morphism of transitive systems.
\end{prop}

\begin{proof}
    As the boundary maps count holomorphic discs that do not cross the $\mathbf{w}$ and $\mathbf{z}$ basepoints it is clear that $s_1$ is a chain map. The main thing we need to analyse is the interaction between $s_1$ and holomorphic triangle counts. To this end let $\mathcal{T}=(\Sigma,\boldsymbol{\alpha}',\boldsymbol{\alpha},\boldsymbol{\beta},\mathbf{w},\mathbf{z})$ be a Heegaard triple and let $\tilde{\mathcal{T}}$ be the result of swapping $w_1$ and $z_1$. Then we claim that,
    \begin{equation}
    \label{eq: holomorphic triangle and s1}
        F_{\tilde{\mathcal{T}}}(s_{1}(\boldsymbol{\theta}),s_1(\mathbf{x}))=s_1(F_{\mathcal{T}}(\boldsymbol{\theta},\mathbf{x}))
    \end{equation}
    for all $\boldsymbol{\theta}\in \mathbb{T}_{\boldsymbol{\alpha}'}\cap \mathbb{T}_{\boldsymbol{\alpha}} $ and $ \mathbf{x}\in \mathbb{T}_{\boldsymbol{\alpha}}\cap \mathbb{T}_{\boldsymbol{\beta}}$. This is again simply a consequence of the fact that the maps are defined by counting holomorphic triangles that do not cross the $\mathbf{w}$ and $\mathbf{z}$ basepoints.

    That $s_1$ commutes with the stabilisation maps follows directly from the definition. Let $\mathcal{T}_{\alpha,\beta}=(\Sigma,\boldsymbol{\alpha},\boldsymbol{\beta},\mathbf{w},\mathbf{z})$ be a Heegaard diagram for $\mathbb{L}$ and suppose that $\boldsymbol{\beta}'$ is the result of a handle slide of the $\beta $ curves. Denote by $\mathcal{T}_{\beta,\beta'}=(\Sigma,\boldsymbol{\beta},\boldsymbol{\beta}',\mathbf{w},\mathbf{z})$ then we want to show that,
    \begin{equation*}
        s_1\colon \widehat{HFL}(\mathcal{T}_{\beta,\beta'})\to\widehat{HFL}(\tilde{\mathcal{T}}_{\beta,\beta'}),
    \end{equation*}
    sends the highest graded element $\boldsymbol{\theta}_{\beta,\beta'}$ to the highest graded element $\tilde{\boldsymbol{\theta}}_{\beta,\beta'}$. To see this note that $w_1$ and $z_1$ are in the same connected component of $\Sigma\backslash\boldsymbol{\beta}$ and hence also in the same component of $\Sigma\backslash\boldsymbol{\beta}'$ as the handleslide is in the complements of the basepoints. Therefore, they are in the same component of $\Sigma\backslash(\boldsymbol{\beta}\cup \boldsymbol{\beta}')$. It then follows that $s_1$ preserves the graded elements. Thus, $s_1$ commutes with the $\boldsymbol{\beta}$ handleslides. The $\boldsymbol{\alpha}$ handleslides follow similarly. This completes the proof.
\end{proof}

Let $\mathbb{L}$ be a link and $\tilde{\mathbb{L}}$ be the result of swapping the basepoints and orientation on $L_1$. Then the above shows that the map,
\begin{equation*}
    s_1\colon \widehat{HFL}(\mathbb{L}) \to \widehat{HFL}(\tilde{\mathbb{L}})
\end{equation*}
is well defined. Using this map we will consider the effect of swapping orientations on a cylinder component of band maps, quasi-stabilisations and pair of pants maps. For notation let $\Sigma$ be a cobordism such that $L_1\times I$ is a component. Then, define $\tilde{\Sigma}$ to be the cobordism obtained by swapping the orientation of $L_1\times I\subseteq\Sigma$ and swapping the decoration on this cylinder. We start with the band map.

\begin{lem}
\label{lem: band orientation reversal}
    Let $F_{\Sigma}\colon\widehat{HFL}(\mathbb{L})\to \widehat{HFL}(\mathbb{L}')$ be the map associated to an $\alpha$-band $B$ and suppose $L_1$ does not intersect the band $B$. Then the following diagram commutes,
    \begin{equation}
        \begin{tikzcd}
            \widehat{HFL}(\mathbb{L}) \arrow[rr,"F_{\Sigma}"] \arrow[dd,"s_1"] && \widehat{HFL}(\mathbb{L}')\arrow[dd,"s_1"]\\
            &&\\
            \widehat{HFL}(\tilde{\mathbb{L}})\arrow[rr,"F_{\tilde{\Sigma}}"]&&\widehat{HFL}(\tilde{\mathbb{L}}')
        \end{tikzcd}
    \end{equation}
\end{lem}

\begin{proof}
    Recall from Lemma \ref{lem: defn of band map} and Proposition \ref{prop: s1 is a morphism of transitive system} both $F_\Sigma$ and $s_1$ are morphisms of transitive systems. Fix a Heegaard triple $\mathcal{T}$ subordinate to $B$. Then, we only need to prove that the following diagram commutes up to chain homotopy,
    \begin{equation}
        \begin{tikzcd}
            \widehat{CFL}(\mathcal{T}_{\alpha,\beta}) \arrow[rr,"F_{\mathcal{T}}"] \arrow[dd,"s_1"] && \widehat{CFL}(\mathcal{T}_{\alpha',\beta})\arrow[dd,"s_1"]\\
            &&\\
            \widehat{CFL}(\tilde{\mathcal{T}}_{\alpha,\beta})\arrow[rr,"F_{\tilde{\mathcal{T}}}"]&&\widehat{CFL}(\tilde{\mathcal{T}}_{\alpha',\beta})
        \end{tikzcd}
    \end{equation}
    By the same argument in the previous proof we see that the map,
    \begin{equation*}
        s_1\colon \widehat{HFL}(\mathcal{T}_{\alpha'\alpha})\to \widehat{HFL}(\tilde{\mathcal{T}}_{\alpha'\alpha}),
    \end{equation*}
    preserves the maximal graded element. Therefore using equation \eqref{eq: holomorphic triangle and s1} we get the commutativity.
\end{proof}

Next we prove a similar result for the quasi-stabilisation map.

\begin{lem}
\label{lem: quasi-stab orientation reversal}
    Let $T^+\colon \widehat{HFL}(\mathbb{L})\to \widehat{HFL}(\mathbb{L}^+)$ be the map associated to a quasi-stabilisation on a component other than $L_1$. Then the following diagram commutes,
    \begin{equation}
        \begin{tikzcd}
            \widehat{HFL}(\mathbb{L}) \arrow[rr,"T^+"] \arrow[dd,"s_1"] && \widehat{HFL}(\mathbb{L}^+)\arrow[dd,"s_1"]\\
            &&\\
            \widehat{HFL}(\tilde{\mathbb{L}})\arrow[rr,"\tilde{T}^+"]&&\widehat{HFL}(\tilde{\mathbb{L}}^+)
        \end{tikzcd}
    \end{equation}
\end{lem}

\begin{proof}
    As in the previous proof as both maps are morphisms of transitive systems it is enough to prove commutativity for a single Heegaard diagram. Let $\mathcal{H}$ be a Heegaard diagram for $\mathbb{L}$ and $\mathcal{H}^+$ for the result of quasi-stabilisation. Then we need to show that $s_1(\mathbf{x\times \xi^{\mathbf{w}}})$ is the image of $s_1(\mathbf{x})$ under the stabilisation map $\tilde{T}^+$. To see this we note that $w_1$ and $z_1$ are not in the quasi-stabilisation map so the Maslov grading of $s_1(\mathbf{x\times \xi^{\mathbf{w}}})$ is lower than $s_1(\mathbf{x\times \xi^{\mathbf{z}}})$. Therefore, the maps commute.
\end{proof}

Finally as a consequence of these we have the same result for pair of pants cobordisms.

\begin{lem}
\label{lem: pants orientation reversal}
    Let $P\colon \widehat{HFL}(\mathbb{L})\to \widehat{HFL}(\mathbb{L}')$ be the map associated to a pair of pants that is disjoint from $L_1$. Then the following diagram commutes,
    \begin{equation}
        \begin{tikzcd}
            \widehat{HFL}(\mathbb{L}) \arrow[rr,"F_{P}"] \arrow[dd,"s_1"] && \widehat{HFL}(\mathbb{L}')\arrow[dd,"s_1"]\\
            &&\\
            \widehat{HFL}(\tilde{\mathbb{L}})\arrow[rr,"F_{\tilde{P}}"]&&\widehat{HFL}(\tilde{\mathbb{L}}')
        \end{tikzcd}
    \end{equation}
\end{lem}

\begin{proof}
    This follows directly from Lemmas \ref{lem: band orientation reversal} and \ref{lem: quasi-stab orientation reversal}.
\end{proof}

Using these lemmas we are ready to prove Theorem \ref{thm: main theorem intro version}. First though we define the cobordism maps mentioned in it.

\begin{defn}
\label{defn: F_K map}
    Let $K$ be a negative L-space knot and $n\geq 2g(K)$. Then define $F_{K_{-n}}=F_{1,0}$ when $d$ is even and $F_{K_{-n}}=F_{3,2}$ when $d$ is odd.
\end{defn}

Using this we can restate the main theorem.

\begin{thm}
\label{thm: main thm body}
    Let $K$ be a negative L-space knot and $n\geq2g(K)$. Then $\mathcal{FL}(X_{-n}(K))$ is infinite dimensional if $F_{K_{-n}}$ is an isomorphism.
\end{thm}

\begin{proof}
    First recall from Theorem \ref{thm: cable and lasagna} that we have an isomorphism,
    \begin{equation*}
        \mathcal{FL}_{i}(X_{-n}(K);\alpha,0)\cong c\widehat{HFK}_{i}(K_{-n};\alpha,0).
    \end{equation*}
    From Section \ref{sec: Link Floer Homology of T(r,rn)} we have computed the groups $\widehat{HFL}(K_{-n}(a,b),\mathbf{0})$ so it remains to understand the cobordism maps. From this section we have descriptions of $F_{P}$ from understanding $F_{K_{-n}}$ in certain gradings. Using these we will show that for $\alpha=2s+1$, with $s\geq0$ for $d$ even and $s\geq 2$ when $d$ is odd, we have,
    \begin{equation}
    \label{eq: not 0 for main proof}
        c\widehat{HFK}_{n(s^2+s)+2c_{K}(0)+1}(K_{-n};\alpha,0)\neq\{0\}.
    \end{equation}
    Here, the Maslov grading comes from the fact that for $t\geq0$ we have,
    \begin{equation*}
        i_{\alpha+t,t}=\frac{n}{4}((\alpha+t-t)^2-1)-2s-2t+2c_{K}(0)=n(s^2+s)-2(s+t)+2c_{K}(0),
    \end{equation*}
    which is exactly the Maslov grading we have studied. To get the claim we want to understand the vector space,
    \begin{equation*}
        \widehat{HFK}_{i_{\alpha+t,t}}(K_{-n}(\alpha+t,t),0)=\bigoplus_{\substack{\mathbf{k}\in \mathbb{Z}^{\alpha+2t}\\ |\mathbf{k}|=\mathbf{0}}}\widehat{HFL}_{i_{\alpha+t,t}}(K_{-n}(\alpha+t,t),\mathbf{k}).
    \end{equation*}
    
    Now recall from Lemma \ref{lem: Kn(a,b) nonmeeting general 3} that if $\mathbf{k}_i=0$ for any term then to be supported in the above grading we must have $\mathbf{k}_i=0$ for all $i$. However, we have not ruled out the case when all of the terms are non-zero but they sum to zero, for example $(-1,-1,2)$. In fact it is entirely possible that there could be non-zero elements there. Hence, we define the following vector space,
    \begin{equation*}
        W_{\alpha+t,t}=\bigoplus_{\substack{\mathbf{k}\in \mathbb{Z}^{\alpha+2t}\\ |\mathbf{k}|=0\\\mathbf{k}_i\neq0, \forall i}} \widehat{HFL}_{i_{\alpha+t,t}}(K_{-n}(\alpha+t,t),\mathbf{k}).
    \end{equation*}
    Then we have the following splitting,
    \begin{equation}
    \label{eq: splitting on each summand for main proof}
        \widehat{HFK}_{i_{\alpha+t,t}}(K_{-n}(\alpha+t,t),0)=\widehat{HFL}_{i_{\alpha+t,t}}(K_{-n}(\alpha+t,t),\mathbf{0})\oplus W_{\alpha+t,t}.
    \end{equation}
    Let $\sim$ be the equivalence relation defining $c\widehat{HFK}(K_{-n})$ then we claim that we have the following isomorphism,
    \begin{align}
        c\widehat{HFK}_{i_{\alpha,0}+\alpha}(K_{-n};\alpha,0)\cong \left(\bigoplus_{t\geq0}\widehat{HFL}_{i_{\alpha+t,t}}(K_{-n}(\alpha+t,t),\mathbf{0})[\alpha+2t] \right)/\sim
        \oplus W_\alpha, \label{eq: splitting for main proof}
    \end{align}
    where $W_{\alpha}$ is some vector space related to the $W_{\alpha+t,t}$ vector spaces. Then we will show the first summand is isomorphic to $\mathbb{F}$, which proves equation \eqref{eq: not 0 for main proof}. Hence, proving these claims is enough to prove the entire theorem.

    We start by examining how the maps defining the equivalence relation interact with the splitting in equation \eqref{eq: splitting on each summand for main proof}. First, we have that
    \begin{equation*}
        F_{m_i}\colon \widehat{HFK}_{i_{\alpha+t,t}}(K_{-n}(\alpha+t,t),0)\to \widehat{HFK}_{i_{\alpha+t,t}}(K_{-n}(\alpha+t,t),0),
    \end{equation*}
    respects the Alexander multigrading so it respects the splitting. Next we have that,
    \begin{equation*}
        F_P(-\otimes B)\colon \widehat{HFK}_{i_{\alpha+t,t}}(K_{-n}(\alpha+t,t),0)\to \widehat{HFK}_{i_{\alpha+t+1,t+1}}(K_{-n}(\alpha+t+1,t+1),0),
    \end{equation*}
    respects the Alexander multigradings except the two new ones appearing. However, this is enough to respect the splitting as new $0$ coordinates cannot be introduced. The map,
    \begin{equation*}
        F_{\tau}\colon \widehat{HFK}_{i_{\alpha+t,t}}(K_{-n}(\alpha+t,t),0)\to \widehat{HFK}_{i_{\alpha+t,t}}(K_{-n}(\alpha+t,t),0),
    \end{equation*}
    is slightly more complicated. First, collapse the Alexander multigrading to $(k_1,k_2)$ where,
    \begin{equation*}
       k_1=\sum_{i=1}^{\alpha+t}\mathbf{k}_i \qquad \text{and} \qquad k_2=\sum_{i=1}^{t}\mathbf{k}_{i+\alpha+t}.
    \end{equation*}
    Then $F_{\tau}$ respects this grading as it does not swap components of opposite orientation. Therefore, it respects the splitting. Finally, the map,
    \begin{equation*}
        F_P(-\otimes T)\colon \widehat{HFK}_{i_{\alpha+t,t}-1}(K_{-n}(\alpha+t,t),0)\to \widehat{HFK}_{i_{\alpha+t+1,t+1}}(K_{-n}(\alpha+t+1,t+1),0),
    \end{equation*}
    must be considered. Suppose there exists an element $v$ such that,
    \begin{equation*}
        F_{P}(v\otimes T)\in \widehat{HFL}_{i_{\alpha+t+1,t+1}}(K_{-n}(\alpha+t+1,t+1),\mathbf{0}).
    \end{equation*}
    Then by Lemma \ref{lem: Kn(a,b) odd} and the fact that $F_{P}$ preserves all the gradings not on the pair of pants,
    \begin{equation*}
        v \in \widehat{HFL}_{i_{\alpha+t,t}-1}(K_{-n}(\alpha+t,t),\mathbf{0}).
    \end{equation*}
    However, the above vector space is $\{0\}$ by Lemma \ref{lem: Kn(a,b) odd}.

    Using these observations about the maps we can define,
    \begin{equation*}
        W_{\alpha}=\left(\bigoplus_{t\geq0}W_{\alpha+t,t}\right)/\sim
    \end{equation*}
    where the equivalence relation is given by the transitive and linear closure of $F_{m_i}(v)\sim v$, $F_{\tau}(v)\sim v$, $F_{P}(v\otimes B)\sim v$ and $v\sim 0$ if there exists $w$ such that $F_{P}(w\otimes T)=v$. The splitting in equation \eqref{eq: splitting for main proof} then follows.
    
    We now show the first summand is isomorphic to $\mathbb{F}$. To see this we note by Lemma \ref{lem: Kn(a,b) odd} we have,
    \begin{equation*}
        \widehat{HFL}_{i_{\alpha+t,t}}(K_{-n}(\alpha+t,t),\mathbf{0})\cong \mathbb{F},
    \end{equation*}
    for all $t\geq 0$. Let $v_{\alpha+t,t}$ be the generator. By the previous paragraphs we see that $F_{m_i}$ and $F_{\tau}$ act trivially. Further, as previously noted using the same lemma we have that there does not exist a $w\in \widehat{HFL}(K_{-n}(\alpha+t-1,t-1))$ such that $F_{P}(w\otimes T)=v_{\alpha+t,t}$. Hence, it remains to understand the maps $F_{P}(-\otimes B)$. Note that by the assumptions these are exactly the maps $F_{\alpha+t,t}$ from Definition \ref{defn: Fab maps}.

    By the assumption of the theorem we have that $F_{1,0}$ or $F_{3,2}$ is an isomorphism, as they are non-vanishing maps between 1-dimensional vector spaces. Further, by Lemma \ref{lem: inductive step} we have that $F_{1+t,t}$ is an isomorphism for all $t\geq 0$ or $t\geq 2$ depending on the parity of $d$. Now by applying Lemma \ref{lem: pants orientation reversal} repeatedly we see that $\tilde{F}_{\alpha+t,t}$ is an isomorphism for $\alpha\geq 1$ or $\alpha\geq 5$ and all $t\geq 0$. Now by applying isomorphisms coming from basepoint moving maps to the top and bottom to each of the cylinders that have had their decoration swapped we see that $F_{\alpha+t,t}$ is an isomorphism. Therefore, the first summand is the quotient of,
    \begin{equation*}
        \left(\bigoplus_{t\geq 0}\mathbb{F}\right)/\sim
    \end{equation*}
    where the equivalence relation is given by $v_{\alpha+t,t}\sim v_{\alpha+t+1,t+1}$. The quotient is clearly $\mathbb{F}$, which completes the proof.
\end{proof}

\subsection{Cobordism Maps}
\label{subsec: Cobordism Maps}
\hfill

We now apply Theorem \ref{thm: main thm body} to the case of the unknot. The main technical result will be that the band map in Figure \ref{fig: diagram for pair of pants} is non-vanishing for $n<0$. For comparison we will first show that for $n>0$ the band map vanishes. 

\begin{figure}[h]
    \centering
    \def\svgwidth{1\textwidth}
\begingroup%
  \makeatletter%
  \providecommand\color[2][]{%
    \errmessage{(Inkscape) Color is used for the text in Inkscape, but the package 'color.sty' is not loaded}%
    \renewcommand\color[2][]{}%
  }%
  \providecommand\transparent[1]{%
    \errmessage{(Inkscape) Transparency is used (non-zero) for the text in Inkscape, but the package 'transparent.sty' is not loaded}%
    \renewcommand\transparent[1]{}%
  }%
  \providecommand\rotatebox[2]{#2}%
  \newcommand*\fsize{\dimexpr\f@size pt\relax}%
  \newcommand*\lineheight[1]{\fontsize{\fsize}{#1\fsize}\selectfont}%
  \ifx\svgwidth\undefined%
    \setlength{\unitlength}{838.68928203bp}%
    \ifx\svgscale\undefined%
      \relax%
    \else%
      \setlength{\unitlength}{\unitlength * \real{\svgscale}}%
    \fi%
  \else%
    \setlength{\unitlength}{\svgwidth}%
  \fi%
  \global\let\svgwidth\undefined%
  \global\let\svgscale\undefined%
  \makeatother%
  \begin{picture}(1,0.31421724)%
    \lineheight{1}%
    \setlength\tabcolsep{0pt}%
    \put(0,0){\includegraphics[width=\unitlength,page=1]{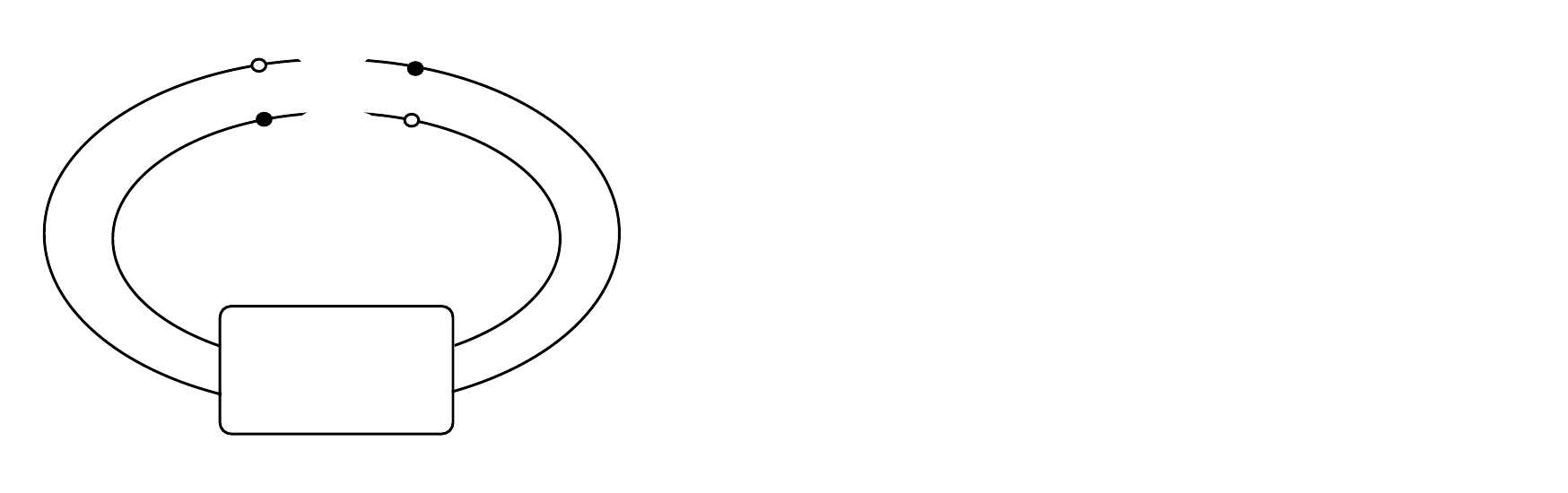}}%
    \put(0.2080784,0.07351851){\color[rgb]{0,0,0}\makebox(0,0)[lt]{\lineheight{1.25}\smash{\begin{tabular}[t]{l}$n$\end{tabular}}}}%
    \put(0,0){\includegraphics[width=\unitlength,page=2]{Pants_cobordism.pdf}}%
    \put(0.78854069,0.07316644){\color[rgb]{0,0,0}\makebox(0,0)[lt]{\lineheight{1.25}\smash{\begin{tabular}[t]{l}$n$\end{tabular}}}}%
    \put(0,0){\includegraphics[width=\unitlength,page=3]{Pants_cobordism.pdf}}%
    \put(0.49694838,0.19930853){\color[rgb]{0,0,0}\makebox(0,0)[lt]{\lineheight{1.25}\smash{\begin{tabular}[t]{l}$\Sigma$\end{tabular}}}}%
    \put(0.35278248,0.2718155){\color[rgb]{0,0,0}\makebox(0,0)[lt]{\lineheight{1.25}\smash{\begin{tabular}[t]{l}$\mathbb{U}'$\end{tabular}}}}%
    \put(0.92181384,0.27605567){\color[rgb]{0,0,0}\makebox(0,0)[lt]{\lineheight{1.25}\smash{\begin{tabular}[t]{l}$U_n(1,1)$\end{tabular}}}}%
  \end{picture}%
\endgroup%

    \caption{The cobordism corresponding to the band in Lemma \ref{lem: band map vanishes n>0} and Proposition \ref{prop: band map injective for n<0}. The $n$ in the box means $n$ full right twists if $n>0$ and $|n|$ full left twists if $n<0$.}
    \label{fig: diagram for pair of pants}
\end{figure}

\begin{lem}
\label{lem: band map vanishes n>0}
    Let $\mathbb{U}$ be the unknot with four basepoints, $B$ be the $\alpha$-band in Figure \ref{fig: diagram for pair of pants} and let $n>0$. Then the induced map $F\colon\widehat{HFL}(\mathbb{U})\to \widehat{HFL}(U_{n}(1,1))$ is the $0$ map. 
\end{lem}

\begin{proof}
    First note that by Theorem \ref{thm: grading change formulas} we have that $F$ reduces the Alexander grading by $\frac{1}{2}$ and preserves the Maslov grading. We claim this is enough to see that the maps vanish. To see this recall that $\widehat{HFL}(\mathbb{U})\cong\mathbb{F}\langle x,y \rangle$, where $M(x)=-1,M(y)=0,A(x)=-\frac{1}{2}$ and $A(y)=\frac{1}{2}$. Further, by Corollary \ref{cor: HFL of Kn(a+t,b+t)} we have,
    \begin{align}
        &\widehat{HFL}_{*}\left(U_{n}(1,1),\left(\frac{n}{2},-\frac{n}{2}\right)\right)\cong \widehat{HFL}_{*}\left(U_{n}(1,1),\left(-\frac{n}{2},\frac{n}{2}\right)\right)\cong \mathbb{F}_{0} \\
        &\widehat{HFL}_{*}\left(U_{n}(1,1),\left(\frac{n}{2}-j,-\frac{n}{2}+j\right)\right)\cong \mathbb{F}_{0}^2,
    \end{align}
    for $1\leq j \leq n-1$. Similarly by Corollary \ref{cor:HFL of Kn(a+t,b+t) non meeting} we have,
    \begin{align}
        &\widehat{HFL}_{*}\left(U_{n}(1,1),\left(\frac{n}{2}-j,-\frac{n}{2}+j+1\right)\right)\cong \mathbb{F}_{1} \\
        &\widehat{HFL}_{*}\left(U_{n}(1,1),\left(\frac{n}{2}-j-1,-\frac{n}{2}+j\right)\right)\cong \mathbb{F}_{-1},
    \end{align}
    for $0\leq j \leq n-1$. Finally by \cite[Theorem 3]{linkLspacernrm} $\widehat{HFL}(U_{n}(1,1))$ is $\{0\}$ for all other Alexander gradings. Collapsing these to a single grading gives that $\widehat{HFK}(U_{n}(1,1))$ is supported in gradings $(1,1),(0,0)$ and $(-1,-1)$, where $(M,A)$ is the Maslov and Alexander grading. However, by the grading change we have $\big(M(F(x)),A(F(x))\big)=(-1,0)$ and $\big(M(F(y)),A(F(y))\big)=(0,1)$, so $F$ is the zero map.
\end{proof}

We now move onto the more complicated computation in the case of $n<0$. In this case the map will be injective and we prove this by an explicit computation of the map using a holomorphic triangle count. To do this we recall the following combinatorial description of the Maslov index for Whitney triangles from \cite{MaslovIndexPaper}.

\begin{thm}
\label{thm: formula for maslov index}
    Let $(\Sigma,\boldsymbol{\alpha}',\boldsymbol{\alpha},\boldsymbol{\beta},\mathbf{w},\mathbf{z})$ be a Heegaard triple and let $\boldsymbol{\theta}\in \mathbb{T}_{\boldsymbol{\alpha}'}\cap \mathbb{T}_{\boldsymbol{\alpha}}$, $\mathbf{x}\in \mathbb{T}_{\boldsymbol{\alpha}} \cap  \mathbb{T}_{\boldsymbol{\beta}}$ and $\mathbf{y}\in  \mathbb{T}_{\boldsymbol{\alpha}'} \cap  \mathbb{T}_{\boldsymbol{\beta}}$. Further, let $\psi \in \pi_{2}(\boldsymbol{\theta},\mathbf{x},\mathbf{y})$ and define,
    \begin{equation*}
        a(D(\psi))=\partial D(\psi)\cap\boldsymbol{\alpha}' \qquad \text{and} \qquad c(D(\psi))= \partial D(\psi)\cap\boldsymbol{\beta}
    \end{equation*}
    Finally, define $a(D(\psi))\cdot c(D(\psi))$ to be the average of the algebraic intersections between the push-off of $a(D(\psi))$ in the four diagonal directions and $c(D(\psi))$. Then,
    \begin{equation}
        \mu(\psi)=e(D(\psi))+n_{\boldsymbol{\theta}}(D(\psi))+n_{\mathbf{x}}(D(\psi))-a(D(\psi))\cdot c(D(\psi))-\frac{|\boldsymbol{\alpha}'|}{2},
    \end{equation}
    where $e(D(\psi))$, $n_{\boldsymbol{\theta}}(D(\psi))$ and $n_{\mathbf{x}}(D(\psi))$ are the Euler and point measures.
\end{thm}

Using this and Figure \ref{fig: Heegaard Triple}, which can be seen to be a Heegaard triple subordinate to the band in question, we can prove the following.

\begin{prop}
\label{prop: band map injective for n<0}
    Let $\mathbb{U}$ be the unknot with four basepoints and let $n>0$. Then the induced map $F\colon\widehat{HFL}(\mathbb{U})\to \widehat{HFL}(U_{-n}(1,1))$ of the $\alpha$-band map in Figure \ref{fig: diagram for pair of pants} is injective. 
\end{prop}

\begin{figure}[h]
        \centering
        \def\svgwidth{1\textwidth}
        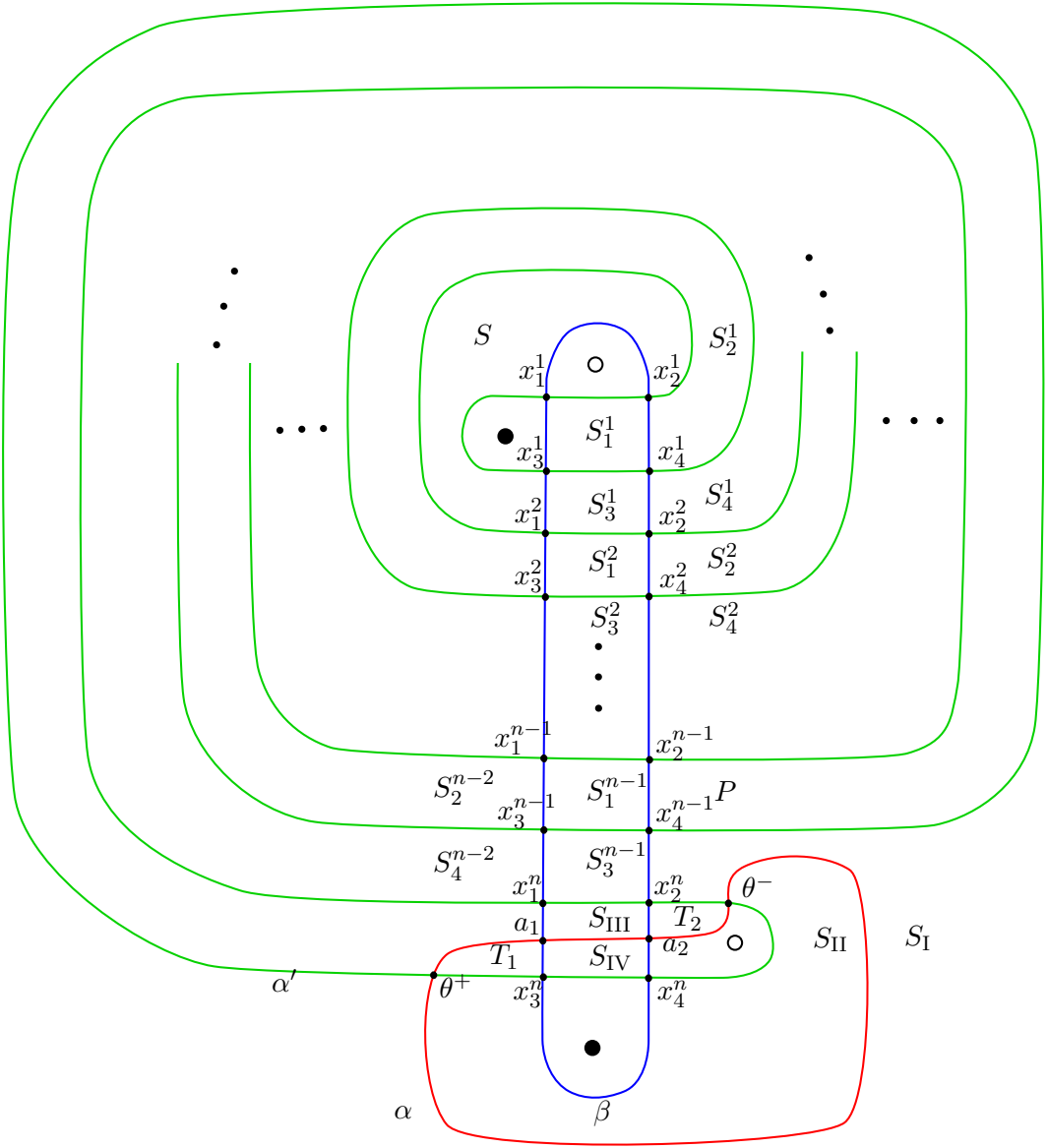
        \caption{Heegaard triple for Proposition \ref{prop: band map injective for n<0}.}
        \label{fig: Heegaard Triple}
\end{figure}

\begin{proof}
     We will compute the map $F$ using the Heegaard triple $(S^2,\alpha',\alpha,\beta,\mathbf{w},\mathbf{z})$ in Figure \ref{fig: Heegaard Triple}. To start we want to determine what tuples $(\theta^{+},a_{k},x_{j}^{i})$ have classes of Whitney triangles with positive domains between them. Recall that if there exists a $\psi_{i,j,k}\in \pi_{2}(\theta^{+},a_{k},x_{j}^{i})$ then we have that,
     \begin{equation*}
         \partial(\partial D(\psi_{i,j,k})\cap \alpha')=x_{j}^i-\theta^{+},\,\,\,\, \partial(\partial D(\psi_{i,j,k})\cap \alpha)=\theta^{+}-a_{k} \,\,\,\, \text{and} \,\,\,\, \partial(\partial D(\psi_{i,j,k})\cap \beta)=a_{k}-x_{j}^i.
     \end{equation*}
     To rule out possible combinations of domains we have given the calculations of boundary conditions of each region in Table \ref{tab:boundary of regions}.
     \begin{table}[h!]
    \centering
    \begin{tabular}{|c||c|c|c|}
        \hline
        Region D & $\partial(\partial D\cap \alpha')$ & $\partial(\partial D\cap \alpha)$ & $\partial(\partial D\cap \beta)$ \\
        \hline
        $S_{1}^i$ & $x_{1}^i-x_{2}^i+x_{4}^i-x_{3}^i$ &$0$& $x_{3}^i-x_{1}^i+x_{2}^i-x_{4}^i$ \\
        $S_{2}^i$ & $x_{2}^i-x_{1}^{i+1}+x_{3}^{i+1}-x_{4}^i$ & 0  & $x_{4}^i-x_{2}^i+x_{1}^{i+1}-x_{3}^{i+1}$\\
        $S_{3}^i$ & $x_{3}^i-x_{4}^i+x_{2}^{i+1}-x_{1}^{i+1}$ & 0 & $x_{1}^{i+1}-x_{3}^i+x_{4}^i-x_{2}^{i+1}$ \\
        $S_{4}^i$ & $x_{4}^i-x_{3}^{i+1}+x_{1}^{i+2}-x_{2}^{i+1}$ & 0 & $x_{2}^{i+1}-x_{4}^i+x_{3}^{i+1}-x_{1}^{i+2}$ \\
        $S$ & $x_{1}^1-x_{3}^1+x_{1}^2-x_{2}^1$ & 0 & $x_{2}^1-x_{1}^1+x_{3}^1-x_{1}^2$ \\
        $S_{\text{I}}$ & $x_{4}^{n-1}-\theta^{+}+\theta^{-}-x_{2}^n$ & $\theta^{+}-\theta^{-}$ & $x_{2}^n-x_{4}^{n-1}$\\
        $S_{\text{II}}$ & $x_{4}^n-\theta^{-}+\theta^{+}-x_{3}^n$ & $\theta^{-}-\theta^{+}$ & $x_{3}^n-x_{4}^n$ \\
        $S_{\text{III}}$ & $x_{1}^n-x_{2}^n$ & $a_{2}-a_{1}$ & $a_{1}-x_{1}^n+x_{2}^n-a_{2}$ \\
        $S_{\text{IV}}$ & $x_{4}^n-x_{3}^n$ & $a_{1}-a_{2}$ & $x_{3}^n-a_{1}+a_{2}-x_{4}^n$ \\
        $T_{1}$ & $x_{3}^n-\theta^{+}$ & $\theta^{+}-a_{1}$ & $a_{1}-x_{3}^n$\\
        $T_{2}$ & $x_{2}^n-\theta^{-}$ & $\theta^{-}-a_{2}$ & $a_{2}-x_{2}^n$ \\
        $P$ & $x_{2}^{n-1}-x_{1}^n+\theta^{+}-x_{4}^{n-1}$ & $a_{1}-\theta^{+}$ & $x_{1}^n-a_{1}+x_{4}^{n-1}-x_{2}^{n-1}$ \\
        \hline
    \end{tabular} 
    \caption{Boundary conditions of regions in $\Sigma\backslash\alpha'\cup\alpha\cup\beta$.}
    \label{tab:boundary of regions}
    \end{table}
    Let $c_{D}$ be the coefficient of the region $D$ in the domain $D(\psi_{i,j,k})$, and set $c_{b}^a:=c_{S_{b}^a}$.
    Then from Table \ref{tab:boundary of regions} we have the following,
    \begin{align}
        \partial(\partial D(\psi_{i,j,k})\cap \alpha)=& (c_{S_\text{I}}-c_{P}+c_{T_1}-c_{S_{\text{II}}}) \theta^+ + (c_{T_2}-c_{S_{\text{I}}}+c_{S_{\text{II}}})\theta^- \nonumber \\
        &-(c_{S_{\text{III}}}-c_{S_{\text{IV}}}+c_{T_{1}}-c_{P})a_{1}-(c_{T_{2}}+c_{S_{\text{IV}}}-c_{S_{\text{III}}})a_{2}. \label{eq: alpha prime coeff}
    \end{align}
    We similarly collate the data for coefficients in $\partial(\partial D(\psi_{i,j,k})\cap \alpha')$ in Table \ref{tab: coef of intersects}. From this there are some useful initial observations. First, if $D(\psi_{i,j,k})$ is positive and $x_{j}^i\neq x_{4}^n$, we must have $c_{S_{II}}=c_{S_{IV}}=0$. Otherwise the coefficient of $x_{4}^n$ will be non-zero. Similarly, if $x_{j}^i\neq x_{1}^1$ then $c_{S}=c_{1}^1=0$. Finally, note that this implies that if $x_{j}^i\notin\{x_{3}^n,x_{4}^n\}$ then $c_{T_{1}}=0$ as well. Using these observations and the calculations in Table \ref{tab: coef of intersects} we now show for all tuples $(i,j,k)$, that in $\pi_{2}(\theta^{+},a_{k},x_{j}^{i})$ there are either no class with positive domain or a unique class $\psi_{i,j,k}$ that has a positive domain. 
    \begin{table}[h!]
    \centering
        \begin{tabular}{|c|c|}
        \hline
            Intersection & Coefficient \\
            \hline
            $x_{1}^i$ & $c_{1}^i-c_{2}^{i-1}-c_{3}^{i-1}+c_{4}^{i-2}$\\
            $x_{2}^i$ & $-c_{1}^i+c_{2}^i+c_{3}^{i-1}-c_{4}^{i-1}$\\
            $x_{3}^i$ & $-c_{1}^i+c_{2}^{i-1}+c_{3}^{i}-c_{4}^{i-1}$\\
            $x_{4}^i$ & $c_{1}^i-c_{2}^i-c_{3}^i+c_{4}^i$\\

            $x_{1}^1$ & $c_{S}+c_{1}^1$ \\
            $x_{2}^1$ &  $-c_{S}-c_{1}^1+c_{2}^1$ \\
            $x_{3}^1$ &  $-c_{S}-c_{1}^1+c_{3}^1$ \\
            $x_{1}^2$ & $c_{1}^2-c_{2}^1-c_{3}^1+c_{S}$\\

            $x_{2}^{n-1}$ & $-c_{1}^{n-1}+c_{P}+c_{3}^{n-2}-c_{4}^{n-2}$\\
            $x_{4}^{n-1}$ & $c_{1}^{n-1}-c_{P}-c_{3}^{n-1}+c_{S_{I}}$\\

            $x_{1}^n$ & $c_{S_{III}}-c_{P}-c_{3}^{n-1}+c_{4}^{n-2}$\\
            $x_{2}^n$ & $-c_{S_{III}}+c_{T_{2}}+c_{3}^{n-1}-c_{P}$\\
            $x_{3}^n$ & $-c_{S_{IV}}+c_{T_{1}}-c_{S_{II}}$\\
            $x_{4}^n$ & $c_{S_{IV}}+c_{S_{II}}$\\

            $\theta^+$ & $-c_{T_{1}}+c_{S_{II}}-c_{S_{I}}+c_{P}$\\
            $\theta^-$ & $-c_{T_2}+c_{S_{I}}-c_{S_{II}}$\\
            \hline
            
        \end{tabular}
        \caption{Coefficients of intersection points in $\partial(\partial D(\psi_{i,j,k})\cap \alpha')$. The first four rows hold for all $i \in \{1,\ldots,n\}$ that do not appear in the lower rows.}
        \label{tab: coef of intersects}
    \end{table}

    We will first show that there is no $\psi_{n,3,2}\in \pi_{2}(\theta^{+},a_{2},x_{3}^n)$ such that $D(\psi_{n,3,2})$ is positive. Note that by the condition $\partial(\partial D(\psi_{n,3,2})\cap \alpha')=x^n_{3}-\theta^+$ all other $x_{j}^i$ must have coefficient 0. Using the above observations, we have $c_{S_{II}}=c_{S_{IV}}=c_{S}=c_{1}^1=0$. Therefore, from Table \ref{tab: coef of intersects} we have $c_{T_{1}}=1$ and $c_{2}^1=c_{3}^1=0$. This implies that $c_{4}^1=0$ using the coefficient of $x_{4}^1$, also  define $c_{4}^0:=c_{S}$. We now induct to show that $c_{j}^i=0$ for all $i,j$. Suppose that for all $0\leq i\leq k$ and $j\in \{1,2,3, 4\}$ we have $c_{j}^i=0$. Then by the coefficient of $x_{1}^{k+1}$ we have that $c_{1}^{k+1}=0$ and this implies that $c_{2}^{k+1}=c_{3}^{k+1}=0$ using the coefficients of $x_{2}^{k+1}$ and $x_{3}^{k+1}$ respectively. Further, this implies that $c_{4}^{k+1}=0$, completing the induction. Now using the coefficient of $x_{2}^{n-1}$ we have that $c_{P}=0$ and the coefficients of $x_{4}^{n-1}$ and $x_{1}^n$ imply that $c_{S_{I}}=c_{S_{III}}=0$ respectively. The coefficient of $x_{2}^n$ then implies that $c_{T_{2}}=0$, which then in turn implies that $D(\psi_{n,3,2})=T_{1}$. However, this is a contradiction as $\partial(\partial D(\psi_{n,3,2})\cap \alpha)=\theta^+-a_{1}$ which is not $\theta^+-a_{2}$. This completes the claim and further tells us that there are no holomorphic triangles connecting these points.

    Now using a similar argument we will prove that for $i\in \{2,\ldots,n\}$ and $k=1$ or $2$ there is no class in $\pi_{2}(\theta^{+},a_{k},x_{1}^i)$ that has a holomorphic representative. A similar argument to the previous paragraph shows that for $p<i$ we have $c_{j}^p=0$ for all $j \in \{1,2,3, 4\}$. Suppose first that $i \neq n$ then we have $c_{1}^i=1$ by the coefficient of $x_{1}^i$ in $\partial(\partial D(\psi_{i,1,k})\cap \alpha')$. We now claim that $c_{1}^{n-1}=c_{P}=c_{3}^{n-1}=c_{S_{I}}=n-i$. To see this note that the coefficients of $x_{2}^i,x_{3}^i$ and $x_{4}^i$ imply that $c_{2}^i=c_{3}^i=c_{4}^i=1$ respectively. We can then argue inductively in the same way as the previous paragraph that $c_{j}^p=p+1-i$ for all $i\leq p < n-1$. The coefficients of $x_{1}^{n-1},x_{2}^{n-1},x_{3}^{n-1}$ and $x_{4}^{n-1}$ imply the claim. Using the claim and the coefficients of $x_{1}^n$ and $x_{2}^n$ we then have that $c_{S_{III}}=c_{T_{2}}=n+1-i$. A similar argument shows that these equations hold for the case $i=n$. Now applying these to equation \eqref{eq: alpha prime coeff} we get that $\partial(\partial D(\psi_{i,1,k})\cap \alpha)=\theta^--a_{1}$, which is a contradiction. A similar argument shows that for all $i\in \{1,\ldots,n\}$ and $\psi_{i,2,k}\in \pi_{2}(\theta^{+},a_{k},x_{2}^i)$ that is positive we have $\partial(\partial D(\psi_{i,2,k})\cap \alpha)=\theta^--a_{2}$, which is again a contradiction. Therefore, for all of these cases there are no classes with a holomorphic representative.

    \begin{figure}[h]
        \centering
        \def\svgwidth{1\textwidth}
        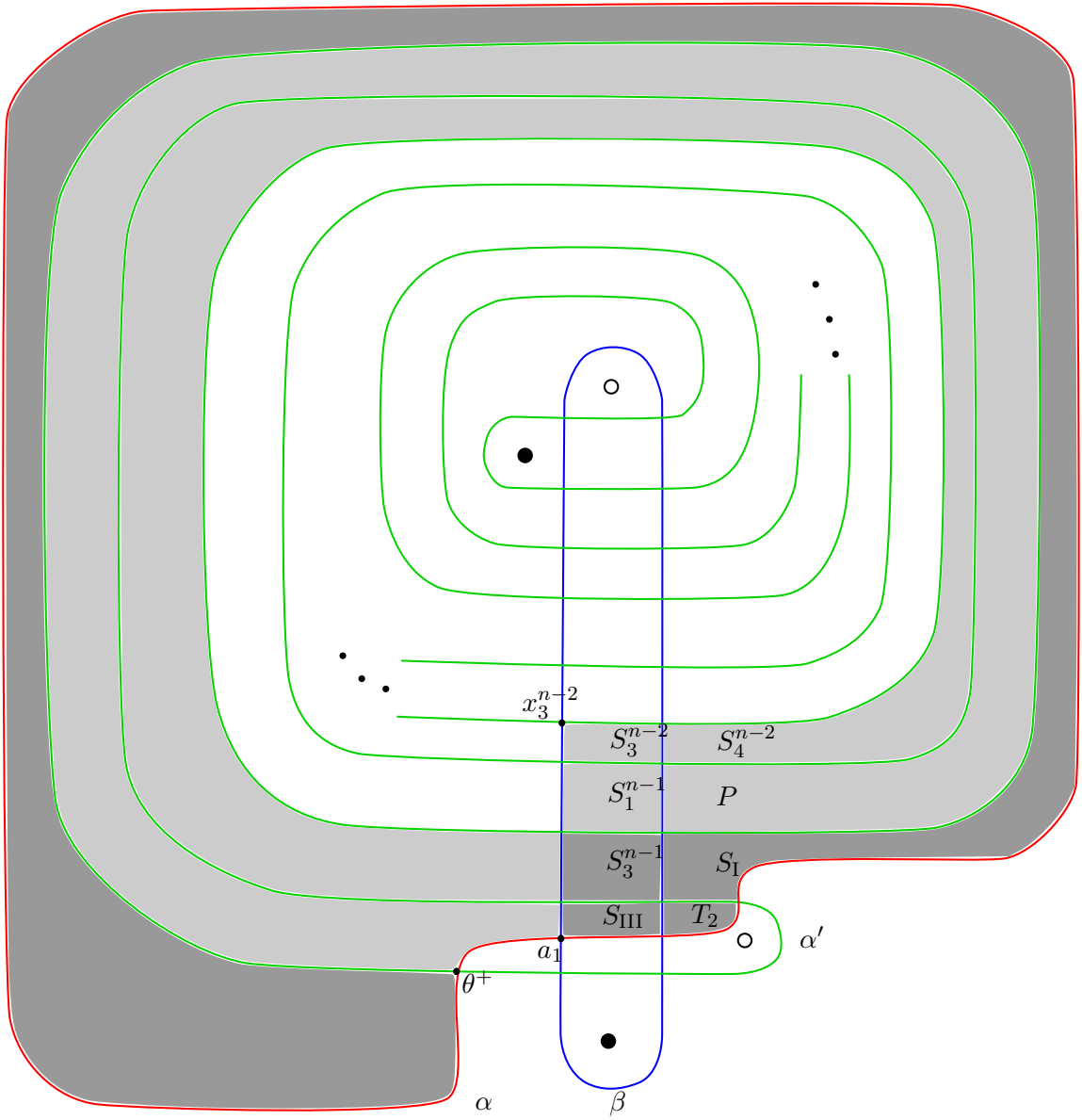
        \caption{Positive domain from equation \eqref{eq: domain for triangle} for the case $D(\psi_{n-2,3,1})$.}
        \label{fig: big Heegaard Triple Triangle example}
    \end{figure}

    Next, using the same arguments we have that for all $i\in \{1,\ldots, n\}$ there is exactly one class $\psi_{i,3,1}\in \pi_{2}(\theta^{+},a_{1},x_{3}^i)$ that has a positive domain. Precisely we have,
    \begin{align}
        D(\psi_{i,3,1})=\sum_{p=i}^{n-2}\bigg((p-i+1)(S_{3}^p+S_{4}^p)+(p-i)(S_{1}^p+S_{2}^P)\bigg) + (n-i-1)(S_{1}^{n-1}+P)\nonumber \\+(n-i)(S_{3}^{n-1}+S_{I}+S_{III}+T_{2}), \label{eq: domain for triangle}
    \end{align}
    for $i\neq n$ and $D(\psi_{n,3,1})=T_{1}$. For an example of the domain in equation \eqref{eq: domain for triangle} see Figure \ref{fig: big Heegaard Triple Triangle example}, which gives the case $i=n-2$. By similar arguments we also have that for all $i\in\{1,\ldots,n\}$ and $\psi_{i,4,k}\in \pi_{2}(\theta^{+},a_{k},x_{4}^i)$ that is positive we have $\partial(\partial D(\psi_{i,4,k})\cap \alpha)=\theta^+-a_{2}$. So there are no holomorphic representatives for the case $k=1$ and in fact we have a unique domain in the case of $k=2$. Doing this argument a final time we see that there is a unique $\psi_{1,1,2}\in \pi_{2}(\theta^{+},a_{2},x_{1}^1)$ that is positive and no $\psi_{1,1,1}\in \pi_{2}(\theta^{+},a_{1},x_{1}^1)$ that is positive. To compute the map we now need to compute $\mu(\psi_{i,j,k})$ and $\#\mathcal{M}(\psi_{i,j,k})$ for the cases where there may be holomorphic representatives.

    We start by computing $\mu(\psi_{i,3,1})$ using Theorem \ref{thm: formula for maslov index}. First, we note that $e(D)=1-\frac{n}{4}$ when $D$ is an $n$-gon. Therefore, using equation \eqref{eq: domain for triangle} we can see that $e(D(\psi_{i,3,1}))=\frac{1}{4}$. Next, we note that 
    \begin{equation*}
        n_{\theta^{+}}(D(\psi_{i,3,1}))=\frac{c_{S_{\text{I}}}+c_{P}}{4}=\frac{n-i}{2}-\frac{1}{4} \qquad \text{and} \qquad n_{a_{1}}(D(\psi_{i,3,1}))=\frac{c_{S_{\text{III}}}+c_{P}}{4}=\frac{n-i}{2}-\frac{1}{4}.
    \end{equation*}
    Lastly, we see that,
    \begin{equation*}
        a(D(\psi_{i,3,1}))\cdot c(D(\psi_{i,3,1}))=\frac{3(n-i-1)+(n-i)}{4}=n-i-\frac{3}{4},
    \end{equation*}
    and hence that $\mu(\psi_{i,3,1})=0$. We can similarly see that $\mu(\psi_{n,3,1})=0$.

    \begin{figure}[h!]
        \centering
        \def\svgwidth{1\textwidth}
        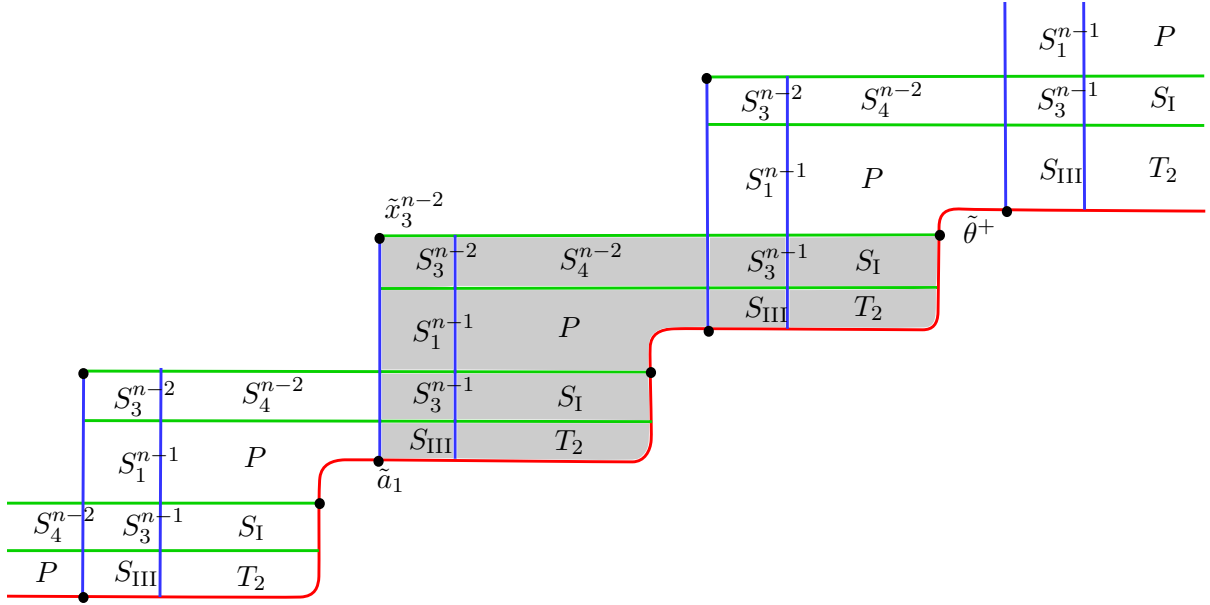
        \caption{Universal cover of the domain $D(\psi_{n-2,3,1})$. The shaded region corresponds to the domain of $\tilde{f}$.}
        \label{fig: lift of Heegaard Triple Triangle example}
    \end{figure}

    Next, we need to compute $\#\mathcal{M}(\psi_{i,3,1})$. First, we consider the case $i=n$. In this case note that the domain is exactly a disc. Therefore, by the Riemann mapping theorem there exists a holomorphic representative and due to the boundary conditions, this representative is unique. Thus, we have $\#\mathcal{M}(\psi_{n,3,1})=1$. Now for the case $i\neq n$ we will note that the domain is an annulus, $A$. Hence, an analytic universal cover $(\Omega,\tau)$ can be taken to be the infinite strip, with the exponential map. We also pick a lift $\tilde{\theta}^+$ such that $\tau\big(\tilde{\theta}^+\big)=\theta^+$. Let $u\in \psi_{i,3,1}$ be any representative, then by covering space theory there is a unique lift to a map $\tilde{u}\colon \Delta \to \Omega$ such that $\tau\circ\tilde{u}=u$ and $\tilde{u}(v_{\beta})=\tilde{\theta}^{+}$. Let us denote by $\tilde{u}(v_{\alpha})=\tilde{x}_{3}^i$ and $\tilde{u}(v_{\alpha}')=\tilde{a}_{1}$, and note that $\tau(\tilde{x}_{3}^i)=x_3^i$ and $\tau(\tilde{a}_1)=a_1$. Note further, that we can lift the segments of the curves $\alpha',\alpha$ and $\beta$ that intersect $A$ to $\Omega$. For an example see Figure \ref{fig: lift of Heegaard Triple Triangle example} for the lift of the domain in Figure \ref{fig: big Heegaard Triple Triangle example}. Now as the lift of $u$ is a triangle there exists a unique holomorphic map $\tilde{f}\colon \Delta \to \Omega$ such that $\tilde{f}(e_{\beta})=\tilde{\theta}^{+}$ and $\tau\circ \tilde{f}\in \psi_{i,3,1}$. Hence, as $\tilde{f}$ and $\tau$ are holomorphic, there exists a holomorphic representative of $\psi_{i,3,1}$. Now suppose $g\in \psi_{i,3,1}$ is a holomorphic representative. Then by standard results from complex analysis, see for example \cite[Chapter 16, Corollary 2.6]{ComplexAnalysisRef}, we obtain a unique lift $\tilde{g}\colon \Delta \to \Omega$ such that $\tau \circ \tilde{g}=g$ and $\tilde{g}(v_{\beta})=\tilde{\theta}^+$. However, by uniqueness of $\tilde{f}$ we have $\tilde{g}=\tilde{f}$ and so $g=\tau \circ \tilde{f}$. Therefore, there is a unique holomorphic representative of $\psi_{i,3,1}$, so $\#\mathcal{M}(\psi_{i,3,1})=1$. 

    Similar arguments to the previous paragraphs show that $\mu(\psi_{i,4,2})=0$ and $\#\mathcal{M}(\psi_{i,4,2})=1$, and $\mu(\psi_{1,1,2})=0$ and $\#\mathcal{M}(\psi_{1,1,2})=1$. Combining this with the previous paragraph we get that $F(\theta^+\otimes a_{1})=x_{3}^1 +\ldots +x_{3}^n$ and $F(\theta^+\otimes a_{2})=x_{4}^1+\ldots+x_{4}^n+x_{1}^1$, which completes the proof.
\end{proof}

With this computation done we cannot yet apply Theorem \ref{thm: main thm body} as we need the map $F_{U_{-n}}$ as defined in Definition \ref{defn: F_K map}. However, by using Corollary \ref{cor: spec maps for pants} we can compute $F_{U_{-n}}$ from the above calculation. This is the content of the next lemma.

\begin{lem}
\label{lem: FUn is an isomorphism}
    For all $n>0$ the map,
    \begin{equation*}
        F_{U_{-n}}\colon \widehat{HFL}_{i_{1,0}-1}(U_{-n}(1,0)\sqcup U,0)\to \widehat{HFL}_{i_{2,1}}(U_{-n}(2,1),(0,0,0)),
    \end{equation*}
    is an isomorphism.
\end{lem}

\begin{proof}
    Let $F\colon\widehat{HFL}(\mathbb{U})\to \widehat{HFL}(U_{-n}(1,1))$ be the map computed in Proposition \ref{prop: band map injective for n<0}. Further, let $T^+_U\colon \widehat{HFL}(U)\to \widehat{HFL}(\mathbb{U})$ be the quasi-stabilisation from the unknot with two basepoints to the unknot with four. Next we identify
    \begin{equation*}
        \widehat{HFL}(U)\cong\mathbb{F}\langle y \rangle \qquad \text{and} \qquad \widehat{HFL}(\mathbb{U})\cong \mathbb{F}\langle a_1, a_2 \rangle,
    \end{equation*}
    where $a_1$ and $a_2$ are the intersections from the previous proposition. Note that $M(a_1)=0$, $M(a_2)=-1$, $A(a_1)=\frac{1}{2} $ and $A(a_2)=-\frac{1}{2}$. Therefore, by the definition of quasi-stabilisation maps we have $T^{+}_U(y)=a_2$. Now we identify $\widehat{HFL}(U_{-n}(1,1))$ with the vector space generated by $x_{j}^i$ for $j\in \{1,2,3,4\}$ and $i \in \{1,\ldots ,n\}$ as in the previous proposition. Then the pair of pants map $G=F\circ T^+_U\colon \widehat{HFL}(U)\to \widehat{HFL}(U_{-n}(1,1))$ is given by,
    \begin{equation}
    \label{eq: computation of basic pants map}
        G(y)=x_{4}^1+\ldots+x_{4}^n+x_{1}^1.
    \end{equation}
    We will use this and an argument similar to the proof of Lemma \ref{lem: inductive step} to prove that $F_{U_{-n}}$ is non-vanishing, which is equivalent to being an isomorphism as the domain and codomain are 1-dimensional.

    We start by analysing the spectral sequence,
    \begin{equation}
        \widehat{HFL}(U_{-n}(1,0)\sqcup U)\implies \widehat{HFL}(U) \otimes V.
    \end{equation}
    First, note that $U_{-n}(1,0)$ is simply an unknot with two basepoints, which we will denote by $L_1$. Hence, the $E^1$-page is isomorphic to $\mathbb{F}\otimes V$ by Lemma \ref{lem:  unknot component}. However, the $E^{\infty}$-page is also isomorphic to $\mathbb{F}\otimes V$ so the spectral sequence collapses on the $E^1$-page. We denote by $v_{1,0}$ the element of minimal grading on the $E^1$-page and note that it is identified with $y\otimes B$ on the $E^\infty$-page.

    Next, label the components of $U_{-n}(2,1)$ by $L_1,L_2$ and $L_3$ such that $L_1$ and $L_2$ are oriented the same way and $L_3$ the opposite way. Further, we specify that $L_2$ and $L_3$ are the components coming from the pair of pants defining $F_{U_{-n}}$. Then consider the spectral sequence,
    \begin{equation*}
        \widehat{HFL}(U_{-n}(2,1))\implies \widehat{HFL}(U_{-n}(1,1))\otimes V,
    \end{equation*}
    given by forgetting the component $L_{1}$. Note that, $lk(L_{1},L_{2})=-n$ and $lk(L_1,L_3)=n$, so the spectral sequence sends the part supported in Alexander multigrading $(x,y,z)$ to $(y+\frac{n}{2},z-\frac{n}{2})$. Further, from the calculation in Section \ref{sec: Link Floer Homology of T(r,rn)} we have that,
    \begin{equation*}
        \widehat{HFL}(U_{-n}(1,1),(y,z))\cong \{0\},
    \end{equation*}
    for $|y|,|z|>\frac{n}{2}$. Therefore, if $v \in \widehat{HFL}(U_{-n}(2,1),(x,y,z))$ with $y>0$ or $z<0$ then $[v]^{\infty}=0$. Then using Corollaries \ref{cor: HFL of Kn(a+t,b+t)} and \ref{cor:HFL of Kn(a+t,b+t) non meeting} we can see that the spectral sequence collapses on the $E^1$-page for the part supported in Alexander gradings $(x,y,z)$ with $y<0$ or $z>0$. The remaining part has $E^1$-page $\bigoplus_{x\in\{-1,0,1\}}\widehat{HFL}(U_{-n}(2,1),(x,0,0))$ with differentials as follows,
    \begin{equation}
        \begin{tikzcd}
            \mathbb{F}_{-2} & \mathbb{F}_{-2}\oplus \mathbb{F}_{-1}^3 \arrow[l,"g"'] & \mathbb{F}_{0}.\arrow[l,"h"']
        \end{tikzcd}
    \end{equation}
    Further, the $E^\infty$-page is $\widehat{HFL}(U_{-n}(1,1),(\frac{n}{2},-\frac{n}{2}))\otimes V\cong \mathbb{F}_{-2}\oplus \mathbb{F}_{-1}$, so we see that $g$ is surjective and $h$ injective. Let $v_{2,1}$ be the unique non-zero element of minimum grading in $\widehat{HFL}(U_{-n}(2,1), (0,0,0))$. Then note from analysing Figure \ref{fig: Heegaard Triple} we see that $x_4^n$ is the unique non-zero element in $\widehat{HFL}(U_{-n}(1,1),(\frac{n}{2},-\frac{n}{2}))$. Hence, $[v_{2,1}]^{\infty}=x_{4}^n \otimes B$. 

    Let $F_P\colon \widehat{HFL}(U_{-n}(1,0)\sqcup U)\to \widehat{HFL}(U_{-n}(2,1))$ be the map associated to the pair of pants cobordism $P$ defining $F_{U_{-n}}$. Then note that $L_1\times I\subseteq P$ is a component of the cobordism. Further, let,
    \begin{equation*}
        f_P\colon \widehat{CFL}_{z_1}(U_{-n}(1,0)\sqcup U) \to \widehat{CFL}_{z_1}(U_{-n}(2,1)),
    \end{equation*}
    be the morphism of transitive systems associated to the cobordism $P$. Then note that $E^{1}(f_P)=F_P$ and by Corollary \ref{cor: spec maps for pants}, $(f_P)_*$ equals $G\otimes \text{Id}_V$ under the canonical identification maps from Proposition \ref{prop: canonical identification}. Now by computing $A_1$ gradings of the elements $x_{4}^i$ and $x_{1}^1$, using equation \eqref{eq: computation of basic pants map} and applying Lemma \ref{lem: filtration and associated grading in single degree} we see that $E^{\infty}(f_P)(y\otimes B)=x_{4}^n\otimes B$. Then we have,
    \begin{equation*}
        [F_{U_{-n}}(v_{1,0})]^\infty=[F_{P}(v_{1,0})]^\infty=E^{\infty}(f_P)([v_{1,0}]^{\infty})=x_{4}^n\otimes B.
    \end{equation*}
    However, by the previous paragraph this shows that $F_{U_{-n}}(v_{1,0})=v_{2,1}$. This completes the proof.
\end{proof}

Now by applying Theorem \ref{thm: main thm body} to the case of the unknot and using the above lemma we have the following corollary. Note that we restate this as we can now explicitly state which gradings are non-zero.

\begin{cor}
    Let $X_{-n}(U)$ be the $-n$ trace of the unknot, for $n>0$. Then for all $\alpha=2s+1$ with $s\geq 0$,
    \begin{equation*}
        \mathcal{FL}_{n(s^2+s)+1}(X_{-n}(U);\alpha,0)\neq \{0\}.
    \end{equation*}
\end{cor}

\begin{proof}
    By Lemma \ref{lem: FUn is an isomorphism} we have that $F_{U_{-n}}$ is an isomorphism for all $n>0$. Hence, by Theorem \ref{thm: main thm body} the result follows.
\end{proof}

\bibliographystyle{alpha}
\bibliography{bibliography}

@misc{floerlasagna,
      title={Floer lasagna modules from link Floer homology}, 
      author={Daren Chen},
      year={2022},
      eprint={2203.07650},
      archivePrefix={arXiv},
      primaryClass={math.GT},
      url={https://arxiv.org/abs/2203.07650},
      note={arxiv:2203.07650},
}

@article{OGSkLas,
    author = {Morrison, Scott and Walker, Kevin and Wedrich, Paul},
    issn = {1465-3060},
    journal = {Geometry \& topology},
    language = {eng},
    number = {8},
    pages = {3367-3420},
    title = {Invariants of 4–manifolds from Khovanov–Rozansky link homology},
    volume = {26},
    year = {2022},
}

@misc{exoticsklas,
      title={Khovanov homology and exotic $4$-manifolds}, 
      author={Qiuyu Ren and Michael Willis},
      year={2024},
      eprint={2402.10452},
      archivePrefix={arXiv},
      primaryClass={math.GT},
      url={https://arxiv.org/abs/2402.10452}, 
      note={arXiv:2402.10452}
}

@book{SpecSeqBook,
author = {McCleary, John},
address = {Cambridge},
booktitle = {A user's guide to spectral sequences},
edition = {Second},
isbn = {0-511-82182-4},
language = {eng},
publisher = {Cambridge University Press},
series = {Cambridge studies in advanced mathematics ; 58},
title = {A user's guide to spectral sequences },
year = {2001},
}

@article{KhovanovFloer,
author = {Baldwin, John A. and Hedden, Matthew and Lobb, Andrew},
copyright = {2019},
issn = {0001-8708},
journal = {Advances in mathematics (New York. 1965)},
language = {eng},
pages = {1162-1205},
publisher = {Elsevier Inc},
title = {On the functoriality of Khovanov–Floer theories},
volume = {345},
year = {2019},
}

@article{OSLinkFloer,
author = {Ozsváth, Peter and Szabó, Zoltán},
issn = {1472-2747},
journal = {Algebraic \& geometric topology},
language = {eng},
number = {2},
pages = {615-692},
title = {Holomorphic disks, link invariants and the multi-variable Alexander polynomial},
volume = {8},
year = {2008},
}

@article{OSHFOG,
author = {Ozsváth, Peter and Szabó, Zoltán},
address = {Princeton, NJ},
copyright = {Copyright 2004 Princeton University (Mathematics Department)},
issn = {0003-486X},
journal = {Annals of mathematics},
language = {eng},
number = {3},
pages = {1027-1158},
publisher = {Princeton University Press},
title = {Holomorphic Disks and Topological Invariants for Closed Three-Manifolds},
volume = {159},
year = {2004},
}

@misc{nntorus,
      title={Heegaard Floer homology of (n,n)-torus links: computations and questions}, 
      author={Joan E. Licata},
      year={2012},
      eprint={1208.0394},
      archivePrefix={arXiv},
      primaryClass={math.GT},
      url={https://arxiv.org/abs/1208.0394}, 
      note={arXiv:1208.0394}
}

@article{linkLspacernrm,
author = {Gorsky, Eugene and Hom, Jennifer},
copyright = {public},
journal = {Quantum Topology},
language = {eng},
number = {4},
pages = {629-666},
publisher = {eScholarship, University of California},
title = {Cable links and L-space surgeries},
volume = {8},
year = {2017},
}

@article{LemmasPaper,
author = {Baldwin, John A. and Grigsby, J. Elisenda},
copyright = {Copyright 2015, American Mathematical Society},
issn = {0002-9939},
journal = {Proceedings of the American Mathematical Society},
language = {eng},
number = {7},
pages = {2801-2814},
publisher = {American Mathematical Society},
title = {Categorified invariants and the braid group},
volume = {143},
year = {2015},
}

@article{CobMapPaper,
author = {Zemke, Ian},
copyright = {2018 London Mathematical Society},
issn = {1753-8416},
journal = {Journal of topology},
language = {eng},
number = {1},
pages = {94-220},
title = {Link cobordisms and functoriality in link Floer homology},
volume = {12},
year = {2019},
}

@article{CobGradings,
    author = {Ian Zemke},
    title = {Link cobordisms and absolute gradings on link Floer homology},
    journal = {Quantum Topology},
    year = {2019},
    language = {eng},
    volume={10},
    number = {2},
    pages={207-323}
}

@article{QuasiMap,
author = {Zemke, Ian},
issn = {1472-2747},
journal = {Algebraic \& geometric topology},
language = {eng},
number = {6},
pages = {3461-3518},
title = {Quasistabilization and basepoint moving maps in link Floer homology},
volume = {17},
year = {2017},
}

@article{Naturality,
author = {Juh\'asz, Andr\'as and Thurston, Dylan P. and Zemke, Ian},
copyright = {Open access},
language = {eng},
publisher = {American Mathematical Society},
journal = {Memoirs of the American Mathematical Society},
title = {Naturality and mapping class groups in Heegaard Floer homology},
year = {2021},
volume= {273},
number={1338},
}

@article{KhovanovOG,
author = {Khovanov, Mikhail},
copyright = {Copyright 2000 Duke University Press},
issn = {0012-7094},
journal = {Duke mathematical journal},
language = {eng},
number = {3},
pages = {359-426},
publisher = {DUKE University Press},
title = {A categorification of the Jones polynomial},
volume = {101},
year = {2000},
}

@article{KhovanovRozOG,
author = {Khovanov, Mikhail and Rozansky, Lev},
issn = {0016-2736},
journal = {Fundamenta mathematicae},
language = {eng},
number = {1},
pages = {1-91},
title = {Matrix factorizations and link homology},
volume = {199},
year = {2008},
}

@article{SkeinLasagna2Handle,
author = {Manolescu, Ciprian and Neithalath, Ikshu},
issn = {0075-4102},
journal = {Journal für die reine und angewandte Mathematik},
language = {eng},
number = {788},
pages = {37-76},
publisher = {De Gruyter},
title = {Skein lasagna modules for 2-handlebodies},
volume = {2022},
year = {2022},
}

@article{SkeinLasangaHandleDecomp,
author = {Manolescu, Ciprian and Walker, Kevin and Wedrich, Paul},
copyright = {2023},
issn = {0001-8708},
journal = {Advances in mathematics (New York. 1965)},
language = {eng},
pages = {109071},
publisher = {Elsevier Inc},
title = {Skein lasagna modules and handle decompositions},
volume = {425},
year = {2023},
}

@article{AndrasCobordisms,
author = {Juhász, András},
copyright = {2016 Elsevier Inc.},
issn = {0001-8708},
journal = {Advances in mathematics (New York. 1965)},
language = {eng},
pages = {940-1038},
publisher = {Elsevier Inc},
title = {Cobordisms of sutured manifolds and the functoriality of link Floer homology},
volume = {299},
year = {2016},
}

@article{MapsAreSamePaper,
author = {Juh\'asz, Andr\'as and Zemke, Ian},
address = {COVENTRY},
issn = {1465-3060},
journal = {Geometry \& topology},
language = {eng},
number = {1},
pages = {179-307},
publisher = {Geometry & Topology Publications},
title = {Contact handles, duality, and sutured Floer homology},
volume = {24},
year = {2020},
}

@article{LSpaceAlexMonic,
author = {Ozsváth, Peter and Szabó, Zoltán},
copyright = {2005 Elsevier Ltd},
issn = {0040-9383},
journal = {Topology (Oxford)},
language = {eng},
number = {6},
pages = {1281-1300},
publisher = {Elsevier Ltd},
title = {On knot Floer homology and lens space surgeries},
volume = {44},
year = {2005},
}

@article{OGSuturedPaper,
author = {Juhász, András},
issn = {1472-2747},
journal = {Algebraic \& geometric topology},
language = {eng},
number = {3},
pages = {1429-1457},
title = {Holomorphic discs and sutured manifolds},
volume = {6},
year = {2006},
}

@book{ComplexAnalysisRef,
isbn = {0387944605},
language = {eng},
lccn = {95002331},
publisher = {Springer-Verlag},
series = {Graduate texts in mathematics ; 159},
title = {Functions of one complex variable II },
year = {1995},
author = {Conway, John B.},
address = {New York},
booktitle = {Functions of one complex variable II},
}

@article{HFKdetectstrefoil,
copyright = {Copyright 2008 The Johns Hopkins University Press},
issn = {0002-9327},
journal = {American journal of mathematics},
language = {eng},
number = {5},
pages = {1151-1169},
publisher = {Johns Hopkins University Press},
title = {Knot Floer Homology Detects Genus-One Fibred Knots},
volume = {130},
year = {2008},
author = {Ghiggini, Paolo},
address = {Baltimore, MD},
}

@misc{NoSkeinLasExotica,
      title={The Rasmussen s-invariant and exotic 4-manifolds}, 
      author={Gheehyun Nahm},
      year={2026},
      eprint={2602.20138},
      archivePrefix={arXiv},
      primaryClass={math.GT},
      note ={arxiv:2602.20138}, 
}

@article{MaslovIndexPaper,
author = {Sarkar, Sucharit},
copyright = {Copyright 2011 International Press of Boston},
issn = {1527-5256},
journal = {Journal of symplectic geometry},
language = {eng},
number = {2},
pages = {251-270},
publisher = {International Press of Boston},
title = {Maslov index formulas for Whitney $n$-gons},
volume = {9},
year = {2011},
}

@article{LinkSurgeryFormulaOG,
author = {Manolescu, Ciprian and Ozsváth, Peter},
issn = {1465-3060},
journal = {Geometry \& topology},
language = {eng},
number = {6},
pages = {2783-3062},
title = {Heegaard Floer homology and integer surgeries on links},
volume = {29},
year = {2025},
}

@article{GraphCobZemke,
author = {Zemke, Ian},
issn = {1465-3060},
journal = {Geometry \& topology},
language = {eng},
number = {2},
pages = {389-528},
title = {Graph cobordisms and Heegaard Floer homology},
volume = {30},
year = {2026},
}

@article{CylindricalReformulation,
author = {Lipshitz, Robert},
issn = {1465-3060},
journal = {Geometry \& topology},
language = {eng},
number = {2},
pages = {955-1096},
title = {A cylindrical reformulation of Heegaard Floer homology},
volume = {10},
year = {2006},
}

\end{document}